\documentclass[11pt,letterpaper]{amsart}
\usepackage[utf8]{inputenc}
\usepackage{mathtools,amssymb,amsthm,tikz-cd,multirow}
\usepackage[marginratio=1:1]{geometry} 
\usepackage{soul,xcolor,enumerate}
\usepackage{tikz}
\usetikzlibrary{arrows}

\newtheorem*{thm*}{Theorem}
\newtheorem{thm}{Theorem}
\newtheorem{propn}[thm]{Proposition}
\newtheorem{lem}[thm]{Lemma}
\newtheorem{cor}[thm]{Corollary}

\theoremstyle{definition}
\newtheorem{defn}[thm]{Definition}
\newtheorem{eg}[thm]{Example}
\newtheorem{q}[thm]{Question}

\theoremstyle{remark}
\newtheorem{rmk}[thm]{Remark}
\newtheorem*{rmk*}{Remark}

\numberwithin{thm}{section}

\allowdisplaybreaks

\let\del\partial 
\let\onto\twoheadrightarrow
\let\into\hookrightarrow
\newcommand{\simto}{\overset{\sim}{\to}}
\renewcommand{\P}{\mathbf{P}}
\newcommand{\A}{\mathbf{A}}
\newcommand{\R}{\mathbf{R}}
\newcommand{\C}{\mathbf{C}}
\newcommand{\N}{\mathbf{N}}
\newcommand{\F}{\mathbf{F}}
\newcommand{\Z}{\mathbf{Z}}
\newcommand{\Q}{\mathbf{Q}}
\renewcommand{\d}{\mathrm{d}}
\newcommand{\Num}{\mathcal{N}}
\newcommand{\prim}{\mathrm{prim}}
\newcommand{\id}{\mathrm{id}}
\DeclareMathOperator{\Nm}{Nm}
\DeclareMathOperator{\rad}{rad}
\DeclareMathOperator{\Cl}{Cl}
\DeclareMathOperator{\Ann}{Ann}
\DeclareMathOperator{\Aut}{Aut}
\DeclareMathOperator{\Stab}{Stab}
\DeclareMathOperator{\lcm}{lcm}
\DeclareMathOperator{\pr}{pr}
\DeclareMathOperator{\im}{im}

\DeclareMathOperator{\Res}{Res}
\DeclareMathOperator{\vol}{vol}
\DeclareMathOperator{\area}{area}
\DeclareMathOperator{\Lip}{Lip}
\DeclareMathOperator{\ord}{ord}
\DeclareMathOperator{\Span}{span}
\DeclareMathOperator{\PGL}{PGL}
\DeclareMathOperator{\rk}{rk}
\DeclareMathOperator{\Fix}{Fix}
\DeclareMathOperator{\Irr}{Irr}
\DeclareMathOperator{\dom}{dom}
\let\Bar\overline \let\Tilde\widetilde 
\let\bdy\partial
\let\grad\nabla
\let\mf\mathfrak 
\let\le\leqslant \let\ge\geqslant 
\let\acts\curvearrowright 
\let\Hat\widehat

\title{Heights and morphisms in number fields}
\author{Matt Olechnowicz}
\date{\today} 
\address{Department of Mathematics \& Statistics, Concordia University}
\email{matt.olechnowicz@concordia.ca}

\begin{document}
	
\begin{abstract}
We give a formula with explicit error term 
for the number of $K$-rational points $P$ satisfying $H(f(P)) \le X$ as $X \to \infty$, where $f$ is a nonconstant morphism between projective spaces 
defined over a number field $K$ and $H$ is the absolute multiplicative Weil height. 
This yields formulae for the counting functions of $f(\P^m(K))$ with respect to the Weil height as well as of $\P^m(K)$ with respect to the Call--Silverman canonical height.
\end{abstract}

\maketitle

\section{Introduction}

\subsection*{Motivation}
Let $K$ be a number field of degree $n$ and ring of integers $\mathcal{O}_K$.
The \emph{height} of a nonzero point $x$ in $K^{m+1}$ is the quantity
\[
H(x) = \Big(\frac{1}{\Nm {\langle x \rangle}} \prod_{\sigma : K\into \C} |\sigma(x)| \Big)^{1/n}
\]
where $|\cdot|$ is the max norm, 
$\sigma(x) := (\sigma(x_0), \ldots, \sigma(x_m))$, 
and 
$\langle x \rangle := x_0 \mathcal{O}_K + \ldots + x_m \mathcal{O}_K$ is the fractional ideal generated by $x$.
If $\lambda$ is a nonzero scalar,
then
\[\Nm \lambda \mathcal{O}_K = \prod_{\sigma : K \into \C} |\sigma(\lambda)|;\]
thus, the height descends to $\P^m(K) := (K^{m+1} - 0^{m+1}) / K^\times$,
the projective $m$-space over $K$.
For $m = 1$ and $K = \Q$,
these definitions simplify to 
\begin{equation} \label{intro:HP1Q}
H(a : b) = \frac{|(a, b)|}{\gcd(a, b)}
\end{equation}
assuming $a, b \in \Z$.

In \cite[Theorem 1]{Northcott}, Northcott proved that for any real number $X$, the set 
\begin{equation} \label{intro:set}
	\{P \in \P^m(K) : H(P) \le X\}
\end{equation}
is finite (though he used the 1-norm at the infinite places).
Lang \cite[p.~58]{Lang} inquired about an estimate for the cardinality $\Num_{H, \P^m(K)}(X)$ of \eqref{intro:set} when $m = 1$ and $X \to \infty$, 
noting that the case $n = 1$ is a classical result in elementary number theory \cite[18.5]{HW}:
\[
\Num_{H, \P^1(\Q)}(X)
= \frac{12}{\pi^2} X^2 + O(X \log X).
\]
Soon after Lang's book appeared, Schanuel \cite[Theorem 1]{Schanuel63} announced a formula involving all the classical invariants of $K$.

\begin{thm*}[Schanuel]
Suppose $nm > 1$.
Then as $X \to \infty$, 
\[
\Num_{H,\P^m(K)}(X) 
= 
\frac{hR}{w} \frac{(m+1)^r}{\zeta_K(m+1)} \Big(\frac{2^{r_1} (2\pi)^{r_2}}{\sqrt {d_K}} \Big)^{m+1} 
X^{n(m+1)} + O(X^{n(m+1)-1})
\]
where 
$h$ is the class number,
$R$ is the regulator, 
$w$ is the number of roots of unity, 
$r_1$ (resp.~$r_2$) is the number of real (resp.~complex) embeddings, 
$r = r_1 + r_2 - 1$ is the rank of the unit group, 
$d_K$ is the absolute value of the discriminant,
and $\zeta_K$ is the Dedekind zeta function.
\end{thm*}

The proof, published a decade and a half later \cite{Schanuel}, proceeds by 
fibering over ideal classes, 
replacing $K^\times$ by $\mathcal{O}_K^\times$ (and later by $\mu(K)$), 
performing a M\"obius inversion,
embedding everything in Minkowski space, 
counting lattice points in a homogeneously expanding domain,
and summing the resulting Dirichlet series.

Schanuel's theorem is the first in a long line of results counting rational points 
on increasingly general varieties 
by increasingly general heights.
We cannot do justice to this immense body of literature here,
but the general philosophy (due to Batyrev, Manin, Tschinkel, Peyre, \ldots\!)~is that if $V$ is a smooth projective variety with a Zariski-dense set of $K$-rational points 
and if $\mathcal{L}$ is a metrized ample line bundle on $V$,
then for any sufficiently small Zariski-open subset $U$ of $V$,
\[\Num_{H_{\mathcal L}, U(K)}(X) \sim c X^\alpha (\log X)^\beta\]
for some constants $\alpha$, $\beta$, $c$ with definite geometric interpretations (in some cases conjectural).
The lecture notes by Chambert-Loir \cite{ACL} offer an excellent introduction.

In counting problems where the objects in question lie in algebraic families, one commonly wants to know the number of points in the image of a morphism.
For instance, 
Harron and Snowden \cite{HS} estimated the number of elliptic curves with prescribed torsion over $\Q$
using Kubert's parametrizations.
Analogously, in \cite{Olechnowicz} the author
used the classification theorems of Poonen and Canci--Vishkautsan to give formulae for the number of quadratic rational maps admitting certain preperiodic point portraits over $\Q$; and recently Siu gave many more similar formulae \cite{Siu}.
Thus it would be handy to have a general tool for such scenarios.

\begin{q} \label{q:img}
Let $f : \P^m \to \P^M$ be a nonconstant morphism
defined over a number field $K$.
What is the behaviour of 
$\Num_{H, f(\P^m(K))}(X)$
as $X \to \infty$?
\end{q}

\noindent 
We answer Question \ref{q:img} with an explicit main term and a power-saving error term: see Theorem \ref{thm:img}.

In arithmetic dynamics, 
one studies the iterates of an endomorphism $f : \P^m \to \P^m$ of degree $d \ge 2$ 
by way of the so-called \emph{canonical height}:
\[\Hat{H}_f(P) := \lim_{i \to \infty} H(f^i(P))^{1/d^i}\]
---though usually in the logarithmic form $\hat{h}_f := \log \Hat{H}_f$.
The canonical height, 
which was introduced by Call and Silverman \cite[Theorem 1.1]{CS}, 
is useful for two reasons:
it interacts well with $f$, 
and it differs from the usual height by a bounded scale factor.
More precisely, 
\[
\Hat{H}_f \circ f = (\Hat{H}_f)^d
\quad\text{and}\quad 
\Hat{H}_f \asymp H.
\]
In particular, $\Hat{H}_f(P) = 1$ if and only if the sequence $P, f(P), f^2(P), \ldots$ is eventually periodic, 
and the set of such points is therefore finite.

The value distribution of $\Hat{H}_f$ is of considerable interest.
To cherry-pick one example, it is believed that the canonical height satisfies an analogue of Lehmer's conjecture, namely 
that if $\Hat{H}_f(P) \ne 1$ then $\Hat{H}_f(P) \ge C^{1 /{ \deg P}}$ for some constant $C$ depending only on $f$.
We refer the reader to the survey \cite[Sections 15 and 16]{CurrentTrends} for the state of the art.
Recent work on arithmetic dynamics of random polynomials
\cite[Lemma 2.11]{LBM} 
required upper bounds on $\Num_{\Hat{H}_f, \P^1(\Q)}(X)$
where $f(z) = a_d z^2 + a_{d-2} z^{d-2} + \ldots + a_1 z + 1 \in \Z[z]$ ($d \ge 3$).
One is thus led to wonder about an analogue of Schanuel's theorem for the canonical height.

\begin{q} \label{q:ds}
Let $f$ be an endomorphism of $\P^m$ of degree $d \ge 2$ defined over a number field $K$ of degree $n \ge 1$.
What is the behaviour of 
$\Num_{\Hat{H}_f, \P^m(K)}(X)$
as $X \to \infty$?
\end{q}

\noindent
We answer Question \ref{q:ds} with an explicit main term only: see Theorem \ref{thm:ds}.

As a first approximation,
both questions 
may be answered using Schanuel's original theorem and standard facts about morphisms.
Immediately from $\Hat{H}_f \asymp H$
we obtain
\[\Num_{\Hat{H}_f, \P^m(K)}(X) \asymp X^{n(m+1)}.\]
Similarly,
given that $H \circ f \asymp H^d$ and that $f$ is generically $\gamma$-to-1 over its image for some positive integer $\gamma$,
\[
\Num_{H, f(\P^m(K))}(X) \asymp \frac{1}{\gamma} X^{n(m+1)/d}.
\]

Pursuing a more refined estimate 
leads one to confront (in \emph{both} cases) the function
\[
\Num_{f^*\!H, \P^m(K)}(X) 
= \#\{P \in \P^m(K) : H(f(P)) \le X\},
\]
which counts points by pullback height.
It can be deduced from Franke \textit{et al.}~\cite{FMT} 
that the main term is 
\begin{equation} \label{eq:FMT}
\Num_{f^*\!H, \P^m(K)}(X) \sim c X^{n(m+1)/d}
\end{equation}
for some positive constant $c$, which was later expressed by Peyre \cite{Peyre} in terms of a global Tamagawa measure.
However, 
as suggested by Masser and Vaaler \cite[p.~430]{MV}, 
it seems unlikely that those general methods could lead to a power-saving error term \textit{\`a la} Schanuel.
Establishing such an error term,
with explicit dependence on $f$, 
is the focus of the present article.\footnote{
While this paper was in its final stages of preparation, 
Siu posted a preprint containing a variant of our main result, namely
that if $f$ is a degree-$d$ endomorphism of $\P^m$ 
and $\mathcal L = (\mathcal{O}(1), \lVert \, {\cdot} \, \rVert)$ is a metrized line bundle on $\P^m$, 
all defined over $K$,
then 
\begin{equation} \label{siu}
	\Num_{H_{f^*\!\mathcal L}, \P^m(K)}(X) 
	= \frac{\tau_{\P^m}(f^*\!\mathcal L)}{m+1} X^{n(m+1)/d} + O(X^{(n(m+1)-1)/d} \log X)
\end{equation}
where $\tau_{\P^m}$ is the global Tamagawa measure
and the $O$ term depends on $f$, $\mathcal{L}$, and $K$
\cite[Theorem 3.29]{Siu}.
Siu called \eqref{siu} a ``known result''
and indicated 
that ``to get an explicit power-saving error term, one can modify Masser--Vaaler's work''.
Having already carried out these modifications ourselves, we wish to record the details, which elucidate the dependence on $f$.	
}

\subsection*{Notation}

Before we state our main result (Theorem \ref{thm:main}),
we set down some notation.

For each positive integer $m$, 
we collect the classical invariants of $K$ into the constant
\begin{equation} \label{cKm}
c_K(m) 
:= 
\frac{hR (m+1)^r}{w \zeta_K(m+1)} 
\Big(\!\frac{2^{r_2}}{\sqrt{d_K}} 
\Big)^{m+1}.
\end{equation}
Beware this differs from Schanuel's constant by a factor of $(2^{r_1} \pi^{r_2})^{m+1}$.

Two nontrivial absolute values on $K$ are called \emph{equivalent} if they define the same topology; 
the equivalence classes are called \emph{places}.
By Ostrowski's theorem, 
every place $v$ of $K$
is represented either by
\begin{align} 
|x|_{\mf p} &:= \Nm \mf p^{-{\ord_{\mf p}(x \mathcal{O}_K)} / n_{\mf p}} \label{abs_prime}
\intertext{where $\mf p$ is a prime ideal of local degree $n_{\mf p} = e_{\mf p} f_{\mf p}$; or else by}
|x|_\sigma &:= \sqrt{\Re \sigma(x)^2 + \Im \sigma(x)^2} \label{abs_sigma}
\end{align}
where $\sigma : K \into \C$ is a field embedding,
with local degree $n_\sigma = 1$ or $2$ according to whether $\sigma(K) \subseteq \R$ or not (i.e.,~whether $\sigma$ is \emph{real} or \emph{complex}).
We say $v$ is \emph{nonarchimedean} ($v \nmid \infty$) in case \eqref{abs_prime} and \emph{archimedean} ($v \mid \infty$) in case \eqref{abs_sigma}.

Recall that every morphism $f : \P^m \to \P^M$ 
of degree $d$
defined over a field $K$
is given by $M+1$ homogeneous forms 
$F_0, \ldots, F_M$ of degree $d$ in $m+1$ variables with coefficients in $K$ and no common zero in $\Bar K$.
Any two such \emph{lifts} of $f$ differ only by a nonzero scalar multiple.
Given any absolute value $|\cdot|$ on $K$, 
we have
\begin{equation} \label{intro:asymp}
|F(z)| \asymp |z|^d
\end{equation}
for all $z \in K^{m+1}$ (where the absolute value of a tuple denotes the maximum of the absolute values of its entries).
The upper bound in \eqref{intro:asymp}
is immediate from the triangle inequality (and holds whenever $f$ is a rational map), 
while the lower bound requires the Nullstellensatz (and holds only when $f$ is a morphism).

Writing 
\[
F_j(X_0, \ldots, X_m) = \sum F_{j,\alpha} X_0^{\alpha_0} \ldots X_m^{\alpha_m} 
\qquad (j = 0, \ldots, M)
\]
and setting $|F| := \max |F_{j,\alpha}|$,
it follows from \eqref{intro:asymp} that 
there exist real numbers $C_v$,
one for each place $v$ of $K$,
such that 
\begin{equation} \label{intro_constants_1}
\frac{|z|_v |F|_v^{1/d}}
{|F(z)|_v^{1/d}} \le C_v \text{ for all nonzero $z$ in $K_v^{m+1}$}.
\end{equation}
Now, it is well known (and follows from Definition \ref{def:good_red} and Corollary \ref{good_reduction_epsilon} below) 
that
if $v$ is nonarchimedean and $f$ has good reduction at $v$,
then
\begin{equation} \label{intro:good_red_Fz} 
|F(z)|_v = |F|_v |z|_v^d
\end{equation} 
for all $z \in K_v^{m+1}$.
Thus we may choose
\begin{equation} \label{intro_constants_2}
C_v = 1 \text{ for all but finitely many } v.
\end{equation}
Given any constants $C_v$ satisfying \eqref{intro_constants_1} and \eqref{intro_constants_2}, 
put 
\begin{equation} \label{error_constants}
C_f^0 := \prod_{v \nmid \infty} C_v^{n_v}
\quad\text{and}\quad 
C_f^\infty := \max_{v \mid \infty} C_v
\end{equation}
and note that the infinite product converges by \eqref{intro_constants_2}.

Finally, 
we adopt the Vinogradov-style notation 
$A \ll_{i, \ldots, j} B$ 
to simply mean $A \le CB$ for some positive constant $C$ depending only on the quantities $i, \ldots, j$.
Note there is no implied asymptotic range of validity.

\subsection*{Main result}

\begin{thm} \label{thm:main}
Let $f : \P^m \to \P^M$ be a morphism of degree $d \ge 1$ 
defined over a number field $K$ of degree $n$.
Let $F$ be a lift of $f$ and let $H(f) = H(F_{j,\alpha})$
be the height of $f$ viewed as a point of $\P^{(M+1)\binom{m+d}{d}-1}$ by listing its coefficients.
For each standard absolute value $| \cdot |_v$ on $K$,
let $K_v$ be the completion,
$O_v^{m+1} := \{z \in K_v^{m+1} : |z|_v \le 1\}$ the unit polydisc, 
$n_v := [K_v : \Q_v]$ the local degree,
and $\mu_v$ the Haar measure with $\mu_v(O_v^{m+1}) = 1$.
Let 
\[
D_{f,v} := \{z \in K_v^{m+1} : |F(z)|_v \le |F|_v\}
\]
be the \emph{local fundamental domain} for $f$ at $v$.
Define
\[
c_{K,v}(f) := \mu_v(D_{f,v})
\]
if $v$ is archimedean, 
and 
\[c_{K,v}(f) := \int_{O_v^{m+1}} \bigg(\frac{|z|_v |F|_v^{1/d}}{|F(z)|_v^{1/d}} \bigg)^{n_v(m+1)} \, d\mu_v(z)\]
if $v$ is nonarchimedean.
Put
\[
c_{K,\infty}(f) := \prod_{v \mid \infty} c_{K,v}(f) 
\quad\text{and}\quad 
c_{K,0}(f) := \prod_{v \nmid \infty} c_{K,v}(f)\]
and set
\[c_K(f) := 
c_K(m) \frac{c_{K,\infty}(f) c_{K,0}(f)}{H(f)^{n(m+1)/d}}.
\]
Then
\[
	\Num_{f^*\! H, \P^m(K)}(X) 
	= c_K(f) X^{n(m+1)/d} + E(X)
\]
where the error term $E(X)$ depends on $m$, $K$, and $f$.
More precisely,
there exist
a positive integer $N_f$
and a real number $L_f$
such that for each $v \mid \infty$,
the set 
\[\bdy D_{f,v} = \{z \in K_v^{m+1} : |F(z)|_v = |F|_v\}\]
is covered by the images of at most $N_f$ maps $[0, 1]^{n_v(m+1) - 1} \to K_v^{m+1}$ with Lipschitz constant at most $L_f$; 
and, given any such pair $(N_f, L_f)$ we have
\begin{align*}
|E(X)| &\ll_{m,K}
 N_f^n c_{K,0}(f) (C_f^0)^{d/n} \Big( \frac{L_f + C_f^\infty}{H(f)^{1/d}} X^{1/d}\Big)^{n(m+1)-1} (1 + {\log^+} X^{1/d})
\\
&\hspace{20em} + \big(N_f^n (C_f^0)^{d(m+1)} + c_K(f)\big) X^{n/d}
\end{align*}
for all $X \ge 0$.
The ${\log^+} X^{1/d}$ and $C_f^\infty$ disappear (i.e.,~may be replaced by 0) if $nm > 1$ and $r = 0$, respectively.
\end{thm}

\subsection*{Remarks}
The quantities appearing in the main term and error term are new and interesting invariants of a morphism $f$. 
These merit further study, especially in the dynamical context ($m = M$).
Several remarks are in order.

\begin{rmk}
There are no nonconstant morphisms $\P^m \to \P^M$ when $m > M$.
\end{rmk}

\begin{rmk} \label{rmk:good_reduction_cKvf}
If $f$ has good reduction at $v$, then $c_{K,v}(f) = 1$
by \eqref{intro:good_red_Fz}.
\end{rmk}

\begin{rmk}
If $x^d = (x_0^d : \ldots : x_m^d)$ is the powering map,
then 
\[
c_{K,v}(x^d) = \begin{cases}
2^{m+1} & 
v \text{ is real}, \\
\pi^{m+1} & 
v \text{ is complex}, \\
1 & 
v \text{ is nonarchimedean}.
\end{cases}
\]
Taking $d = 1$ recovers Schanuel's theorem.
\end{rmk}

\begin{rmk} \label{rmk:aut}
If $f$ is an automorphism of $\P^m$ 
then \emph{a fortiori}
\[
c_K(f) = c_K(\id)
\]
because $f$ induces a bijection 
\[\{P \in \P^m(K) : H(f(P)) \le X\} \overset{\sim}\longrightarrow \{Q \in \P^m(K) : H(Q) \le X\}.\]
\end{rmk}

\begin{rmk}
Our expression for $c_K(f)$ is not a special case of Widmer's formula \cite[Theorem 3.1]{Widmer_prim} because the maps $z \mapsto |F(z)|_v^{1/d}$ need not form an adelic Lipschitz system---they may fail axiom (iv), the ultrametric inequality. 
For instance, if $F(x, y) = (2x^2 + xy, xy + 2y^2)$
then 
\[F(1, 0) = (2, 0) \qquad F(0, 1) = (0, 2) \qquad F(1, 1) = (3, 3)\]
so $\big|F\big((1, 0) + (0, 1)\big)\big|_v^{1/2} = 1$
whereas $\big|F(1, 0)\big|_v^{1/2} = \big|F(0, 1)\big|_v^{1/2} < 1$ whenever $v \mid 2$.
\end{rmk}

\begin{rmk}
For $v \nmid \infty$, 
the integral
evaluates to
\[
c_{K,v}(f) = \sum_{i\ge 0} \Nm \mf p_v^{(m+1)i/d} \delta_{f,v}(i)
\]
where $\delta_{f,v}(i)$ may be interpreted as the probability that $v(F(z)) = dv(z) + v(F) + i$ 
(cf.~Lemma \ref{many_faces_of_delta}
and 
Corollary \ref{nonarch_not_vol}(i)).
In particular, the nonarchimedean local factors are not necessarily the ``obvious'' (i.e.,~analogous) $v$-adic volumes.
In fact,
\begin{equation} \label{nonarch_not_vol_strict}
c_{K,v}(f) > \mu_v (D_{f,v})
\end{equation}
unless $\delta_{f,v}(i) = 0$ for all $i$ not divisible by $d$
(cf.~Corollary \ref{nonarch_not_vol}(iii)).
This is because, loosely speaking, 
the volume of a 
$v$-adic disc $B(a, r) = \{z \in K_v : |z - a| \le r\}$ is not necessarily proportional to its radius $r$ unless $r$ is \emph{rational} (i.e.,~belongs to the value group $|K^\times_v|$).
For a concrete instance of \eqref{nonarch_not_vol_strict}, see Example \ref{eg:nonarch_not_vol}.
\end{rmk}

\begin{rmk}
When chosen optimally, 
the error constants \eqref{error_constants} 
have the following interpretations:
\[
C_f^\infty = \max_{v \mid \infty} \sup_{z \in D_{f,v}} |z|_v
\]
and
\[
C_f^0 = \sqrt[d]{\Nm \underset{z \ne 0}{\lcm} \ \ell_f(z)}
\]
where $\ell_f$ is the excess divisor, defined in \eqref{intro:ell} just below
(cf.~Propositions \ref{propn:kappa} and Proposition \ref{ell_properties}(iv)).
In particular,
\[
(C_f^0)^d
\le \Nm \Res f
\lessapprox 
\big(d^m H(f)\big)^{n(md)^m}
\]
(cf.~Lemmas \ref{lem:eps_trio} and \ref{Nm_res_bound})---but \emph{this} upper bound is very likely far too crude.
\end{rmk}

\subsection*{Sketch}

The proof of Theorem \ref{thm:main} follows the strategy outlined in our earlier work \cite[Remark 8.6]{Olechnowicz} for the case $n = m = M = 1$,
which we paraphrase here.
With $f : \P^1 \to \P^1$ a degree-$d$ morphism defined over $\Q$ and $F$ an integral homogeneous lift of $f$, we have
\[
\Num_{f^*\!H, \P^1(\Q)}(X) 
= 
\frac{1}{2} 
\# \{(a, b) \in \Z^2_\prim : H(F(a, b)) \le X\}
\]
where $\Z^2_\prim = \{(a, b) \in \Z^2 : \gcd(a, b) = 1\}$ is the set of lattice points ``visible'' from the origin,
and the factor of a half 
accounts for the symmetry $(a : b) = (-a : -b)$ 
coming from roots of unity in $\Q$.
By \eqref{intro:HP1Q}, 
\[H(F(a, b)) = \frac{|F(a, b)|}{\gcd F(a, b)}.\]
The idea is simple: 
although both $F$ and $(a, b)$ may be in lowest terms, 
$F(a, b)$ need not; 
there may be a nontrivial ``excess'' divisor $\ell := \gcd F(a, b)$.
For instance, if the reduction of $f$ modulo some prime $p$ has an $\F_p$-rational point of indeterminacy $Q$, then $\gcd F(a, b)$ will be divisible by $p$ for every primitive $(a, b)$ lying over $Q$.

Summing over all possible excess divisors,
we get
\begin{equation} \label{intro:heuristic}
\Num_{f^*\!H, \P^1(\Q)}(X)
= 
\frac{1}{2} 
\sum_\ell
\{(a, b) \in \Z^2_\prim : \gcd F(a, b) = \ell \text{ and } |F(a, b)| \le \ell X\}.
\end{equation}
Now, heuristically, $\gcd(a, b) = 1$ with probability $6/\pi^2$;
and examples suggest that
\[
\delta_f(\ell) := \Pr {[\gcd F(a, b) = \ell \mid \gcd(a, b) = 1]}
\]
exists for each excess divisor $\ell$.
Then each summand in \eqref{intro:heuristic}
is approximately 
\begin{equation} \label{intro:heuristic_summands}
\frac{6}{\pi^2} \cdot \delta_f(\ell) \cdot \#\{(a, b) \in \Z^2 : |F(a, b)| \le \ell X\}.
\end{equation}
The latter term counts the number of lattice points in the region $F^{-1}([-\ell X, \ell X]^2)$, 
which by geometry of numbers is approximately the region's area.
Since $F^{-1}$ is homogeneous of degree $1/d$, 
and since ``area'' is homogeneous of degree 2, 
we have 
\begin{equation} \label{intro:heuristic_area}
\Z^2 \cap F^{-1}([-\ell X, \ell X]^2)
\sim c_{\Q,\infty}(f) \cdot (\ell X)^{2/d}
\end{equation}
where $c_{\Q,\infty}(f) = \area F^{-1}([-1, 1]^2)$.
Combining \eqref{intro:heuristic}, 
\eqref{intro:heuristic_summands}, 
and \eqref{intro:heuristic_area}
yields
\[
\Num_{f^*\!H,\P^1(\Q)}(X) = 
\frac{1}{2} 
\cdot \frac{6}{\pi^2} 
\cdot c_{\Q,\infty}(f)
\cdot \sum_\ell \ell^{2/d} \delta_f(\ell)
\cdot X^{2/d}
\]
as $X \to \infty$.
From here it's just a matter of factoring the sum over $\ell$ using the Chinese Remainder Theorem,
which identifies it as $c_{\Q,0}(f)$.
This explains the main term in Theorem \ref{thm:main}.

To turn this heuristic sketch over $\P^1(\Q)$ into a rigorous proof over $\P^m(K)$, 
we closely follow Schanuel's original argument,
with strategic modifications to accommodate the presence of $f$.
Although this could be seen as little more than a tedious exercise, 
we believe that our exposition clarifies Schanuel's original intent (cf.~\cite[p.~430]{MV}).

The main obstacle is the \emph{a priori} irregularity of the nonarchimedean contribution. 
For an arbitrary morphism $f$ with lift $F$,
it is hard to predict the behaviour of the excess divisor 
\begin{equation} \label{intro:ell}
\ell_f(x) := \frac{ \langle F(x) \rangle }{\langle x \rangle^d \langle F \rangle }
\end{equation}
---which is now a fractional ideal---as $x$ varies over $K^{m+1}$.
But it turns out that $\ell_f$ is always one of just finitely many integral ideals $\mf l$, each arising with a well-defined nonzero probability $\delta_f(\mf l)$.
To prove this, 
we give a new definition of the resultant ideal $\Res f$,
one that is uniform in $M \ge m \ge 1$ and compatible with the Sylvester--Macaulay resultant when $M = m$.
We show that 
$\ell_f$ divides $\Res f$
and, moreover, that $\ell_f$ is periodic,
in the sense that it factors through the canonical reduction map 
\begin{equation} \label{intro:reduction_map}
K^{m+1} - 0^{m+1} \to \P^m(\mathcal{O}_K / \Res f).
\end{equation}
These regularity properties---boundedness and periodicity---imply that $\ell_f$ is determined by its values on a finite set.

Fibering over all excess divisors and using the Generalized Chinese Remainder Theorem 
lets us reduce the problem to 
counting points in \emph{cosets} of a lattice.
The subsequent appeal to geometry of numbers is streamlined by the wonderful technical results of Masser--Vaaler and Widmer, including the volume formula \cite[Lemma 4]{MV},
the counting principle \cite[Theorem 5.4]{Widmer_prim}, 
and the Lipschitz estimate \cite[Lemma 7.1]{Widmer_prim}.
Although we ended up not quoting anything from it,
we also benefitted greatly from consulting Krumm's article \cite{Krumm} on computational matters surrounding Schanuel's theorem.

\subsection*{Plan}

The paper is laid out as follows.

Section \ref{sec:applications} 
presents two general consequences of Theorem \ref{thm:main}, 
answering Questions \ref{q:img} and \ref{q:ds}. 
These applications use Hilbert's irreducibility theorem and dynamical Green's functions, respectively.
The latter is accompanied by some intriguing examples, showing that the asymptotic constant for the canonical height differs from Schanuel's constant and is invariant under $K$-conjugacy but not $\Bar K$-conjugacy.

Section \ref{sec:morphisms} 
reviews classical elimination theory as it pertains to zero-loci of morphisms. 
We introduce the notion of a pseudoinverse, define the resultant, and give an algebraic characterization of good reduction.

Section \ref{sec:primitivity}
begins with a short discussion on group actions (reviving Schanuel's notion of a relative fundamental domain).
We define primitivity and generalize Jordan's totient function to any commutative unital ring, and use these to describe projective $m$-space over any finite quotient of a Dedekind domain.
This section culminates in the definition of the reduction map \eqref{intro:reduction_map}.

Section \ref{sec:excess}
establishes a nonarchimedean continuity estimate
and thereby derives the main properties of the excess divisor function,
including boundedness and periodicity.
Here we also introduce the global density $\delta_f(\mf l)$ 
and prove that it decomposes as a product of local densities $\delta_{f,v}(v(\mf l))$.

The next two Sections focus on various nonarchimedean terms that will arise in the proof of Theorem \ref{thm:main}.
Section \ref{sec:nonarch} evaluates $c_{K,v}(f)$ and compares it to the Haar measure of $D_{f,v}$.
Section \ref{sec:nt} expresses a certain $\mathcal{O}_K^\times$-invariant set as a union of cosets of a lattice, and evaluates two pertinent Dirichlet series.

Section \ref{sec:aux} recalls classical results on lattices and number fields, then adapts Schanuel's framework to our setting. We introduce the homogeneously-expanding domain $D_{F,K}(T)$, show that it is bounded, and calculate its volume.

Section \ref{sec:geometry} 
prepares all the ingredients for the application of Masser--Vaaler and Widmer's geometry of numbers results.
We prove that $D_{F,K}(T)$ can be chosen to depend minimally on $K$, and we bound the Lipschitz class of its boundary.

Section \ref{sec:final}, the last one, is devoted to the proof of Theorem \ref{thm:main}.

In the words of Macaulay \cite[p.~vi]{Macaulay_book}, ``The subject is full of pitfalls.''
To help the reader avoid stumbling, 
we have included many remarks and examples along the way.
We hope the length of the article is justified by 
its careful attention these details.

\subsection*{Acknowledgements}

I wish to thank 
Patrick Ingram for suggesting Lemma \ref{X_f} and Example \ref{eg:chebyshev};
Carlo Pagano for pointing me toward Proposition \ref{delta_product_formula};
Andrew Granville for his advice;
and David McKinnon for his encouragement.
I am indebted to 
Jonathan Love for teaching me about lattices 
and proving Lemma \ref{lem:jon} and Proposition \ref{shapely_domains_exist};
to Jason Bell and John Doyle for their assistance with Lemma \ref{jasonjohn};
and to Xander Faber for answering my questions about projective space over arbitrary rings (cf.~Remark \ref{rmk:proj}).

\section{Applications} \label{sec:applications} 

To motivate Theorem \ref{thm:main}, we present two general applications.

\subsection{Counting points in the image of a morphism}

\begin{thm} \label{thm:img}
	Let $f : \P^m \to \P^M$ be a morphism of degree $d \ge 1$ defined over a number field $K$.
Let $\gamma$ be the number of $\varphi$ in $\PGL_{m+1}(K)$ such that $f \circ \varphi = f$.
	Then
	\[\Num_{H, f(\P^m(K))}(X) = \frac{c_K(f)}{\gamma} X^{n(m+1)/d} + 
	O
\begin{dcases}
X^{1/d} \log X & m = n = 1 \\
X^{(2n-1)/d} & m = 1 < n \\
X^{n(m+1/2)/d} \log X & m > 1, \, n \le 2 \\
X^{(n(m+1) - 1)/d} & m > 1, \, n > 2
\end{dcases}
	\]
where the implicit constants depend on $f$ and $K$.
\end{thm}

The proof of Theorem \ref{thm:img} requires a few preliminaries.

\begin{lem} \label{jasonjohn}
Let $\varphi : \P^k \dashrightarrow \P^m$ be a rational map
and 
let $f : \P^m \to \P^l$ be a nonconstant morphism.
Then $\varphi$ and $f \circ \varphi$ have the same domain.
\end{lem}

\begin{proof}[Proof]
Let $F$ and $\Phi$ be homogeneous lifts of $f$ and $\varphi$, respectively.
Assume $\gcd(\Phi) = 1$.
Since $f$ is a morphism, the Nullstellensatz implies there exist polynomials $G_{ij}$ and a positive integer $e$ such that 
\begin{equation} \label{jason_eq1}
	X_i^e = \sum_{j=0}^l G_{ij}(X) F_j(X)
\end{equation} 
for all $0 \le i \le m$.
Replacing $X$ by $\Phi(X)$ in \eqref{jason_eq1} shows that 
\[
	\Phi_i(X)^e = \sum_{j=0}^l G_{ij}(\Phi(X)) F_j(\Phi(X))
\]
whence $\gcd(F \circ \Phi) = 1$ as well.
Thus by \cite[Example A.1.2.6(d), p.~20]{HindrySilverman}, 
the domains of $\varphi$ and $f \circ \varphi$ are precisely
\begin{equation} \label{hs_1}
\dom {(\varphi)} = \P^m \setminus V(\Phi)
\quad \text{ and } \quad 
\dom {(f \circ \varphi)} = \P^m \setminus V(F \circ \Phi).
\end{equation}
But $V(\Phi) \subseteq V(F \circ \Phi)$ because $f$ is nonconstant, 
and $V(F \circ \Phi) \subseteq V(\Phi)$ because $f$ is a morphism.
Thus they are equal, so by \eqref{hs_1} the domains of $\varphi$ and $f \circ \varphi$ coincide.
\end{proof}

Phrased another way,
Lemma \ref{jasonjohn} says that $f \circ \varphi$ and $\varphi$ have the same indeterminacy locus. 
In particular, if $f \circ \varphi$ is a morphism, then $\varphi$ must be a morphism.

\begin{rmk}
In general, if $\psi$, $\varphi$ are rational maps whose composition $\psi \circ \varphi$ is defined, 
then their degrees and indeterminacy loci are related by
\[
\deg \psi \circ \varphi \le \deg \psi \cdot \deg \varphi
\quad\text{and}\quad 
\mathcal{I}_{\psi \, \circ \, \varphi} \subseteq 
\mathcal{I}_\varphi \cup \varphi^{-1}(\mathcal{I}_\psi).
\]
Equality need not hold in either case.
For instance, if 
\begin{align*}
	\varphi : \P^2 &\dashrightarrow \P^2 \\
	\varphi(x : y : z) &= (xy : yz : xz) 
\end{align*}
then $\deg \varphi = 2$
and $\mathcal{I}_\varphi = \{(1 : 0 : 0), (0 : 1 : 0), (0 : 0 : 1)\}$;
yet
\begin{align*}
\varphi(\varphi(x : y : z)) 
&= ((xy)(yz) : (yz)(xz) : (xy)(xz)) \\
&= (x y^2 z : xyz^2 : x^2 yz) \\
&= (y : z : x)
\end{align*}
so that $\deg \varphi \circ \varphi = 1$ and $\mathcal{I}_{\varphi \, \circ \, \varphi} = \varnothing$.
\end{rmk}

\begin{defn}
	Let $f : \P^m \to \P^M$ be a morphism defined over an algebraically closed field.
	A \emph{mapping symmetry} of $f$ is a rational map $\varphi : \P^m \dashrightarrow \P^m$ such that $f \circ \varphi = f$.
	The collection of all mapping symmetries of $f$ will be denoted $\Gamma_f$.
\end{defn}

\begin{propn} \label{symmetry_group}
Suppose $f$ is nonconstant.
Then:
\begin{enumerate}[(i)]
	\item $\Gamma_f$ is a finite subgroup of $\Aut(\P^m) \cong \PGL_{m+1}$ of order at most $(\deg f)^M$.
	\item The critical locus of $f$ is fixed (as a variety) by every element of $\Gamma_f$.
\end{enumerate}
\end{propn}
\begin{rmk*}
We shall not need part (ii) in the present article.
\end{rmk*}
\begin{proof} \hfill 
\begin{enumerate}[(i)]
\item 
Suppose $f \circ \varphi = f$.
By Lemma \ref{jasonjohn}, 
$\varphi$ is a morphism. 
Comparing degrees yields $\deg f \cdot \deg \varphi = \deg f$.
But $\deg f \ne 0$, so $\deg \varphi = 1$.
Since every linear endomorphism of $\P^m$ is invertible,
it follows that $\Gamma_f \leq \Aut(\P^m)$. 
For finiteness,
simply note that $\varphi(\xi) \in f^{-1}(f(\xi))$ 
where $\xi$ is generic. 
The cardinality estimate 
follows from B\'ezout's theorem
\cite[Example 12.3.1, p.~223]{Fulton}.
\item Let $P \in \P^m$ be a critical point of $f$, 
meaning $\rk f'(P) < m$ where $f'(P)$ is the differential of $f$ at $P$.
If $f = f \circ \varphi$ 
then 
\[f'(P) = (f \circ \varphi)'(P) = f'(\varphi(P)) \circ \varphi'(P)\]
by the chain rule.
Part (i) implies $\varphi'(P)$ is invertible, 
so $\rk {(f \circ \varphi)}'(P) = \rk f'(P) < m$.
Thus $P$ is a critical point of $f \circ \varphi$.
\qedhere 
\end{enumerate}
\end{proof}

The geometric significance of $\Gamma_f$ is encapsulated by the following result.

\begin{lem} \label{X_f}
	Let 
	\[X_f := \P^m \times_f \P^m = \{(P, Q) \in \P^m \times \P^m : f(P) = f(Q)\}\]
	be the fibre square of $f$,
	and let 
	\begin{align*}
		\pi : X_f &\to \P^m
		\\
		(P, Q) &\mapsto P
	\end{align*}
	be the first projection.
	Then: 
	\begin{enumerate}[(i)] 
		\item $\pi^{-1}(P) \cong f^{-1}(f(P))$ for all $P$ in $\P^m$.
		\item Every rational section of $\pi$ has the form $\id \times \varphi$ for some $\varphi \in \Gamma_f$ (and conversely).
		\item Let $\Irr(X_f)$ denote the set of irreducible components of $X_f$.
		If $f$ is nonconstant, then the assignment
		\[
		\Gamma_f \to \operatorname{Irr}(X_f), \quad \varphi \mapsto (\id \times \varphi)(\P^m) \ \text{(= the graph of $\varphi$)}
		\]
		is well-defined and injective.
	\end{enumerate}
\end{lem}
\begin{proof}
Parts (i) and (ii) are trivial.
Part (iii) follows from Lemma \ref{jasonjohn} and standard facts about irreducible components.
To wit, since each mapping symmetry $\varphi$ is continuous and closed,
the image of $\id \times \varphi$ is a closed irreducible subset of $X_f$, 
so it is contained in some irreducible component $Z$ of $X_f$.
Clearly $\dim {(\id \times \varphi)(\P^m)} = m$,
whereas by part (i) and the fibre--dimension theorem, $\dim Z \le m$.
It follows that the dimensions are equal,
whence $(\id \times \varphi)(\P^m) = Z$.
Injectivity is obvious: functions are determined by their graphs.
\end{proof}

\begin{proof}[Proof of Theorem \ref{thm:img}]
By Proposition \ref{symmetry_group}(i), 
$\gamma = \#\Gamma_f(K)$ is finite and positive.
Since $f$ is nonconstant,
$f^{-1}(f(P))$ is finite for every $P \in \P^m$
(and nonempty because it contains $P$).
Counting fibrewise, 
we thus have 
\[
\Num_{H,f(\P^m(K))}(X)
=
\sum_{\substack{Q \in f(\P^m(K)) \\ H(Q) \le X}} 1
= 
\sum_{\substack{P \in \P^m(K) \\ H(f(P)) \le X}} 
\frac{1}{\# f^{-1}(f(P))(K)}
\]
so that 
\[
\Num_{H,f(\P^m(K))}(X)
-
\frac{\Num_{f^* \! H, \P^m(K)}(X)}{\gamma}
=
\sum_{\substack{P \in \P^m(K) \\ H(f(P)) \le X}} 
\Big(
\frac{1}{\# f^{-1}(f(P))(K)} - \frac{1}{\gamma}
\Big).
\]
The sum is supported on 
\[
\Omega := \{P \in \P^m(K) : \#f^{-1}(f(P))(K) \ne \gamma\},
\]
and each summand---being a difference of reciprocal integers---lies in $(-1, 1)$.
Moreover,
since $f$ is a morphism, 
there exists a constant $C > 0$ 
such that if $H(f(P)) \le X$ then $H(P) \le CX^{1/d}$.
Hence
\begin{equation} \label{N_minus_gN}
\left| \Num_{H,f(\P^m(K))}(X)
-
\frac{\Num_{f^* \! H, \P^m(K)}(X)}{\gamma}
\right| 
\le \Num_{H, \Omega}(CX^{1/d}).
\end{equation}

Consider the natural action of $\Gamma_f$ on $\P^m$.
Clearly,
the orbit under $\Gamma_f$ of any given point $P$
is contained in the fibre $f^{-1}(f(P))$.
In particular,
\begin{equation} \label{orbit_vs_fibre}
\# \Gamma_f(K)P \le \#f^{-1}(f(P))(K).
\end{equation}
Moreover, by the orbit--stabilizer theorem, 
\begin{equation} \label{orbit_vs_group}
\# \Gamma_f(K)P \ne \gamma 
\iff
P \in \bigcup_{\varphi \in \Gamma_f \setminus \{\id\}} \Fix(\varphi).
\end{equation}
Note that by Proposition \ref{symmetry_group}(i),
the set on the right 
is a finite union of proper (linear) subvarieties of $\P^m$, 
and is therefore a thin set (of type 1).

Suppose \eqref{orbit_vs_fibre} is a strict inequality.
Then there exists a point $Q \in \P^m(K)$ 
such that 
$f(P) = f(Q)$ 
yet $\varphi(P) \ne Q$ for any $\varphi$ in $\Gamma_f(K)$.
It follows that $(P, Q)$ is a $K$-point 
of the variety $X'$ given by the union of all the irreducible components of $X_f$ which are \emph{not} 
the graph of 
some $\varphi \in \Gamma_f(K)$ (cf.~Lemma \ref{X_f}(iii)).
In particular, 
\begin{equation} \label{orbit_vs_fibre_implies}
	\#\Gamma_f(K)P < \#f^{-1}(f(P))(K) \implies P \in \pi(X'(K)).
\end{equation}
By Lemma \ref{X_f}(i) and (ii), 
the set on the right is thin (of type 2).

Combining \eqref{orbit_vs_group} and \eqref{orbit_vs_fibre_implies} 
shows that $\Omega$ itself is thin.
By \cite[9.7, 13.1]{Serre}, 
the counting function of $\Omega$ 
satisfies 
\begin{equation} \label{N_thin}
	\Num_{H, \Omega}(X) 
	=
	O 
	\begin{dcases}
		X^n & m = 1 \\
		X^{n(m+1/2)} \log X & m > 1
	\end{dcases}
\end{equation}
where the implicit constants depend on $f$ (via $\Omega$) and $K$.
Inserting \eqref{N_thin} into \eqref{N_minus_gN}
and using Theorem \ref{thm:main} 
yields 
\[
\Num_{H, f(\P^m(K))}(X)
=
\frac{c_K(f)}{\gamma} X^{n(m+1)/d}
+ 
O\big(X^{(nm+n-1)/d} \underbrace{\cdot \log X}_{\text{if }nm = 1}\big)
+
O\big(X^{nm/d}\underbrace{\cdot X^{n/2d} \log X}_{\text{if }m>1}\big).
\]
The rest is untangling cases.
\end{proof}

An explicit error term in Theorem \ref{thm:img} 
could be obtained using an effective form of Hilbert's ireducibility theorem 
(as in, e.g.,~\cite[Theorem 1.1]{PS} for $m=1$).

\subsection{A dynamical Schanuel theorem}

\begin{thm} \label{thm:ds}
	Let $f$ be an endomorphism of $\P^M$ of degree $d \ge 2$ defined over a number field $K$ 
	and let 
	\[\Hat{H}_f(P) = \lim_{i \to \infty} H(f^i(P))^{1/d^i} = \exp \hat{h}_f(P) \qquad (P \in \P^M(\Bar{K}))\]
	be the exponential canonical height attached to $f$.
	Then for every morphism $g : \P^m \to \P^M$ of degree $e \ge 1$ defined over $K$, 
	\begin{enumerate}[(i)]
\item the constants $c_K(f^i \circ g)$ tend to a finite positive limit $\Hat{c}_{K,f}(g)$ as $i \to \infty$, and 
\item as $X \to \infty$, \[\Num_{g^*\!\Hat{H}_f\!, \P^m(K)} \sim \Hat{c}_{K,f}(g) X^{n(m+1)/e}.\]
\end{enumerate}
\end{thm}

\begin{proof}
The canonical height enjoys (and is characterized by) the properties:
	\begin{enumerate}[(a)]
		\item there exists $C > 0$ such that $C^{-1} H(Q) \le \Hat{H}_f(Q) \le C H(Q)$ for all $Q$ in $\P^M(K)$;
		\item $\Hat{H}_f(f(Q)) = \Hat{H}_f(Q)^d$ for all $Q$ in $\P^M(K)$.
	\end{enumerate}
Let $P \in \P^m(K)$ and $X \ge 1$.
If $\Hat{H}_f(g(P)) \le X$,
then for any fixed $i \ge 0$,
\[H(f^i(g(P))) \le C\Hat{H}_f(f^i(g(P))) = C\Hat{H}_f(g(P))^{d^i} \le CX^{d^i}\]
by properties (a) and (b).
Thus 
\[\Num_{g^*\!\Hat{H}_f\!, \P^m(K)}(X)
	\le \Num_{(f^i \circ g)^*\! H, \P^m(K)}(CX^{d^i}).
\]
Now Theorem \ref{thm:main} implies the r.h.s.~is asymptotic to $c_K(f^i \circ g) (CX^{d^i})^{n(m+1)/ed^i}$ as $X \to \infty$.
It follows that
\begin{equation} \label{eq:cHat_upper}
\limsup_{X \to \infty} 
\frac{\Num_{g^*\!\Hat{H}_f\!, \P^m(K)}(X)}{X^{n(m+1)/e}}
\le c_K(f^i \circ g) C^{n(m+1)/ed^i}
\end{equation}
for each fixed $i$.
But since the l.h.s.~of \eqref{eq:cHat_upper} is independent of $i$, taking $i \to \infty$ yields
\begin{equation} \label{eq:limsup_liminf}
	\limsup_{X \to \infty} 
	\frac{\Num_{g^*\!\Hat{H}_f\!, \P^m(K)}(X)}{X^{n(m+1)/e}}
	\le
	\liminf_{i \to \infty} c_K(f^i \circ g) \underbrace{C^{n(m+1)/ed^i}}_{\to 1}
	= 
	\liminf_{i \to \infty} c_K(f^i \circ g)
\end{equation}
because $C > 0$ and $d > 1$.
Similar reasoning shows
\[
\Num_{(f^i \circ g)^*\! H, \P^m(K)}(C^{-1} X^{d^i}) \le \Num_{g^*\!\Hat{H}_f\!, \P^m(K)}(X)
\]
whence
\begin{equation} \label{eq:cHat_lower}
c_K(f^i \circ g) C^{-n(m+1)/ed^i} \le \liminf_{X \to \infty} 
\frac{\Num_{g^*\!\Hat{H}_f\!, \P^m(K)}(X)}{X^{n(m+1)/e}}
\end{equation}
for each fixed $i$,
so that
\begin{equation} \label{eq:liminf_limsup}
\limsup_{i \to \infty} c_K(f^i \circ g)
\le
\liminf_{X \to \infty} 
\frac{\Num_{g^*\!\Hat{H}_f\!, \P^m(K)}(X)}{X^{n(m+1)/e}}.
\end{equation}
Stringing together the inequalities 
\eqref{eq:liminf_limsup} $\le$ \eqref{eq:limsup_liminf}
reveals, in one fell swoop, that the limit
\[
\Hat{c}_{K,f}(g) := 
\lim_{i \to \infty} c_K(f^i \circ g) 
\]
exists, and thereupon that
\begin{equation} \label{eq:cHat_is}
\lim_{X \to \infty} \frac{\Num_{g^*\!\Hat{H}_f\!, \P^m(K)}(X)}{X^{n(m+1)/e}}
=
\Hat{c}_{K,f}(g).
\end{equation}
Finally, \eqref{eq:cHat_lower} and \eqref{eq:cHat_upper} imply via \eqref{eq:cHat_is} that $0 < \Hat{c}_{K,f}(g) < \infty$.
\end{proof}

\begin{rmk}
While Theorem \ref{thm:ds}(ii) could be called a ``dynamical Schanuel'' result, 
a true analog of Schanuel's theorem for the canonical height would include the power-saving error term, 
namely: 
\[\Num_{\Hat{H}_f, \P^m(K)}(X) \overset{?}{=} \Hat{c}_{K,f} X^{n(m+1)} + O(X^{n(m+1)-1} \log X)\]
where the $\log X$ disappears for $nm > 1$.
To obtain an error term in Theorem \ref{thm:ds}(ii) from Theorem \ref{thm:main}, 
one would need to understand 
the rate of convergence of $c_K(f^i \circ g)$ 
as well as the behaviour of $\Res f^i \circ g$
and the geometry of the archimedean fundamental domains $D_{f^i \circ g}$ as $i \to \infty$.
\end{rmk}

\begin{cor} \label{cor:ds}
	Let $g : \P^m \to \P^M$ be a morphism of degree $e \ge 1$ defined over $K$.
Then:
\begin{enumerate}[(i)]
	\item $\Hat{c}_{K,f^i}(g) = \Hat{c}_{K,f}(g)$ for all positive integers $i$;
	\item $\Hat{c}_{K,f^\varphi}(g) = \Hat{c}_{K,f}(g^{\varphi^{-1}})$ for all automorphisms $\varphi$ defined over $K$.
\end{enumerate}
\end{cor}
\begin{proof}
By Theorem \ref{thm:ds}(i),
\[\Hat{c}_{K,f^i}(g) = \lim_{j \to \infty} c_K(f^{ij} \circ g) = \Hat{c}_{K,f}(g)\]
is a subsequential limit, which proves (i).
Next, $\Hat{H}_{f^\varphi} = \Hat{H}_f \circ \varphi$ (since $H \circ \varphi^{-1} \asymp H$), 
so 
\[g^*\Hat{H}_{f^\varphi} 
= \Hat{H}_{f^\varphi} \circ g 
= \Hat{H}_f \circ \varphi \circ g 
= \Hat{H}_f \circ \varphi \circ g \circ \varphi^{-1} \circ \varphi
= \Hat{H}_f \circ g^{\varphi^{-1}} \circ \varphi\]
whence
\[
\{P : \Hat{H}_{f^\varphi}(g(P)) \le X\}
= 
\{P : \Hat{H}_f(g^{\varphi^{-1}}(\varphi(P))) \le X\}.
\]
The latter set is in one-to-one correspondence with 
$\{Q : \Hat{H}_f(g^{\varphi^{-1}}(Q)) \le X\}$ 
via $\varphi$
(cf.~Remark \ref{rmk:aut}),
so taking cardinalities shows
$\Num_{g^*\!\Hat{H}_{f^\varphi\!},\P^m(K)}(X) 
= 
\Num_{(g^{\varphi^{-1}})^*\!\Hat{H}_f\!,\P^m(K)}(X)$.
Applying Theorem \ref{thm:ds}(ii) proves (ii).
\end{proof}

Suppose $f$ has degree $d$ 
and let $F$ be a lift of $f$.
We can give a formula for $\Hat{c}_{K,f}$ in terms of 
the \emph{Green's functions} of $F$, 
which are defined at each place $v$ of $K$ by 
\begin{align*}
	\mathcal{G}_{F,v} &: K_v^{M+1} - 0^{M+1} \to \R \\
	x &\mapsto \lim_{i \to \infty} \frac{\log {|F^i(x)|_v}} {d^i}.
\end{align*}
That this limit converges 
is a consequence of the estimate $\log |F(x)|_v = d \log |x|_v + O(1)$ and a routine ``telescoping'' argument.
It is known that the convergence is uniform (see, e.g.,~\cite[proof of Proposition 5.58(e)]{SilvermanADS})
and that $\mathcal{G}_{F,v}$ is H\"older continuous (see, e.g.,~\cite[Theorem 3.1]{Gauthier}).

\begin{propn} \label{green_properties}
The Green's function has the following properties.
\begin{enumerate}[(i)]
	\item 
	(characterization):
	$\mathcal{G}_{F,v}$ is the unique function satisfying both
	\[\mathcal{G}_{F,v}(F(x)) = d \cdot \mathcal{G}_{F,v}(x)\]
	and 
	\[\mathcal{G}_{F,v}(x) - \log |x|_v \text{ is bounded}.\]
	\item 
	(homogeneity):
	\[\mathcal{G}_{\lambda F,v}(\eta x) 
	= \mathcal{G}_{F,v}(x)
	+ \log |\eta|_v + \frac{\log |\lambda|_v}{d - 1}\]
	for all nonzero scalars $\lambda, \eta$.
	\item 
	(exact formula): 
	if $f$ has good reduction at $v$, 
	then 
	\[\mathcal{G}_{F,v}(x) = \log |x|_v + \frac{\log |F|_v}{d - 1}.\]
	\item 
	(canonical height decomposition):
	if $x \in \A^{M+1}$ is a lift of $P \in \P^M$ then
	\[
	\hat{h}_f(P) = \sum_v \frac{n_v}{n} \mathcal{G}_{F,v}(x).
	\]
\end{enumerate}
\end{propn}
\begin{proof}
For $M = 1$ these are proved in \cite[Section 5.9]{SilvermanADS}; 
the same proofs work in general (e.g.,~\cite[Lemma 6]{Ingram_macaulayres}).
\end{proof}

We now give the promised formula for $\Hat{c}_{K,f}$.

\begin{propn} \label{green_local_factors}
Let $g : \P^m \to \P^M$ be a morphism of degree $e \ge 1$ defined over $K$ 
and let $G$ be a lift of $g$.
Then $\Hat{c}_{K,f}(g) =$
\begin{align*}
c_K(m)
\prod_{v \mid \infty}
\mu_v {\{z \in K_v^{m+1} : \mathcal{G}_{F,v}(G(z)) \le 0\}} 
\prod_{v \nmid \infty} 
\int_{O_v^{m+1}} 
\bigg(
\frac{|z|_v}{\exp \mathcal{G}_{F,v}(G(z))/e}
\bigg)^{n_v(m+1)}
\d \mu_v(z)
\end{align*}
where 
$c_K(m)$ is as in \eqref{cKm}.
In fact, 
if 
\[
|\Hat{F} \circ G|_v := \lim_{i \to \infty} |F^i \circ G|_v^{1/d^i}
\]
then the lift-independent 
local factors converge as well:
\[
\lim_{i \to \infty} c_{K,v}(f^i \circ g)
=
\Hat{c}_{K,v,f}(g)
:=
\begin{dcases}
	\mu_v \{z \in K_v^{m+1} : \exp \mathcal{G}_{F,v}(G(z)) \le  |\Hat{F} \circ G|_v\} & v \mid \infty \\
	\int_{O_v^{m+1}} 
\bigg( 
\frac{|z|_v \, |\Hat{F} \circ G|_v^{1/e}}
{\exp \mathcal{G}_{F,v}(G(z))/e}
\bigg)^{n_v(m+1)}
	\d \mu_v(z) & v \nmid \infty 
\end{dcases}.
\]
\end{propn}
\begin{proof}
Redistributing the factors of 
\[
H(f^i \circ g)^{n(m+1)/d^i e}
= \prod_v |F^i \circ G|_v^{n_v(m+1)/d^i e}
\]
yields $c_K(f^i \circ g) =$
\[
c_K(m) 
\prod_{v \mid \infty} 
\mu_v {\{z \in K_v^{m+1} : |F^i(G(z))|_v \le 1\}}
\prod_{v \nmid \infty}
\int_{O_v^{m+1}} 
\bigg(
\frac{|z|_v}{|F^i(G(z))|_v^{1/d^i e}}
 \bigg)^{n_v(m+1)}
\d \mu_v(z).
\]
By definition of the Green's function,
\[
\lim_{i \to \infty} |F^i(G(z))|_v^{1/d^i e}
=
\exp \mathcal{G}_{F,v}(G(z))/e
\]
so the Proposition follows from uniform convergence.
\end{proof}

\begin{eg} \label{eg:powermap}
	For $d \ge 2$ let 
	\[f(x_0 : \ldots : x_M) = (x_0^d : \ldots : x_M^d)\]
	be the $d$\textsuperscript{th} power map.
	Then
for all number fields $K$ 
and all nonconstant morphisms $g$ defined over $K$,
\[\Hat{c}_{K,f}(g) = c_K(g).\]
Moreover, for all places $v$ of $K$,
\[\Hat{c}_{K,v,f}(g) = c_{K,v}(g).\]
In fact, if $v$ is nonarchimedean, 
then $c_{K,v}(f^i \circ g) = c_{K,v}(g)$ for all $i \ge 0$.
\end{eg}
\begin{proof}
	It is well-known that $\Hat{H}_f = H$.
Thus by Theorem \ref{thm:ds}(ii) and Theorem \ref{thm:main},
	\[
	\Hat{c}_{K,f}(g)
	\sim 
	\frac{\Num_{g^*\!\Hat{H}_f, \P^m(K)}(X)}{X^{n(m+1)/\deg g}}
	= 
	\frac{\Num_{g^*\!H,\P^m(K)}(X)}{X^{n(m+1)/\deg g}}
	\sim 
	c_K(g).
	\]
	Letting $X \to \infty$ proves the first claim.
Now if $F = (X_0^d, \ldots, X_M^d)$ 
and if $G$ is any lift of $g$
then for all $i \ge 0$,
\begin{equation} \label{eq:FGx}
|F^i(G(z))|_v 
= \max_j |G_j(z)^{d^i}|_v
= \max_j |G_j(z)|_v^{d^i}
= |G(z)|_v^{d^i}
\end{equation}
for all $z \in K_v^{m+1}$
and all places $v$ of $K$.
If $v \nmid \infty$,
then 
\[
|F^i \circ G|_v
= \max_j |G_j^{d^i}|_v
= \max_j |G_j|_v^{d^i}
= |G|_v^{d^i}
\]
by the Gauss lemma,
so the integral formula implies $c_{K,v}(f^i \circ g) = c_{K,v}(g)$,
whence \emph{a fortiori} $\Hat{c}_{K,v,f}(g) = c_{K,v}(g)$.
Meanwhile, if $v \mid \infty$,
then by \eqref{eq:FGx} 
and homogeneity,
\[
D_{f^i \circ g, v}
= \frac{|F^i \circ G|_v^{1/d^i e}}{|G|_v^{1/e}} D_{g,v}.
\]
Since $|F^i \circ G|_v^{1/d^i} \to |G|_v$ as $i \to \infty$,
taking volumes 
shows $\Hat{c}_{K,v,f}(g) = c_{K,v}(g)$.
\end{proof} 

Our next Example shows that $\Hat{c}_{K,f}(g) \ne c_K(g)$ in general.

\begin{eg} \label{eg:chebyshev}
	Let $t_d$ denote 
	the dynamicist's Tchebyshev polynomial of degree $d$,
	defined by
\begin{equation} \label{eq:Cheb_def} 
t_d(z + 1/z) = z^d + 1/z^d
\end{equation}
	and regarded as an endomorphism of $\P^1$
	(see, e.g.,~\cite[Section 6.2]{SilvermanADS}).
If $d \ge 2$ then:
\[\Hat{c}_{\Q,p,t_d}(\id) = 1\]
for all primes $p$; 
\[\Hat{c}_{\Q,\infty,t_d}(\id) = \frac{24 + 8\sqrt{5}}{3}\]
and
\[
\Hat{c}_{\Q,t_d}(\id) = \frac{16}{\pi^2} 
\qquad \Big({\ne c_\Q(\id) = \frac{12}{\pi^2}} \Big).
\]
\end{eg}
\begin{proof}
Choose the lift $T_d(X, Y) := (t_d(X/Y) Y^d, Y^d)$ of $t_d$.
Then $T_d$ inherits the composition law from $t_d$,
so that
\begin{equation} \label{eq:Cheb_law}
	T_d \circ T_e = T_{de}
\end{equation}
for all $d, e \ge 0$.\footnote{This is not the only lift of $t_d$ that satisfies \eqref{eq:Cheb_law}. In general, if $T_d(X, Y) = \lambda_d(t_d(X/Y) Y^d, Y^d)$ for some nonzero scalars $\lambda_d$ then $T_d \circ T_e = T_{de}$ if and only if $\lambda_{de} = \lambda_d^{\phantom{d}} \lambda_e^d$ for all $d, e \ge 0$, which is equivalent to $\lambda_d = \lambda_2^{d-1}$ for all $d \ge 0$.}

Let $p$ be a finite prime.
Since $t_d$ is monic with integer coefficients, 
$t_d$ has good reduction at $p$ and $|T_d|_p = 1$.
Thus by Proposition \ref{green_properties}(iii), 
$\mathcal{G}_{T_d,p}(x, y) = \log |x, y|_p$.
It follows from \eqref{eq:Cheb_law} 
that $|T_d^i|_p = |T_{d^i}|_p = 1$ for all $i \ge 0$, 
whence $|\Hat{T}_d|_p = 1$.
So $\Hat{c}_{\Q,p,t_d}(\id) = 1$ as well.

The situation is less trivial at the infinite place.
First, we show that 
\begin{equation} \label{eq:Cheb_green}
\mathcal{G}_{T_d,\infty}(x, y) 
= \log \begin{dcases*}
|y| & if $|x| \le 2|y|$ \\
\frac{|x| + \sqrt{x^2 - 4y^2}}{2} & if $|x| \ge 2|y|$
\end{dcases*}
\end{equation} 
when $x$ and $y$ are real.
To that end, we homogenize \eqref{eq:Cheb_def}:
define 
\[F_d(X, Y) = (X^d, Y^d) \quad\text{and}\quad P(X, Y) = (X^2 + Y^2, XY)\] 
so that $T_d \circ P = P \circ F_d$. 
This and the identity $\log |P(z, w)| = 2\log|z, w| + O(1)$
yield
\begin{equation} \label{eq:G_T_P}
\mathcal{G}_{T_d,\infty}(P(a, b)) 
= \lim_{i \to \infty} \frac{2 \log |F^i(a, b)| + O(1)}{d^i} 
= \lim_{i \to \infty} \frac{2 \log |a^{d^i}, b^{d^i}|}{d^i}
= \log |a^2, b^2|.
\end{equation}
Now if $P(a, b) = (x, y)$ 
then $a^2$ and $b^2$ 
are roots of the polynomial
\[\lambda^2 - (a^2 + b^2) \lambda + a^2 b^2 = \lambda^2 - x\lambda + y^2;\]
conversely, either $P(a, b) = (x, y)$ or else $P(-a, b) = (x, y)$.
By the quadratic formula,
\[
|a^2, b^2| = 
\bigg|\frac{x + \sqrt{x^2 - 4y^2}}{2}, \frac{x - \sqrt{x^2 - 4y^2}}{2}\bigg|
\]
(cf.~\cite[Examples 1.2(1)]{DG}).
If $|x| \le 2|y|$ then $a^2$ and $b^2$ are complex conjugates,
so 
\begin{equation} \label{eq:Cheb_green_piece1}
|a^2, b^2| = \bigg|\frac{x \pm i\sqrt{4y^2 - x^2}}{2}\bigg|
= \frac{1}{2} \sqrt{x^2 + (4y^2 - x^2)}
= |y|.
\end{equation}
If $|x| \ge 2|y|$ then the discriminant is non-negative, 
and the inequalities 
\[0 \le \sqrt{x^2 - 4y^2} \le \sqrt{x^2} = |x|\]
imply 
\begin{equation} \label{eq:Cheb_green_piece2}
|a^2, b^2| = \frac{|x| + \sqrt{x^2 - 4y^2}}{2}.
\end{equation}
Putting together 
\eqref{eq:G_T_P},
\eqref{eq:Cheb_green_piece1}, and \eqref{eq:Cheb_green_piece2} 
proves \eqref{eq:Cheb_green}.

Next we compute the area of the normalized limiting fundamental domain 
\[
\Hat{D}_{T_d,\infty} := \{(x, y) \in \R^2 : \exp \mathcal{G}_{T_d,\infty}(x, y) \le 1\}.
\]
Let $(x, y) \in \R^2$.
If $|x| \le 2|y|$ 
then $\exp \mathcal{G}_{T_d,\infty}(x, y) \le 1$ if and only if $|y| \le 1$.
If $|x| \ge 2|y|$ 
then the following are equivalent: 
\begin{align*}
\exp \mathcal{G}_{T_d,\infty}(x, y) &\le 1 \\
|x| + \sqrt{x^2 - 4y^2} &\le 2 \\
x^2 - 4y^2 &\le (2 - |x|)^2 = 4 - 4|x| + x^2 \\
|x| &\le 1 + y^2
\end{align*}
Thus, the region $\Hat{D}_{T_d,\infty}$
is bounded by the lines $y = \pm 1$ 
and the parabolae $x = \pm(1 + y^2)$.
Its area is therefore 
\begin{equation} \label{eq:area_hatD_T}
\area \Hat{D}_{T_d, \infty} 
= \int_{-1}^1 \int_{-1-y^2}^{1+y^2} \mathrm{d}x \, \mathrm{d}y
= 4 \int_0^1 1 + y^2 \, \mathrm{d}y
= \frac{16}{3}.
\end{equation}
See Table \ref{table:D_Chebyshev} for a picture.

By Proposition \ref{green_local_factors}, 
$\Hat{c}_{\Q,t_d}(\id)$ is the product of \eqref{eq:area_hatD_T} and 
\begin{equation} \label{ok3}
c_\Q(2) = \frac{1}{2\zeta(2)} = \frac{3}{\pi^2}
\end{equation}
whereas $\Hat{c}_{\Q,\infty,t_d}(\id)$ 
is the product of 
\eqref{eq:area_hatD_T}
and $|\Hat{T}_d|^2$, by homogeneity.
The calculation of $|\Hat{T}_d|$ is arduous and we only give a sketch.
From the explicit formula
\cite[Proposition 6.6(e)]{SilvermanADS}
\[
t_d(z) 
=
\sum_{0 \le k \le d/2} 
(-1)^k \frac{d}{d-k} \binom{d-k}{k} z^{d-2k} \qquad (d \ge 1)
\]
and Stirling's approximation 
\[
\log n! = n \log n - n + \tfrac{1}{2} \log n + O(1)
\]
we have 
\[
\log |T_d| 
= 
\max_{0 < k < d/2} \theta\big(\tfrac{k}{d}\big)d + O(\log d) 
\]
where 
\[\theta(\lambda) 
= (1 - \lambda) \log (1 - \lambda) 
- \lambda \log \lambda 
- (1 - 2\lambda) \log (1 - 2\lambda).
\]
Now $\theta$ is concave on $(0, \tfrac{1}{2})$ 
with a unique maximum at $\lambda_0 = \tfrac{1}{2}\big(1-\tfrac{1}{\sqrt 5}\big) \approx 0.276\ldots$;
this is because 
\[\theta'(\lambda) = \log \frac{(1-2\lambda)^2}{\lambda(1-\lambda)}
= 0 \iff 5\lambda^2 - 5\lambda + 1 = 0\]
(n.b.~$(t \log t)' = 1 + \log t$).
It follows that 
\[\lim_{d \to \infty} |T_d|^{1/d} = \exp \theta(\lambda_0) = \frac{1 + \sqrt{5}}{2},\]
at least according to \texttt{WolframAlpha}.
\end{proof}

\begin{table}[h]
	\centering
	\begin{tabular}{ccc}
		\includegraphics[width=7cm,page=3]{DT_2.pdf} & & \includegraphics[width=7cm,page=3]{DT_4.pdf} \\ 
		\includegraphics[width=7cm,page=3]{DT_8.pdf} & & \includegraphics[width=7cm,page=3]{DT_16.pdf} \\ 
		\multicolumn{3}{c}{\includegraphics[width=7cm,page=3]{DT_hat.pdf}}
	\end{tabular}
	\caption{Normalized fundamental domains for $T_2^i$ ($i = 1, 2, 3, 4$) and their limit.}
	\label{table:D_Chebyshev}
\end{table}

Our final Example shows that 
Corollary \ref{cor:ds}(ii) is false if $\varphi$ is \emph{not} defined over $K$.

\begin{eg}
Let 
\[s(z) = \frac{z^2 - 1}{2z}.\]
Then $s$ is $\Q(i)$-conjugate to $z^2$, 
but 
\[
\Hat{c}_{\Q,s}(\id) = \frac{4}{\pi} 
\qquad \Big({\ne \Hat{c}_{\Q,z^2}(\id) = \frac{12}{\pi^2}} \Big).
\]
\end{eg}

\begin{proof}
If
\[\varphi(z) = i\frac{1+z}{1-z}\]
then $s(\varphi(z)) = \varphi(s(z))$.
Choose the lift 
$S(X, Y) = (X^2 - Y^2, 2XY)$.
By uniqueness of the Green's function (Proposition \ref{green_properties}(i)),
\[
\mathcal{G}_{S,\infty}(x,y) = \log \sqrt{x^2 + y^2}
\]
for all $(x, y) \in \R^2 - (0, 0)$.
Indeed, 
\[\log \sqrt{x^2 + y^2} = \log |x, y| + O(1)\]
because the 2-norm and $\infty$-norm are equivalent; 
and 
\[
\log \sqrt{(x^2 - y^2) + (2xy)^2} 
= \log \sqrt{(x^2 + y^2)^2}
= 2 \log \sqrt{x^2 + y^2}.
\]
Thus the limiting normalized fundamental domain is the unit disk, and so
\begin{equation} \label{ok1}
	\area \Hat{D}_{S,\infty} = \pi.
\end{equation}
See Table \ref{table:D_Octagon} for some pictures.
Meanwhile, it's easy to prove that $s$ has good reduction at every odd prime.
As for $v = 2$, 
a straightforward calculation gives
\[
v(S(x, y)) 
= 2v(x, y) 
+ [v(x) = v(y)]
\]
where $[ \ {\cdot} \ ]$ is the Iverson bracket;
and with a bit more work one can show that
\[
\mathcal{G}_{S,2}(x, y) = \log |x, y|_2 - \frac{1}{2}\big[|x|_2 = |y|_2\big].
\]
Thus by Proposition \ref{green_local_factors},
\begin{equation} \label{ok2}
	\Hat{c}_{\Q,2,s}(\id) = \frac{4}{3}.
\end{equation}
Multiplying \eqref{ok3}, \eqref{ok1}, and \eqref{ok2} proves the claim.
\end{proof}

\begin{table}[h]
	\centering
	\begin{tabular}{ccc}
		\includegraphics[width=6cm,page=2]{DS1.pdf} & & \includegraphics[width=6cm,page=2]{DS2.pdf} \\[2ex]
		\includegraphics[width=6cm,page=2]{DS3.pdf} & & \includegraphics[width=6cm,page=2]{DS4.pdf}
	\end{tabular}
	\caption{Normalized fundamental domains for $S^i$ ($i = 1, 2, 3, 4$). 
	The perimeters $8 \int_0^{\pi/4} \sec(u)^{1+1/2^i} \d u$ tend to $8 \log {(1 + \sqrt{2})} > 2\pi$ as $i \to \infty$.
	}
	\label{table:D_Octagon}
\end{table}

In a subsequent article, we will focus on the practical computation of the constant $c_K(f)$ for any morphism $f$,
with an eye to 
evaluating $\Hat{c}_{K,f}$ explicitly.

\section{Morphisms} \label{sec:morphisms}

Let $K$ be a field
and let $K[X] = K[X_0, \ldots, X_m]$ be the polynomial ring in $m+1$ variables over $K$.
For each integer $d$ put 
\[
K[X]_d := \{0\} \cup \{\Phi \in K[X] : \Phi \text{ is homogeneous of degree } d\}.
\]
In this section our principal objects of study will be elements $F = (F_0, \ldots, F_M)$ of $K[X]_d^{M+1}$ for $d \ge 0$.
Our goal is an algebraic characterization of good reduction.
While much of this theory is classical,
we could not locate an adequate treatment of the case $M > m$.

\subsection{Multiindices} \label{sec:multiindices}
Let $\N = \{0, 1, 2, \ldots\}$ 
denote the set of non-negative integers.
Elements of $\N^{m+1}$ are called \emph{multiindices}.
Two multiindices $\alpha$ and $\beta$ 
may be compared via
\[\alpha \le \beta \iff \alpha_i \le \beta_i \text{ for all }i;\]
this is the product order on $\N^{m+1}$ inherited from the usual order on $\N$.
The sum $\alpha + \beta$ is defined componentwise,
as is the difference $\alpha - \beta$ provided $\alpha \ge \beta$.
The \emph{degree} of a multiindex $\alpha$ 
is the sum of its entries:
\[|\alpha| = \alpha_0 + \ldots + \alpha_m.\]
Clearly $|\alpha + \beta| = |\alpha| + |\beta|$,
and if $\alpha \ge \beta$ then $|\alpha| \ge |\beta|$ and $|\alpha - \beta| = |\alpha| - |\beta|$.
The \emph{factorial} and the \emph{binomial coefficients}
are extended to multiindices via
\[\alpha! = \alpha_0! \ldots \alpha_m! 
\quad\text{and}\quad 
\binom{\alpha}{\beta} 
= \binom{\alpha_0}{\beta_0} \ldots \binom{\alpha_m}{\beta_m}
= \begin{dcases*}
	\frac{\alpha!}{\beta!(\alpha-\beta)!} & if $\beta \le \alpha$, \\
	0 & otherwise.
\end{dcases*}
\]
To every multiindex $\alpha$ 
we associate the monomial
\[X^\alpha := X_0^{\alpha_0} \ldots X_m^{\alpha_m}\]
so that
\[
X^\alpha X^\beta = X^{\alpha+\beta} 
\quad\text{and}\quad X^\alpha \mid X^\beta \iff \alpha \le \beta
\quad\text{and}\quad
\deg X^\alpha = |\alpha|.
\]

\subsection{Pseudoinverses}

\begin{defn}
	Let $F \in K[X]_d^{M+1}$ for some $d \ge 0$.
	A \emph{pseudoinverse} for $F$ 
	is an $(m+1)$-by-$(M+1)$ matrix of homogeneous forms $G_{ij} \in K[X]_{e-d}$ such that
	\begin{equation} \label{pseudoinverse}
		\sum_{j=0}^M G_{ij}(X) F_j(X) = X_i^e \qquad (i = 0, \ldots, m)
	\end{equation}
	for some integer $e$. 
	The \emph{degree} of the pseudoinverse $G$ is the integer $e - d$.
\end{defn}

\begin{rmk}
	The terminology is inspired by considering maps $F$ of the form
	\[F_j(X) = \sum_{k=0}^m F_{jk} X_k^d \qquad (j = 0, \ldots, M)\]
	called \emph{minimally critical} in \cite{Ingram_mincrit}.
	Such a map defines a morphism just when the matrix $F_{jk}$ has rank equal to the number $m+1$ of its columns.
	If $G_{ij}$ are constant polynomials,
	then for all $i$ we have 
	\[
	\sum_{j=0}^M G_{ij} F_j 
	= \sum_{j=0}^M G_{ij} \sum_{k=0}^m F_{jk} X_k^d
	= \sum_{k=0}^m \Big( \sum_{j=0}^M G_{ij} F_{jk} \Big) X_k^d 
	= \sum_{k=0}^m (GF)_{ik} X_k^d,
	\]
	so that $G$ is a pseudoinverse for $F$ if and only if $G$ is a left-inverse for $F$ (as matrices).
	Though left-inverses are not unique in general,
	when $K$ is a subfield of $\C$ and the columns of $F$ are linearly independent, then $F$ has a canonical left-inverse given in terms of the conjugate transpose, namely $G = (F^* F)^{-1} F^*$. This, in turn, coincides with the Moore--Penrose pseudoinverse $F^+$ of $F$---whence the name.
\end{rmk}

\begin{eg}
	A degree-0 pseudoinverse of $F(X, Y) = (X^2 + 2Y^2, 3X^2, 5X^2)$ is 
	\[G = F^+ = \begin{bmatrix} 35 & 2 \\ 2 & 4\end{bmatrix}^{-1} \begin{bmatrix} 1 & 3 & 5 \\ 2 & 0 & 0\end{bmatrix} = \frac{1}{68} \begin{bmatrix} 0 & 6 & 10 \\ 34 & -3 & -5 \end{bmatrix}.\]
\end{eg}

\begin{defn}
	The \emph{Sylvester map} is the $K$-linear transformation 
	$\mathcal{S} : K[X]^{M+1} \to K[X]$ 
	defined by 
	\[\mathcal{S}(A_0, \ldots, A_M) := \sum_{j=0}^M A_j F_j.\]
\end{defn}

\begin{lem}[Sylvester's Image Theorem]
	Let $V_D = \prod_{j=0}^M K[X]_{D-d}$.
	Then \[\mathcal{S}(V_D) = \langle F_0, \ldots, F_M \rangle \cap K[X]_D\]
	for all integers $D$.
\end{lem}

\begin{proof}
	The containment $\subseteq$ is obvious.
	For the reverse, 
	suppose $\sum_j A_j F_j \in K[X]_D$. 
	Let $B_j$ be the homogeneous part of $A_j$ in degree $D-d$.
	Then $(B_0, \ldots, B_M) \in V_D$, 
	and $\sum_j A_j F_j - \mathcal{S}(B) = \sum_j (A_j - B_j) F_j$ is homogeneous of degree $D$, 
	hence equal to the sum of its homogeneous parts in degree $D$, hence equal to 0.
	Thus $\sum_j A_j F_j = \mathcal{S}(B)$.
\end{proof}

We say the Sylvester map is \emph{surjective in degree $D$} to mean $\mathcal{S}(V_D) = K[X]_D$.
Trivially, the Sylvester map is surjective in every negative degree,
as well as in degree 0 just when $d = 0$ and $F_j \ne 0$ for some $j$.

\begin{propn} \label{sylv_null}
	Let $F_0, \ldots, F_M \in K[X]_d$ for some $d \ge 0$.
	The following are equivalent.
	\begin{enumerate}[(i)]
		\item $F_0, \ldots, F_M$ have no nontrivial common zeroes over $\Bar{K}$.
		\item There exists an integer $e \ge 0$ such that $X_0^e, \ldots, X_m^e \in \langle F_0, \ldots, F_M \rangle$.
		\item $F$ admits a pseudoinverse of degree $e - d$ defined over $K$.
		\item There exists an integer $D \ge 0$ such that $\langle X_0, \ldots, X_m \rangle^D \subseteq \langle F_0, \ldots, F_M \rangle$.
		\item The Sylvester map is surjective in some non-negative degree $D$.
	\end{enumerate}
\end{propn}

\begin{figure}[h]
	\begin{tikzpicture}[scale=.8, imp/.style={line width=.5pt,double equal sign distance,-implies}, iff/.style={line width=.5pt,double equal sign distance,implies-implies}]
		\node (A) at (0, 1) {(i)};
		\node (B) at (0, -1) {(ii)};
		\node (C) at ({-sqrt(3)}, 0) {(iii)};
		\node (D) at ({+sqrt(3)}, 0) {(iv)};
		\node (E) at ({2+sqrt(3)}, 0) {(v)};
		\draw [imp] (A) -- (B);
		\draw [imp] (B) -- (C);
		\draw [imp] (C) -- (A);
		\draw [imp] (A) -- (D);
		\draw [imp] (D) -- (B);
		\draw [iff] (D) -- (E);
	\end{tikzpicture}
	\caption{Logical flow of the proof of Proposition \ref{sylv_null}.}
\end{figure}
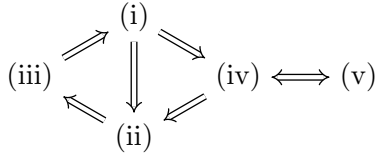

\begin{proof}
	Suppose (i) holds. 
	Then $V(F)(\Bar K) \subseteq \{0^{m+1}\}$,
	so by the semirational Nullstellensatz \cite[Theorem 11.9]{Clark} we have
	\[\rad {\langle F_0, \ldots, F_M \rangle} = I(V(F)(\Bar K)) \supseteq I(\{0^{m+1}\}) = \langle X_0, \ldots, X_m \rangle.\]
	Thus there exist integers $e_i$
	such that $X_i^{e_i} \in \langle F_0, \ldots, F_M \rangle$ for all $i$. 
	Letting $e := \max_i e_i$ proves (ii), 
	while letting $D := 1 + \sum_i (e_i - 1)$ proves (iv).\footnote{To wit, if $|\alpha| = D$ then $X^\alpha$ must be divisible by some $X_i^{e_i}$, lest $|\alpha| = \sum_i \alpha_i \le \sum_i (e_i - 1) < D$.}
	Of course, (iv) implies (ii) as well, but with a worse exponent.
	Clearly (iii) implies (i): if $F_j(x_0, \ldots, x_m) = 0$ for all $j$ for some $x_i$ in $\Bar{K}$ then by \eqref{pseudoinverse} $x_i^e = 0$ for all $i$, so $x_i = 0$ for all $i$.
	The remaining implications follow from Sylvester's Image Theorem.
	Given (ii), $\mathcal{S}(V_e)$ contains each $X_i^e$,
	so there exist $G_i \in V_e$ such that $\mathcal{S}(G_i) = X_i^e$ for all $i$,
	constituting a pseudoinverse $G$ for $F$, giving (iii).
	Finally, writing $\langle F \rangle$ for $\langle F_0, \ldots, F_M \rangle$, 
	we see directly that (iv) and (v) are equivalent:
	\begin{align*}
		\langle X_0, \ldots, X_m \rangle^D \subseteq \langle F \rangle 
		&\iff X^\alpha \in \langle F \rangle \text{ for all } |\alpha| = D \\
		&\iff \Span_K \{X^\alpha : |\alpha| = D\} \subseteq \langle F \rangle \\ 
		&\iff K[X]_D \subseteq \langle F \rangle \\ 
		&\iff \langle F \rangle \cap K[X]_D = K[X]_D\\ 
		&\iff \mathcal{S}(V_D) = K[X]_D. \qedhere 
	\end{align*}
\end{proof}

\subsection{Valuations}
In this section and the next,
\begin{itemize} 
	\item $K$ is a field with discrete valuation $v$
	\item $O = \{z \in K : v(z) \ge 0\}$ is the ring of $v$-integers
	\item $\mf m = \{z \in K : v(z) > 0\}$ is the maximal ideal of $O$ 
	\item $\pi \in O$ is a uniformizer, meaning $v(\pi)$ generates $v(K^\times)$
	\item $\Tilde{k} = O/\mf m$ is the residue field.
\end{itemize} 

As is standard, we extend $v$ 
to arbitrary lists $\Phi$ of polynomials with coefficients in $K$ by setting $v(\Phi)$ to be the minimum valuation among all coefficients in all members of $\Phi$. 
(Since constants are polynomials, this applies in particular to all lists of elements of $K$.)
This extension inherits the following properties from $v$:
\begin{enumerate}
	\item multiplicativity: $v(\Phi\Psi) = v(\Phi) + v(\Psi)$ for all polynomials $\Phi$ and tuples $\Psi$ (Gauss's lemma);
	\item ultrametricity: $v(\Phi + \Psi) \ge v(\Phi, \Psi) \ (= \min\{v(\Phi), v(\Psi)\})$ for all tuples $\Phi$ and $\Psi$ of the same length, with equality if $v(\Phi) \ne v(\Psi)$.
\end{enumerate}
Finally, we say $\Phi$ is \emph{$v$-integral} if $v(\Phi) \ge 0$ and \emph{$v$-normalized} if $v(\Phi) = 0$.

\begin{defn} \label{def:good_red}
	Let $f : \P^m \to \P^M$ be a morphism defined over $K$ 
	and let $F = (F_0, \ldots, F_M)$ be a $v$-normalized lift of $f$.
	The \emph{reduction of $f$ (mod $v$)} is the rational map $\Tilde{f} : \P^m \dashrightarrow \P^M$ defined over $\Tilde k$ by the homogeneous forms $\Tilde{F_0}, \ldots, \Tilde{F_M}$.
	We say $f$ has \emph{good reduction (mod $v$)} just when $\Tilde f$ is a morphism.
\end{defn}

\begin{propn} \label{good_red}
	Let $f : \P^m \to \P^M$ be a morphism of degree $d$ defined over $K$, and let 
	$F \in K[X]_d^{M+1}$ be a $v$-normalized lift of $f$.
	Then $f$ has good reduction (mod $v$) if and only if $F$ admits a $v$-integral pseudoinverse.\footnote{Every $v$-integral pseudoinverse of $F$ is automatically $v$-normalized. 
		Indeed, if \eqref{pseudoinverse} holds 
		then 
		\[0 = v(X_i^e) = v\big(\!\sum_j G_{ij} F_j\big) \ge \min_j v(G_{ij} F_j) \ge \min_j v(G_{ij}) + v(F_j) \ge v(G_i) + v(F)\] 
		for all $i$,
		so that if $v(F) \ge 0$ and $v(G) \ge 0$ then in fact $v(F) = v(G_i) = 0$ for all $i$.
		(This only requires the ``easy direction'' of Gauss's lemma.)
	}
\end{propn}

\begin{proof}
	If $G$ is a $v$-integral pseudoinverse for $F$, 
	then $G_{ij} \in O[X]$ because $v(G) \ge 0$, so we may reduce $G_{ij}$ mod $\mf m$.
	Doing so yields, by \eqref{pseudoinverse}, 
	a pseudoinverse $\Tilde{G}$ of $\Tilde{F}$.
	By Proposition \ref{sylv_null}, $\Tilde{f}$ is a morphism, so by definition, $f$ has good reduction (mod $v$).
	
	The other direction is nontrivial, because a pseudoinverse $\Tilde{G}$ of $\Tilde{F}$ need not lift to a pseudoinverse of $F$; we merely have $\sum_j G_{ij} F_j - X_i^e \in \mf m[X]$ for all $i$.
	Instead, we must use the Sylvester map.
	Supposing $\Tilde{f}$ is a morphism, 
	Proposition \ref{sylv_null} implies that
	the Sylvester map $\Tilde{\mathcal{S}}$ associated to $\Tilde{F}$ is surjective in some non-negative degree $D$.
	That means that the matrix $[\Tilde{\mathcal{S}}]$ with respect to the standard bases $(X^\beta e_j : |\beta| = D-d, \, 0 \le j \le M)$ of $\Tilde{V_D}$ and $(X^\gamma : |\gamma| = D)$ of $\Tilde{k}[X]_D$ has a nonzero maximal minor $\Tilde{\Delta}$.
	Since the reduction map $O \to \Tilde{k}$ is a ring homomorphism, 
	the corresponding submatrix of $[\mathcal{S}]$---the matrix of $\mathcal{S}$ w.r.t.~the standard bases of $V_D$ and $K[X]_D$---has $v$-unit determinant, and is thus invertible.
	By Cramer's rule and the Leibniz formula, plus the hypothesis that $v(F) = 0$,
	there exist $G_i \in V_D$ such that $\mathcal{S}(G_i) = X_i^D$ 
	and $v(G_{ij}) \ge 0 - v(\Delta) = 0$ 
	for all $i, j$.
	Thus, $G$ is a $v$-integral pseudoinverse for $F$.
\end{proof}

\subsection{Macaulay's theory}
The results in this subsection are not required for the proof of Theorem \ref{thm:main}; we simply wish to give some context.

It follows from the equivalence of (iv) and (v) in Proposition \ref{sylv_null} that the set of non-negative degrees in which $\mathcal{S}$ is surjective is upwards closed.
Although the least such degree may vary with $F$, 
Macaulay proved that if $\mathcal{S}$ is surjective in some non-negative degree, then $\mathcal{S}$ is already surjective in degree $D_0 := (m+1)(d-1) + 1$ (\textit{Macaulay's bound}).

\begin{defn}
	The \emph{Sylvester--Macaulay matrix} is the matrix $[\mathcal{S}]$ of the Sylvester map in degree $D_0$
	with respect to the standard bases $(X^\beta e_j : |\beta| = D_0 - d, \, 0 \le j \le M)$ of $V_{D_0}$ and $(X^\gamma : |\gamma| = D_0)$ of $K[X]_{D_0}$.
\end{defn}

Using the formula
\begin{equation} \label{stars_bars}
	\dim K[X]_D = \binom{m+D}{m},
\end{equation}
valid for all integers $D$, 
we get that 
\begin{equation} \label{dimVD}
	\dim V_D = (M+1) \binom{m+D-d}{m}.
\end{equation}
Plugging in $D = D_0$, we have $m + D_0 = (m+1)d$ and $m + D_0 - d = md$,
so the Sylvester--Macaulay matrix has 
\[r := \binom{(m+1)d}{m} \text{ rows} \quad\text{and}\quad c := (M+1)\binom{md}{d} \text{ columns.}\]
Necessarily, $r \le c$.
Seeing as $\mathcal{S}(X^\beta e_j)
= \sum_{|\gamma| = D_0} F_{j,\gamma-\beta} X^\gamma$, the nonzero entries in each column of $[\mathcal{S}]$ are precisely the coefficients of $F_j$ for some $j$.

\begin{q}
	Is Macaulay's bound optimal? 
	A necessary condition for $\mathcal{S}(V_D) = K[X]_D$ to hold is that
	\begin{equation} \label{necessary_D}
		\binom{m+D}{m} \le (M + 1) \binom{m+D-d}{m}.
	\end{equation}
	For instance, if $M = m = 1$ then \eqref{necessary_D} is equivalent to $D \ge 2d-1 = D_0$.
	However, if $M = m = 2$ then the least $D$ satisfying \eqref{necessary_D} is asymptotically $\frac{3+\sqrt{3}}{2}d - \frac{3}{2}$ (and equal to 6 when $d = 3$), whereas Macaulay's bound $D_0 = 3d-2$ is larger.
\end{q}

Here are some examples of Sylvester--Macaulay matrices where the bases have been ordered lexicographically.
Note that since this monomial order respects multiplication\footnote{i.e.,~$X^\alpha \preccurlyeq X^\beta$ implies $X^\alpha X^\gamma \preccurlyeq X^\beta X^\gamma$ for all $\alpha$, $\beta$, $\gamma$.}
the nonzero entries in each column always appear in the same order.

\begin{eg}[$d = 1$]
	There are $m+1$ rows and $M+1$ columns:
	\[
	[\mathcal{S}] =
	\begin{bmatrix}
		&        & F_{j,(1, 0, \ldots, 0)} &        & \\
		& \cdots & \vdots  & \cdots & \\
		&        & F_{j,(0, \ldots, 0, 1)} &        &
	\end{bmatrix}
	\]
	i.e.,~the transpose of $F$.
\end{eg}

\begin{eg}[$m = 1$] 
	There are $2d$ rows and $(M+1)d$ columns:
	\[
	[\mathcal{S}] =
	\begin{bmatrix}
		& 
		\cdots\!
		& 
		\begin{array}{ccc}
			F_{j,(d,0)}        &        & \\
			\multirow[c]{1.6}{*}{\vdots} & \ddots & \\
			&        & F_{j,(d,0)} \\
			F_{j,(0,d)}        &        & \multirow[c]{2.4}{*}{\vdots} \\
			& \ddots & \\
			&        & F_{j,(0,d)}
		\end{array} 
		& 
		\!\cdots 
		& 
	\end{bmatrix}
	\]
\end{eg}

\begin{eg}[$m = d = 2$]
	There are 15 rows and $6(M+1)$ columns:
	\[
	[\mathcal{S}] =
	\begin{bmatrix}
		& & F_{j,(2,0,0)} &               &               &               &               &               & & \\
		& & F_{j,(1,1,0)} & F_{j,(2,0,0)} &               &               &               &               & & \\
		& & F_{j,(1,0,1)} &               & F_{j,(2,0,0)} &               &               &               & & \\
		& & F_{j,(0,2,0)} & F_{j,(1,1,0)} &               & F_{j,(2,0,0)} &               &               & & \\
		& & F_{j,(0,1,1)} & F_{j,(1,0,1)} & F_{j,(1,1,0)} &               & F_{j,(2,0,0)} &               & & \\
		& & F_{j,(0,0,2)} &               & F_{j,(1,0,1)} &               &               & F_{j,(2,0,0)} & & \\
		& &               & F_{j,(0,2,0)} &               & F_{j,(1,1,0)} &               &               & & \\
		& \cdots &        & F_{j,(0,1,1)} & F_{j,(0,2,0)} & F_{j,(1,0,1)} & F_{j,(1,1,0)} &               & \cdots & \\ 
		& &               & F_{j,(0,0,2)} & F_{j,(0,1,1)} &               & F_{j,(1,0,1)} & F_{j,(1,1,0)} & & \\
		& &               &               & F_{j,(0,0,2)} &               &               & F_{j,(1,0,1)} & & \\ 
		& &               &               &               & F_{j,(0,2,0)} &               &               & & \\
		& &               &               &               & F_{j,(0,1,1)} & F_{j,(0,2,0)} &               & & \\
		& &               &               &               & F_{j,(0,0,2)} & F_{j,(0,1,1)} & F_{j,(0,2,0)} & & \\
		& &               &               &               &               & F_{j,(0,0,2)} & F_{j,(0,1,1)} & & \\
		& &               &               &               &               &               & F_{j,(0,0,2)} & & 
	\end{bmatrix}
	\]
\end{eg}

In \cite[p.~3--4]{Macaulay_book}, 
Macaulay defined the resultant 
of $n$ homogeneous polynomials in $n$ variables 
to be ``the \textsc{h.c.f.}~of the determinants of the above array'', that is, the g.c.d.~of the maximal minors of the Sylvester--Macaulay matrix (transposed).
Taking our cue, we make the following definition.

\begin{defn}
	Let $f : \P^m \to \P^M$ be a morphism of degree $d$ defined over $K$.
	Then the \emph{valuation of the resultant} of $f$ 
	is the integer
	\[v(\Res f) := \min_\Delta v(\Delta) - \binom{(m+1)d}{m} v(F)\]
	where $\Delta$ ranges over all maximal minors of the Sylvester--Macaulay matrix associated to any homogeneous lift $F$ of $f$.
\end{defn}

This quantity is lift-independent, 
because replacing $F$ by $\lambda F$ scales each minor $\Delta$ by $\lambda^r$ where $r = \binom{(m+1)d}{m}$ is the number of rows; 
but 
\[v(\lambda^r \Delta) - r v(\lambda F) = v(\Delta) - rv(F).\]

Although the notation ``$v(\Res f)$'' suggests the existence of a global object ``$\Res f$'', we only define the latter when $K$ is the field of fractions of a Dedekind domain (cf.~Section \ref{sec:resideal}).

\begin{propn} \label{vRes}
	Let $f$ be as above.
	Then:
	\begin{enumerate}[(i)]
		\item $v(\Res f) \ge 0$;
		\item Any lift $F$ of $f$ admits a pseudoinverse $G$ with $v(F) + v(G) \ge -v(\Res f)$; and
		\item $f$ has good reduction (mod $v$) if and only if $v(\Res f) = 0$.
	\end{enumerate} 
\end{propn}

\begin{proof}
	Let $F$ be a lift of $f$.
	Fix a maximal minor $\Delta$ of $[\mathcal{S}]$ such that $v(\Res f) = v(\Delta) - rv(F)$.
	\begin{enumerate}[(i)]
		\item By the Leibniz formula, 
		$\Delta$ is a polynomial of degree $r$ with integer coefficients in the entries of $[\mathcal{S}]$,
		each of which is either $0$ or a coefficient $F_{j,\alpha}$ of $F$.
		It follows that $v(\Delta) \ge rv(F)$.
		\item It suffices to construct a pseudoinverse $G$ of $F$ using the submatrix of $[\mathcal{S}]$ corresponding to $\Delta$, as in the proof of Proposition \ref{good_red}; this yields $v(G) \ge (r-1)v(F) - v(\Delta)$.
		\item We may assume $v(F) = 0$.
		From part (ii), if $v(\Res f) = 0$ then $F$ has a $v$-integral pseudoinverse, so by Proposition \ref{good_red}, $f$ has good reduction.
		Conversely, if $f$ has good reduction, 
		then $[\Tilde{\mathcal{S}}] = \Tilde{[\mathcal{S}]}$ has a nonzero maximal minor $\Tilde{\Delta'}$, 
		whence $v(\Res f) \le v(\Delta') = 0$.
		By part (i), $v(\Res f) = 0$. \qedhere 
	\end{enumerate}
\end{proof}

\begin{rmk}
	Beware that when $m = M > 1$, our definition of $v(\Res f)$ need not coincide with the valuation of the homogeneous resultant: 
	we (merely) have 
	\begin{equation} \label{myRes}
		v(\Res f) \ge v(\Res_{d,\ldots,d}(F_0, \ldots, F_m)) - (m+1)d^m v(F).
	\end{equation}
	To see this, recall that $\Res_{d,\ldots,d}$ is an irreducible homogeneous polynomial
	of degree $(m+1)d^m$ with integer coefficients
	that vanishes at $F$ if and only if the $F_j$'s have a common zero \cite[Theorems 2.3 and 3.1]{Cox}.
	By adapting \cite[Proposition 4.7]{Cox}, 
	one can show that when the $F_{j,\alpha}$ are generic, 
	then there exist homogeneous polynomials $E_\Delta \in \Z[F_{j,\alpha}]$
	(called \emph{extraneous factors})
	such that 
	\[\Delta = {\Res_{d,\ldots,d}} \cdot E_\Delta\]
	for each maximal minor $\Delta$ of the Sylvester--Macaulay matrix.
	As elements of the universal coefficient ring, the extraneous factors are in fact coprime, but this is irrelevant: specializing to a particular choice of $F$ may introduce common divisors,\footnote{or, in the terminology of Section \ref{sec:excess}, ``excess divisors''} so we only obtain the lower bound \eqref{myRes}.
	For example,
	if
	\begin{align*}
		F_0 &= X^2 - XY - Y^2 + XZ + YZ - Z^2 \\
		F_1 &= X^2 + XY + Y^2 + XZ - YZ + Z^2 \\
		F_2 &= X^2 + XY - Y^2 + XZ - YZ + Z^2
	\end{align*}
	then (according to \texttt{Sage} and \texttt{Macaulay2}) $v_2(\Res f) = 9$ while $v_2(\Res_{2,2,2}(F))= 8$.
	It would be interesting to bound the gap in \eqref{myRes}.
	
	Of course, \eqref{myRes} is an equality when $m = M = 1$ (because there is only one minor); equality also holds when either side vanishes.
	Were we to define $v(\Res f)$ using the right-hand side of \eqref{myRes}, 
	then Proposition \ref{vRes} would remain true: parts (i) and (iii) are immediate, and part (ii) can be gleaned from \cite{Macaulay_article} according to \cite[Lemma 4]{Ingram_macaulayres}.
	We chose our definition of $v(\Res f)$ for its uniformity in $M$, its theoretical simplicity, and its direct applicability to finding pseudoinverses.
\end{rmk}

\subsection{The resultant ideal} \label{sec:resideal}

As in \cite[p.~101]{MortonSilverman}
we may package the various valuations of the resultant of $f$ into a single object.

\begin{defn} 
	Let $f : \P^m \to \P^M$ be a morphism defined over 
	the field of fractions $K$ of a Dedekind domain $R$.
	The \emph{resultant ideal} of $f$ is 
	\[
	\Res f := \prod_{v \nmid \infty} \mf p_v^{v(\Res f)}.
	\]
\end{defn}

The following estimate, generalizing \cite[Lemma 5.4]{Olechnowicz}, bounds the norm of the resultant ideal in terms of the height of $f$.

\begin{lem} \label{Nm_res_bound}
	Let $f : \P^m \to \P^M$ be a morphism
	of degree $d \ge 1$
	defined over a number field $K$ of degree $n$.
	Let $F$ be a lift of $f$
	and let $N_j = \#\{\alpha : F_{j,\alpha} \ne 0\}$.
	Put 
	\[
	r := \binom{(m+1)d}{m}, 
	\quad 
	b := \binom{md}{m}, 
	\quad \text{and} \quad 
	s := \left\lceil r/b \right\rceil
	\]
	and let $N$ be the product of the $s$ largest $N_j$'s.
	Then 
	\[\Nm \Res f 
	\le 
	N^{nb/2} H(f)^{nr}
	\le
	\Big(
	\sqrt{\textstyle{\binom{m+d}{m}}} H(f) 
	\Big)^{n(m+1)\binom{md}{m}}.
	\]
\end{lem}

\begin{proof}
	Recall that the Sylvester--Macaulay matrix $[\mathcal{S}_F]$ of $F$ has $r$ rows and $(M+1)b$ columns.
	Let $\Delta$ range over all nonzero minors of 
	$[\mathcal{S}_F]$.
	By definition, 
	\begin{equation} \label{mybound_1}
		\Res f
		= \prod_{v \nmid \infty} \mf p_v^{\min_\Delta \! v(\Delta) - rv(F)}
		= \langle F \rangle^{-r} \prod_{v \nmid \infty} \mf p_v^{\min_\Delta \! v(\Delta)}.
	\end{equation}
	By multiplicativity of the norm,
	\begin{equation} \label{mybound_2} 
		\Nm \prod_{v \nmid \infty} \mf p_v^{\min_\Delta \! v(\Delta)}
		\le 
		\min_\Delta \Nm \prod_{v \nmid \infty} \mf p_v^{v(\Delta)} 
		= 
		\min_\Delta \Nm \Delta \mathcal{O}_K.
	\end{equation}
	Since $\Delta \ne 0$, 
	the product formula implies 
	\begin{equation} \label{mybound_3}
		\Nm \Delta \mathcal{O}_K
		= |N_{K/\Q}(\Delta)|
		= \prod_{\sigma : K \into \C} |\sigma(\Delta)|.
	\end{equation}
	Recall that the columns of $[\mathcal{S}_F]$ are indexed by pairs $(\beta, j)$ 
	where $\beta$ is a multiindex of degree $m(d-1)$ 
	and $j$ is an integer between $0$ and $M$.
	For any set $I$ of indices, 
	write $[\mathcal{S}_F]_I$ for the submatrix of $[\mathcal{S}_F]$ obtained by deleting the columns \emph{not} indexed by $I$.
	Also put
	\[
	\nu_j(I) := \#\{\beta : (\beta, j) \in I\}
	\]
	and note that 
	\begin{equation} \label{sum_nu}
		0 \le \nu_0(I), \ldots, \nu_M(I) \le b 
		\quad\text{and}\quad 
		\nu_0(I) + \ldots + \nu_M(I) = \# I.
	\end{equation}
	
	Let $I$ be the set of columns of the submatrix of $[\mathcal{S}_F]$ whose determinant is $\Delta$.
	By naturality of the determinant
	and Hadamard's inequality,
	\begin{equation} \label{mybound_4}
		|\sigma(\Delta)|
		=
		|\sigma(\det [\mathcal{S}_F]_I)|
		=
		|{\det}([\mathcal{S}_{\sigma(F)}]_I)|
		\le 
		\prod_{(\beta, j) \in I}
		\lVert [\mathcal{S}_{\sigma(F)}]_{\{(\beta, j)\}} \rVert_2.
	\end{equation}
	By hypothesis, the 2-norm of the $(\beta, j)$-column
	satisfies
	\begin{equation} \label{mybound_5}
		\lVert [\mathcal{S}_{\sigma(F)}]_{\{(\beta, j)\}} \rVert_2
		= 
		\sqrt{
			\sum_{\substack{|\alpha|=d \\ F_{j,\alpha} \ne 0}}
			|\sigma(F_{j,\alpha})|^2
		}
		\le
		\sqrt{N_j} \, |F|_\sigma.
	\end{equation}
	Combining \eqref{mybound_4} and \eqref{mybound_5}  
	and using \eqref{sum_nu} yields
	\begin{equation} \label{mybound_6}
		|\sigma(\Delta)| 
		\le 
		\prod_{(\beta, j) \in I}
		\big( \sqrt{N_j} \, |F|_\sigma \big)
		= |F|_\sigma^r \prod_{j=0}^M \sqrt{N_j}^{\nu_j(I)}.
	\end{equation}
	To obtain an upper bound independent of $\Delta$,
	it suffices to maximize 
	\begin{equation} \label{objective}
		\log 
		\prod_{j=0}^M \sqrt{N_j}^{\nu_j}
		=
		\frac{1}{2} \sum_{j=0}^M \nu_j \log N_j
	\end{equation}
	where the $\nu_j$ are integers between 0 and $b$ 
	whose sum is $r$.
	(Beware that we cannot freely \emph{minimize} the objective function \eqref{objective}, because we do not know whether the corresponding minor would be nonzero.)
	Relabelling if necessary, 
	we may assume $N_0 \ge N_1 \ge \ldots \ge N_M$. 
	Then clearly the maximum of \eqref{objective} 
	is attained when the initial $\nu_j$'s are as large as possible,
	say 
	\[
	\nu_0 = \ldots = \nu_{k-1} = b \ge \nu_k > 0 = \nu_{k+1} = \ldots = \nu_M
	\]
	for some $k$.
	With these choices,
	\[
	r = \nu_0 + \ldots + \nu_M
	= kb + \nu_k \in (kb, (k+1)b]
	\]
	so $k+1 = \lceil r/b \rceil = s$;
	and therefore
	\begin{equation} \label{mybound_7}
		\frac{1}{2} \sum_{j=0}^M \nu_j \log N_j
		\le
		\frac{b}{2} \sum_{j=0}^k \log N_j
		= \frac{b}{2} \log N.
	\end{equation}
	Inserting \eqref{mybound_7} into \eqref{mybound_6}
	gives
	\[|\sigma(\Delta)| \le N^{b/2} |F|_\sigma^r;\]
	putting this into \eqref{mybound_3}
	gives
	\[\Nm \Delta \mathcal{O}_K \le N^{nb/2} \prod_{\sigma : K \into \C} |F|_\sigma^r;\]
	and so it follows 
	from \eqref{mybound_1}
	and \eqref{mybound_2}
	that 
	\[
	\Nm \Res f 
	\le 
	\frac{N^{nb/2}}{\Nm {\langle F \rangle}^r} \prod_{\sigma : K \into \C} |F|_\sigma^r 
	= N^{nb/2} H(f)^{nr}
	\]
	as desired.
	The subsequent inequality
	follows from the uniform estimates 
	\[
	N_j \le \binom{m+d}{m}, 
	\quad s \le m+1, \quad\text{and}\quad r \le sb. 
	\qedhere 
	\]
\end{proof}

\section{Primitivity} \label{sec:primitivity}

We now set out to define a hands-on notion of projective space over arbitrary rings,
with the ultimate goal of establishing the existence 
of a well-behaved reduction map 
\[
\pi : K^{m+1} - 0^{m+1} \to \P^m(R/I)
\] 
for any nonzero ideal $I$ of a Dedekind domain $R$ with field of fractions $K$.
In order to be as explicit as possible 
(and to facilitate computer implementation in future work)
we build up the theory from first principles.
Many of these results will have direct application to the proof of Theorem \ref{thm:main}.

\subsection{Fundamental domains}

Let $G$ be a group acting on a set $X$.
Then $G$ acts on any $G$-invariant subset $S$ of $X$; the \emph{orbit-space} is the quotient $S/G := \{Gx : x \in S\}$ by the induced action. (We shall not use the notation $S/G$ unless $S$ is $G$-invariant.)

Let $D \subseteq X$ be a subset 
and let $H \leq G$ be a subgroup.
Following Schanuel,
we say $D$ is \emph{fundamental (mod $H$) for $G$ acting on $X$} if 
\begin{enumerate}
	\item $HD = D$
	\item $GD = X$
	\item $gD \cap D = \varnothing$ for $g \not \in H$
\end{enumerate}
When $H$ is trivial, we just say $D$ is \emph{fundamental}.
In this case, Condition (1) holds vacuously.

The significance of these ``generalized'' fundamental domains is that they turn ``big groups acting on big sets'' into ``small groups acting on small sets''. 
The following Proposition makes this precise.

\begin{propn} \label{enlargement}
	Let $D$ be fundamental (mod $H$) for $G$ acting on $X$. 
	Then the natural ``enlargement'' map
	\begin{align*}
		D/H &\to X/G \\
		Hx &\mapsto Gx
	\end{align*}
	is a bijection of orbit-spaces.
	In particular, $|X/G| = |D/H|$.
\end{propn}
\begin{proof}
	Property (1) says $D$ is $H$-invariant, so $D/H = \{Hx : x \in D\}$ is an orbit-space.
	The map in question is well-defined because $H$ is a nonempty subset of $G$.
	Property (2) is equivalent to surjectivity,
	and property (3) implies injectivity.
\end{proof}

For the reader's convenience, we quote a portion of \cite[Lemma 1]{Schanuel},
which produces fundamental domains under pullback.

\begin{lem} \label{schanuel_lemma_1}
	Let $H \leq G \acts X, Y$ and let $\varphi : X \to Y$ be $G$-equivariant. 
	If $D$ is fundamental (mod $H$) for $G \acts Y$, 
	then $\varphi^{-1}(D)$ is fundamental (mod $H$) for $G \acts X$.
\end{lem}
\begin{proof}[Proof (Schanuel).]
	Immediate from the definitions.
\end{proof}

\begin{rmk}
	The hypothesis that $\varphi$ be equivariant is essential.
\end{rmk}

We also adopt a construction used by Schanuel, namely that of group actions induced via homomorphisms.
Let $G$ be a group and let $A$ be a semigroup, 
and let $\alpha : G \to A$ be a semigroup homomorphism such that $\alpha(1)$ is a left identity of $A$.
In this context, two groups naturally act on $A$:
\begin{itemize}
	\item $G$ acts on $A$ by $g \cdot a = \alpha(g)a$ (the \emph{induced action});
	\item $\alpha(G)$ acts on $A$ by left multiplication (the \emph{regular action}).
\end{itemize}
The next Lemma relates fundamental domains for these two actions on $A$, and fills a gap in Schanuel's argument.

\begin{lem} \label{schanuel_missing_lemma}
	If $D$ is fundamental (mod~$H$) for $\alpha(G)$ acting on $A$,
	then $D$ is fundamental (mod~$\alpha^{-1}(H)$) for $G$ acting on $A$.
\end{lem}
\begin{proof}
	Immediate from the definitions.
\end{proof}

\begin{rmk}
	Schanuel asserts that if an abelian group $G$ acts on a commutative, unital semigroup $A$ via a homomorphism $\alpha$, then $G$ is ``effective mod $\ker \alpha$'', i.e.,~$g \cdot a = a$ implies $g \in \ker \alpha$.\footnote{A better term would be \emph{free} mod $\ker \alpha$.} The claim is false: if $A$ is any unital ring under multiplication, then $g \cdot 0 = 0$ for all $g$ in $G$ regardless of $\alpha : G \to A$.
	On the other hand, if we interpret ``effective mod $\ker \alpha$'' to mean 
	\[g \cdot a = a \text{ for \emph{all} $a$ in $A$} \implies g \in \ker \alpha\]
	then the claim is true (take $a = 1$); however, this alternate interpretation renders Lemma 1 false unless $\Phi$ is surjective. 
	Either way, Schanuel apparently never \emph{uses} this claim (nor, for that matter, this notion of relative effectivity).
\end{rmk}

\subsection{Covariant maps}
\begin{lem} \label{lem:covariant_descent}
	Let $X$ and $Y$ be sets, let $\varphi : X \to Y$ be a function, and let $G$ and $H$ be groups acting on $X$ and $Y$, respectively. 
	TFAE:
	\begin{enumerate}[(a)]
		\item (covariance) for all $x$ in $X$ and $g$ in $G$ there exists $h$ in $H$ such that $\varphi(gx) = h\varphi(x)$;
		\item (descent) there exists a function $\Tilde{\varphi} : X/G \to Y/H$ such that $\Tilde{\varphi}(Gx) = H\varphi(x)$ for all $Gx$ in $X/G$, i.e.,~making the obvious square
\[
\begin{tikzcd}[row sep=large]
	X \arrow{r}{\varphi} \arrow{d}{} & Y \arrow{d}{} \\
	X/G \arrow[swap]{r}{\Tilde\varphi} & Y/H
\end{tikzcd}
\]
commute.
	\end{enumerate}
\end{lem}

\begin{proof}
	If $\varphi$ is covariant, 
	then the function $X \ni x \mapsto H\varphi(x) \in Y/H$ is $G$-invariant (as $H\varphi(gx) = Hh\varphi(x) = H\varphi(x)$), so it descends to $X/G$.
	Conversely, if $\Tilde{\varphi}$ exists, then for all $x$ in $X$ and $g$ in $G$ we have 
	$H\varphi(gx) = \Tilde{\varphi}(Ggx) = \Tilde{\varphi}(Gx) = H\varphi(x)$,
	so $\varphi(gx) = h\varphi(x)$ for some $h$ in $H$.
\end{proof}

\begin{lem}[Fibres] \label{lem:fibres}
	Let $G \acts X$ and $H \acts Y$. 
	Let $\varphi : X \to Y$ be covariant.
	Then
	\[\Tilde{\varphi}^{-1}(Hy) = \varphi^{-1}(Hy)/G\]
	for all $Hy \in Y/H$.
\end{lem} 

\begin{proof}
	Implicit is the claim that the set $\varphi^{-1}(Hy)$ is $G$-invariant---which is true because $\varphi$ is covariant.
	Unraveling definitions,
	\begin{gather*}
		\Tilde{\varphi}^{-1}(Hy) 
		= \{Gx : \Tilde{\varphi}(Gx) = Hy\}
		= \{Gx : H\varphi(x) = Hy\} \\ 
		= \{Gx : \varphi(x) \in Hy\} 
		= \{Gx : x \in \varphi^{-1}(Hy)\} 
		= \varphi^{-1}(Hy)/G.
		\qedhere 
	\end{gather*}
\end{proof}

\subsection{Primitive points}
Let $R$ be a commutative unital ring 
and let $I$ be an ideal of $R$.
A tuple $a$ of elements of $R$ shall be called \emph{primitive modulo $I$} if the ideal $\langle a \rangle$ generated by the components of $a$ is coprime to $I$.
For each non-negative integer $k$ 
let 
\[
R^k_{\prim/I} := \{a \in R^k : \langle a \rangle + I = R\}
\]
and put $R^k_\prim := R^k_{\prim,\langle 0\rangle}$.
Elements of $R^k_\prim$ are called \emph{primitive}.
Note that 
\begin{itemize} 
	\item $R^k_{\prim/R} = R^k$,
	\item $R^0_{\prim/I} = \varnothing$ unless $I = R$,
	\item $R^1_\prim = R^\times$;
	\item moreover, if $I \leq J$ then $R^k_{\prim/I} \subseteq R^k_{\prim/J}$.
\end{itemize}

\begin{lem} \label{prim_mod_I}
	Let $I^k = I \times \ldots \times I$, the $k$\textsuperscript{th} cartesian power of $I$, act on $R^k$ by translation. Then the set $R^k_{\prim/I}$ is $I^k$-invariant, 
	and 
	\begin{align*}
		R^k_{\prim/I}/I^k = (R/I)^k_\prim.
	\end{align*}
\end{lem}
\begin{proof}
	Invariance: if $r_1 a_1 + \ldots + r_k a_k + x = 1$ for some $x \in I$ and $r_i \in R$, 
	and if $(x_1, \ldots, x_k) \in I^k$, 
	then putting $x' = x - r_1 x_1 - \ldots - r_k x_k$ we have $x' \in I$ 
	and $r_1(x_1 + a_1) + \ldots + r_k (x_k + a_k) + x' = 1$ by construction.
	Equality: follows from the fact that $\Bar{r_1} \Bar{a_1} + \ldots + \Bar{r_k} \Bar{a_k} = \Bar{1}$ in $R/I$ if and only if $r_1 a_1 + \ldots + r_k a_k - 1 \in I$.
\end{proof}

\begin{lem} \label{lem:R_prim_units}
	Let $I$ be an ideal of $R$.
	We have
	\[ R^k_{\prim/I} \supseteq \{a \in R^k : a_i \in R^\times \text{ for some } i\}.\]
	Equality holds if $R$ is local and $I$ is proper.
\end{lem} 
\begin{proof}
	Let $a \in R^k$.
	Certainly if some $a_i \in R^\times$ then $\langle a \rangle = R$, so $a \in R^k_\prim \subseteq R^k_{\prim/I}$.
	Conversely if no $a_i \in R^\times$ and if $R$ is local, then $\langle a \rangle$ is proper;
	so if $I$ is also proper then $\langle a \rangle + I \ne R$, meaning $a \not \in R^k_{\prim/I}$.
\end{proof}

\begin{rmk}
	$(\Bar 2, \Bar 3) \in (\Z/6\Z)^2_{\prim}$ yet $\Bar 2, \Bar 3 \not \in (\Z/6\Z)^\times$.
\end{rmk}

\begin{lem} \label{prim_prod}
	Let $I \leq R$ and $J \leq S$ be ideals.
	Then $(R \times S)^k_{\prim/I \times J} = R^k_{\prim/I} \times S^k_{\prim/J}$.
\end{lem}
\begin{proof}
	A point $((a_1, b_1), \ldots, (a_k, b_k))$ in $(R \times S)^k$ is primitive modulo $I \times J$ 
	if and only if 
	there exist $(r_i, s_i) \in R \times S$ 
	and $(x, y) \in I \times J$ such that 
	\[
(r_1, s_1)(a_1, b_1) + \ldots + (r_k, s_k)(a_k, b_k) + (x, y) = (1_R, 1_S)
	\]
	if and only if there exist $r_i$ in $R$ and $x$ in $I$, as well as $s_i$ in $S$ and $y$ in $J$, such that 
	\[
r_1 a_1 + \ldots + r_k a_k + x = 1_R \quad\text{and}\quad s_1 b_1 + \ldots + s_k b_k + y = 1_S.
\qedhere
\]
\end{proof}

The next Lemma says that ring homomorphisms preserve primitivity.

\begin{lem} \label{prim_phi}
	Let $\varphi : R \to S$ be a ring homomorphism and let $I \leq R$ and $J \leq S$ be ideals. 
	Then the induced map 
	\begin{align*}
		\varphi_* : R^k &\to S^k \\
		(a_1, \ldots, a_k) &\mapsto (\varphi(a_1), \ldots, \varphi(a_k))
	\end{align*}
	sends $R^k_{\prim/I}$ to $S^k_{\prim/\langle \varphi(I) \rangle}$.
	Moreover, if $\varphi$ is surjective, 
	then $\varphi_*^{-1}(S^k_{\prim/J}) \subseteq R^k_{\prim/\varphi^{-1}(J)}$.
\end{lem}
\begin{proof}
	Applying $\varphi$ to any relation of the form 
	\begin{equation} \label{prim_mod_eg}
	r_1 a_1 + \ldots + r_k a_k + x = 1_R
	\end{equation}
	shows that if $a = (a_1, \ldots, a_k) \in R^k_{\prim/I}$ then $\varphi_*(a)
	\in S^k_{\prim/\langle \varphi(I) \rangle}$.
Now if $a \in R^k$ with $\varphi_*(a) \in S^k_{\prim/J}$ 
then there exist $s_i$ in $S$ 
and $y$ in $J$ 
such that 
\[s_1 \varphi(a_1) + \ldots + s_k \varphi(a_k) + y = 1_S.\]
Assuming $\varphi$ is surjective
yields $r_i$ in $R$ such that $\varphi(r_i) = s_i$ for all $i$;
let $x = 1_R - (r_1 a_1 + \ldots + r_k a_k)$.
By construction, \eqref{prim_mod_eg} holds 
and $\varphi(x) = y \in J$.
Thus $a \in R^k_{\prim/\varphi^{-1}(J)}$.
\end{proof}

\subsection{Jordan totients}

\begin{defn}
Let $I$ be an ideal of a commutative unital ring $R$ and let $k$ be a non-negative integer.
	The cardinal $J_{R,k}(I) := |(R/I)^k_\prim|$
	shall be called
	the \emph{$k$\textsuperscript{th} Jordan totient of $I$}.
\end{defn}

\begin{eg}
	By Lemma \ref{prim_mod_I},
	\[J_{\Z,k}(n\Z) = \left|\{(a_1, \ldots, a_k) : 1 \le a_i \le n, \, \gcd(a_1, \ldots, a_k, n) = 1\}\right|\]
	is number of $k$-tuples of integers between 1 and $n$ which, taken together, are coprime to $n$.
	This number is given by Jordan's totient function $J_k(n)$ (so-called because it reduces to Euler's totient function $\varphi(n)$ when $k = 1$).
\end{eg}

The following Lemma summarizes some useful properties of the Jordan totient function.

\begin{lem} \label{properties_J}
\hfill 
\begin{enumerate}[(i)]
	\item $J_{R,k}(I) \le [R : I]^k$ for all $k$, with equality if $I = R$.
	\item $J_{R,k}(I) = J_{R/I',k}(I/I')$ for all ideals $I' \leq I \leq R$.
\end{enumerate}
\end{lem} 
\begin{proof}
Part (i) is immediate from the inclusion $(R/I)^k_\prim \subseteq (R/I)^k$. Note that $R/R$ is the zero ring, whose sole $k$-tuple is primitive. 
Part (ii) follows from Lemma \ref{prim_phi} applied to the isomorphism $(R/I') / (I/I') \cong R/I$.
\end{proof}

\begin{lem} \label{J_is_mult}
$J_{R,k}$ is multiplicative.
\end{lem}
\begin{proof} 
Suppose $I + I' = R$. 
Then
\begin{align*}
J_{R,k}(II') 
&= |(R/II')^k_\prim| 
	& \textrm{[definition]} \\
&= |(R/I \times R/I')^k_\prim|
	& \textrm{[CRT + Lemma \ref{prim_phi}]} \\
&= |(R/I)^k_\prim \times (R/I')^k_\prim|
	& \textrm{[Lemma \ref{prim_prod}]} \\
&= |(R/I)^k_\prim| \, |(R/I')^k_\prim|
	& \textrm{[cardinal arithmetic]} \\
&= J_{R,k}(I)  J_{R,k}(I'). 
	& \textrm{[definition]} & \qedhere 
\end{align*} 
\end{proof}

\begin{propn} \label{J_local}
	Let $R$ be a local ring with maximal ideal $\mf m$ and let $I \leq R$ be a proper ideal of finite index.
	Then 
	\[J_{R,k}(I) = [R : I]^k \Big(1 - \frac{1}{[R : \mf m]^k}\Big).\]
\end{propn}
\begin{proof}
We have
$J_{R,k}(I) = |(R/I)^k_\prim| = |R^k_{\prim/I}/I^k|$.
Since $R$ is local and $I$ is proper, 
Lemma \ref{lem:R_prim_units} implies 
\[R^k_{\prim/I} = \{a \in R^k : a_i \in R^\times \text{ for some }i\} = R^k \setminus \mf m^k,
\] 
i.e.,~the complement of $\mf m^k$ in $R^k$.
By Lemma \ref{prim_mod_I}, $\mf m^k$ is $I^k$-invariant.
Hence $R^k/I^k = R^k_{\prim/I}/I^k \sqcup \mf m^k/I^k$ as orbit-spaces.
Taking cardinalities,
\[
[R^k : I^k] = J_{R,k}(I) + [\mf m^k : I^k].
\]
Now the natural map $R^k \to (R/I)^k$ has kernel $I^k$ 
so $[R^k : I^k] = [R : I]^k$. 
Restricting to $\mf m$ (which contains $I$) we likewise obtain $[\mf m^k : I^k] = [\mf m : I]^k$.
Thus
\[
	[R : I]^k = J_{R,k}(I) + [\mf m : I]^k. 
	\tag{*}
\]
Since $I$ has finite index, it follows from the tower theorem---$[R : I] = [R : \mf m][\mf m : I]$---that
\[J_{R,k}(I) = [R : I]^k - [\mf m : I]^k = [R : I]^k - \frac{[R : I]^k}{[R : \mf m]^k}. \qedhere\]
\end{proof}

\begin{cor} \label{J_Dedekind}
	Let $R$ be a Dedekind domain and let $I$ be a nonzero ideal of finite index. 
	Then 
	\[J_{R,k}(I) = \Nm I^k \prod_{P \mid I} \Big(1 - \frac{1}{\Nm P^k}\Big).\]
\end{cor}
\begin{proof}
By multiplicativity of the Jordan totient (Lemma \ref{J_is_mult}) and of the norm, it suffices to prove the claim for $I = P^e$ where $P$ is a maximal ideal and $e$ is a positive integer.

We claim that $R/P^e$ is a local ring. 
Indeed, the ideals of $R/P^e$
are in one-to-one correspondence 
with ideals of $R$ containing a.k.a.~dividing $P^e$, and hence are all of the form $P^i / P^e$ for some $i = 0, \ldots, e$. 
Thus, the unique maximal ideal of $R/P^e$ is $P/P^e$.
Since each successive quotient $P^{j-1}/P^j$ is a one-dimensional vector space over $R/P$, we have 
\[[R/P^e : P^i/P^e] = [R : P^i] 
= \prod_{j=1}^i [P^{j-1} : P^j] 
= [R : P]^i\]
for all $i = 0, \ldots, e$.
In particular, $P^e/P^e$ is a proper ideal of finite index.
Thus
\begin{align*}
J_{R,k}(P^e) 
&= J_{R/P^e, k}(P^e/P^e) & \text{[Lem.~\ref{properties_J}(ii)]}\\
&= [R/P^e : P^e/P^e]^k 
\Big( 
1 - \frac{1}{[R/P^e : P/P^e]^k}
\Big) & \text{[Prop.~\ref{J_local}]}  \\
&= [R : P]^{ek}
\Big( 
1 - \frac{1}{[R : P]^k}
\Big) \\
&= \Nm P^{ek} \Big(1 - \frac{1}{\Nm P^k}\Big). && \qedhere
\end{align*}
\end{proof}

\begin{q}
	For which (finite or infinite) rings $R$ is it true that
	\[\sum_{I \leq R} J_{R,k}(I) = |R|^k?\]
\end{q}

\subsection{Na\"ive projective space}

The unit group $R^\times$ acts on $R^k$ by scaling:
\[u(a_1, \ldots, a_k) = (ua_1, \ldots, ua_k).\]
Under this action, each of the sets $R^k_{\prim/I}$ is $R^\times$-invariant, as $\langle ua \rangle = u\langle a \rangle = \langle a \rangle$ for all units $u$.
More importantly:

\begin{lem} \label{lem:unit_group_acts_freely}
The action of $R^\times$ on $R^k_{\prim}$ is free.
\end{lem} 
\begin{proof}
The key is that 
$\Stab_{R^\times}(a) = R^\times \cap (1 + \Ann_R(a))$
for all $a$ in $R^k$, as $ua = a$ if and only if $(u-1)a_i = 0$ for all $i$.
If $a$ is primitive mod $I$, then 
$\Ann_R(a) = \Ann_R(a)(\langle a \rangle + I) = \Ann_R(a)I$.
The claim follows on taking $I = 0$.
\end{proof}

\begin{rmk}
It's entirely possible that the ``free locus'' of $R^\times$, i.e.,~the maximal subset of $R^k$ on which $R^\times$ acts freely, is a strict superset of $R^k_\prim$. 
This is the case, for instance, with $R = \Z/q\Z$: the free locus is $(\Z/q\Z)^k_{\prim/\langle \Bar{q/2}\rangle}$ when $q$ is singly even (and only then).
\end{rmk}

\begin{defn}
Let $R$ be a ring and let $m \ge -1$. 
The \emph{(na\"ive) projective $m$-space over $R$} is the orbit-space \[\P^m(R) := R^{m+1}_\prim/R^\times.\]
The image in $\P^m(R)$ of a primitive tuple $x = (x_0, \ldots, x_m)$ is denoted $[x] = [x_0 : \ldots : x_m]$.
\end{defn}

\begin{rmk} \label{rmk:proj}
For an arbitrary scheme $X$, 
the correct definition of $\P^m(X)$ is as the set of $X$-valued points of $\operatorname{Proj} \Z[X_0, \ldots, X_m]$,
which are given by pairs $(\mathcal{L}, (s_0, \ldots, s_m))$ where 
$\mathcal{L}$ is an invertible $\mathcal{O}_X$-module 
and $s_0, \ldots, s_m$ are global sections of $\mathcal{L}$, up to equivalence \cite{Stacks_01NE}.
When $X = \operatorname{Spec}(R)$ is affine, 
the points of $\P^m(\operatorname{Spec} R)$ a.k.a.~$\P^m(R)$ correspond to $R$-submodules of $R^{m+1}$ which are direct summands of rank one \cite{MSE_Georges}.
This agrees with our na\"ive definition when $R$ is a product of finitely many PIDs and local rings---the only case we shall consider---but disagrees when, say, $R$ is a Dedekind domain with nontrivial class group.\footnote{In this case, the ``correct'' $\P^m(R)$ coincides with $\P^m(K)$ where $K = \operatorname{Frac}(R)$, 
whereas our ``na\"ive'' $\P^m(R)$ comprises just those points in $\P^m(K)$ whose coordinates generate a principal ideal.}
\end{rmk}

\begin{propn}
Let $R$ be a Dedekind domain and let $I$ be a nonzero ideal of finite index.
Then 
\[\#\P^m(R/I) 
= \frac{J_{R,m+1}(I)}{J_{R,1}(I)} 
= \Nm I^m \prod_{P \mid I} \Big(1 + \frac{1}{\Nm P} + \ldots + \frac{1}{\Nm P^m} \Big).
\]
\end{propn}
\begin{proof}
Immediate from Lemma \ref{lem:unit_group_acts_freely} and Corollary \ref{J_Dedekind}.
\end{proof}

\begin{rmk}
	The formula for $\# \P^m$ reflects the ``geometric'' decomposition 
	\[\P^m = \A^m \sqcup \A^{m-1} \sqcup \ldots \sqcup \A^0.\]
\end{rmk}

For reference, Table \ref{table:PmZq} lists the cardinality of $\P^m(\Z/q\Z)$ for small values of $m$ and $q$, computed as quotients of Jordan totients \cite[Sequences 10, 7434, 59376--8, and 69091]{OEIS}.

\begin{table}
\centering
\small 
\renewcommand{\arraystretch}{1.2}
\begin{tabular}{r|rrrrr}
	$q$ & $\#\P^1(\Z/q\Z)$ & $\#\P^2(\Z/q\Z)$ & $\#\P^3(\Z/q\Z)$ & $\#\P^4(\Z/q\Z)$ & $\#\P^5(\Z/q\Z)$ \\ \hline 
1 & 1 & 1 & 1 & 1 & 1 \\
2 & 3 & 7 & 15 & 31 & 63 \\
3 & 4 & 13 & 40 & 121 & 364 \\
4 & 6 & 28 & 120 & 496 & 2016 \\
5 & 6 & 31 & 156 & 781 & 3906 \\
6 & 12 & 91 & 600 & 3751 & 22932 \\
7 & 8 & 57 & 400 & 2801 & 19608 \\
8 & 12 & 112 & 960 & 7936 & 64512 \\
9 & 12 & 117 & 1080 & 9801 & 88452 \\
10 & 18 & 217 & 2340 & 24211 & 246078 \\
11 & 12 & 133 & 1464 & 16105 & 177156 \\
12 & 24 & 364 & 4800 & 60016 & 733824 \\
13 & 14 & 183 & 2380 & 30941 & 402234 \\
14 & 24 & 399 & 6000 & 86831 & 1235304 \\
15 & 24 & 403 & 6240 & 94501 & 1421784 \\
16 & 24 & 448 & 7680 & 126976 & 2064384 \\
17 & 18 & 307 & 5220 & 88741 & 1508598 \\
18 & 36 & 819 & 16200 & 303831 & 5572476 \\
19 & 20 & 381 & 7240 & 137561 & 2613660 \\
20 & 36 & 868 & 18720 & 387376 & 7874496
\end{tabular}
\caption{$\#\P^m(\Z/q\Z)$ for $1 \le q \le 20$ and $1 \le m \le 5$.}
\label{table:PmZq}
\end{table}

\begin{propn} \label{Pm_preserves_products}
	$\P^m$ is a product-preserving functor from the category of commutative unital rings to the category of sets.
\end{propn}
\begin{proof}
Let $\varphi : R \to S$ be a ring homomorphism.
Then the induced map $\varphi_*$ (cf.~Lemma \ref{prim_phi}) is covariant, as $\varphi_*(ua) = \varphi(u) \varphi_*(a)$ and $\varphi(R^\times) \subseteq S^\times$.
By Lemma \ref{lem:covariant_descent}, $\varphi_*$ descends to the quotient:
\begin{align*}
\P^m(\varphi) : \P^m(R) &\to \P^m(S) \\
	[a_0 : \ldots : a_m] &\mapsto [\varphi(a_0) : \ldots : \varphi(a_m)]
\end{align*}
By construction, $\P^m(\id_R) = \id_{\P^m(R)}$ and $\P^m(\psi \circ \varphi) = \P^m(\psi) \circ \P^m(\varphi)$.
The claim that $\P^m(R \times S) = \P^m(R) \times \P^m(S)$ is immediate from Lemma \ref{prim_prod}.
\end{proof}

\begin{lem} \label{Pm_constant_degree}
Let $R$ be a finite local ring and let $\varphi : R \to S$ be a surjective ring homomorphism.
Then $\P^m(\varphi)$ has constant degree (i.e.,~every fibre has the same cardinality).
\end{lem}
\begin{proof}
Without loss of generality, $S \ne 0$.
Recall that $\P^m(\varphi)$ is the descent (modulo $R^\times$ and $S^\times$) of the restriction (to $R^{m+1}_\prim$) of the map $\varphi_* : R^{m+1} \to S^{m+1}$ induced by $\varphi$.
With that in mind, 
let $[y] \in \P^m(S)$. 
By Lemma \ref{lem:fibres}, 
\[\P^m(\varphi)^{-1}([y]) = \big( \varphi_*^{-1}(S^\times y) \cap R^{m+1}_\prim \big) / R^\times.\]
But by Lemmas \ref{prim_phi} and \ref{lem:R_prim_units} (noting that $\varphi$ is surjective, $R$ is local, and $\ker\varphi$ is proper),
\[
	\varphi_*^{-1}(S^\times y) \subseteq \varphi_*^{-1}(S^{m+1}_\prim) 
	\subseteq R^{m+1}_{\prim/{\ker\varphi}} 
	= R^{m+1}_\prim 
\]
so in fact 
\[\P^m(\varphi)^{-1}([y]) = \varphi_*^{-1}(S^\times y)/R^\times.\]
Since $R^\times$ acts freely, counting orbits is easy:
\[\# \P^m(\varphi)^{-1}([y]) = \frac{\# \varphi_*^{-1}(S) }{\# R^\times}.\]
Since $S^\times$ acts freely,
\[S^\times y = \bigsqcup_{u \in S^\times} \{uy\}\]
which implies
\[
\# \varphi_*^{-1}(S^\times y)
= \sum_{u \in S^\times} \# \varphi_*^{-1}(uy)
\]
But $\varphi_*$ is surjective (because $\varphi$ is), 
so every fibre of $\varphi_*$ is a coset of $\ker \varphi_* = (\ker \varphi)^{m+1}$.
It follows that 
\begin{equation} \label{eq:fibre_card}
\# \P^m(\varphi)^{-1}([y]) = \frac{\# S^\times (\#{\ker\varphi})^{m+1}}{\# R^\times}
\end{equation}
---independent of $y$, as desired.
\end{proof}

\begin{rmk}
	Of course, as soon as we know $\P^m(\varphi)$ has constant degree, every fibre automatically has cardinality 
	\[\frac{\#\P^m(R)}{\#\P^m(S)} = \frac{\# R^{m+1}_\prim / \# R^\times}{\# S^{m+1}_\prim / \# S^\times}.\]
This is indeed consistent with \eqref{eq:fibre_card}: 
if $\mf m$ is the maximal ideal of $R$,
then $\mf n = \varphi(\mf m)$ is the maximal ideal of $S$;
the isomorphisms $S \cong R/{\ker \varphi}$ and $\mf n \cong \mf m/{\ker \varphi}$
imply $\# R = \# S \cdot \#{\ker\varphi}$
and $\# \mf m = \# \mf n \cdot \#{\ker\varphi}$;
so 
\[
\# R^{m+1}_\prim = \# R^{m+1} - \# \mf m^{m+1} 
= (\# S^{m+1} - \# \mf n^{m+1}) (\#{\ker\varphi})^{m+1}
= \# S^{m+1}_\prim (\#{\ker\varphi})^{m+1}.
\]
\end{rmk}

\subsection{Reduction modulo $I$}

Henceforth, 
$R$ is a Dedekind domain and $K$ is its field of fractions.
The goal is to show that for any nonzero ideal $I$ of $R$ there is a well-defined reduction map $K^{m+1} - 0^{m+1} \to \P^m(R/I)$ given by scaling: $\pi_I(x) = [\Bar{\lambda x}]$ for some $\lambda$.
We start with a fundamental approximation theorem.

\begin{lem} \label{lem:jon}
Let $S$ be a finite set of finite places of $R$.
Then for each function $e : S \to \Z$ 
there exists an element $\lambda$ of $K^\times$ such that 
$v(\lambda) = e(v)$ when $v \in S$ and $v(\lambda) \ge 0$ when $v \not \in S$.
\end{lem}
\begin{proof}
Initially we prove the claim for non-negative exponents $e$.
Let $v_1, \ldots, v_s$ be the places in $S$
and let $\mf p_i$ be the corresponding prime ideals of $R$.
Choose uniformizers $t_i$ in $R$ so that $v_i(t_i) = 1$ for all $i$.
By the Chinese Remainder Theorem, 
there exists $\lambda$ in $R$ such that
\[\lambda \equiv t_i^{e(v_i)} \pmod {\mf p_i^{e(v_i)+1}}\]
for all $i$.
It follows from the nonarchimedean property that $v_i(\lambda) = e(v_i)$ for all $i$.

Next we deduce the claim for general exponents $e$.
Decompose $e = e^+ - e^-$ into its positive and negative parts.
Since $e^-$ is non-negative, 
there exists $\lambda^-$ 
such that $v(\lambda^-) = e^-(v)$ when $v \in S$ and $v(\lambda^-) \ge 0$ when $v \not \in S$.
Define $S' := \{v \not \in S : v(\lambda^-) > 0\}$
and extend $e^+$ to $S'$ by setting $e^+(v) := v(\lambda^-)$.
Since $e^+$ is non-negative, 
there exists $\lambda^+$ 
such that $v(\lambda^+) = e^+(v)$ when $v \in S \cup S'$ and $v(\lambda^+) \ge 0$ for all other $v$.
By construction,
\begin{itemize} 
	\item if $v \in S$ then $v(\lambda^-) = e^-(P)$ 
and $v(\lambda^+) = e^+(P)$;
	\item if $v \in S'$ then $v(\lambda^-) = v(\lambda^+)$; and
	\item if $v \not \in S \cup S'$ then $v(\lambda^-) = 0$ and $v(\lambda^+) \ge 0$.
\end{itemize} 
Picking $\lambda := \lambda^+ / \lambda^-$ settles the claim.
\end{proof}

First application: it is always possible to pick class group representatives that are coprime to a given fixed integral ideal.\footnote{Of course, this is immediate from the (much deeper) fact that every ideal class is represented by infinitely many prime ideals.}

\begin{cor} \label{cor:class_rep}
Let $I \leq R$ be a nonzero integral ideal and let $J \subseteq K$ be a nonzero fractional ideal.
Then there exists $\lambda \in K^\times$ such that $I + \lambda J = R$.
\end{cor}
\begin{proof}
Clearing denominators, we may assume $J$ is integral.
By Lemma \ref{lem:jon},
there exists $\lambda$ in $K$ such that $v(\lambda) = -v(J)$ for each $v \mid I$ and $v(\lambda) \ge 0$ otherwise.
Hence
\[v(I + \lambda J) = \min\{v(I), v(\lambda) + v(J)\} = 0\]
for all $v$. Thus $I + \lambda J = R$.
\end{proof}

Second application: it is always possible to scale points of $K^{m+1} - 0^{m+1}$ to be primitive modulo any fixed integral ideal.

\begin{propn} \label{prop:reduction_mod_I}
Let $I \leq R$ be a nonzero integral ideal.
For each $a \in K^{m+1} - 0^{m+1}$ there exists a scalar $\lambda \in K^\times$ such that $\lambda a \in R^{m+1}_{\prim/I}$. 
Moreover, 
if $\lambda'$ is another such scalar, 
then there exists a unit $\Bar u \in (R/I)^\times$ such that
$\Bar{\lambda'a} = \Bar{u} \Bar{\lambda a}$ in the ring $(R/I)^{m+1}$.
\end{propn}

\begin{proof}
The existence of $\lambda$ follows from Corollary \ref{cor:class_rep} applied to the pair $I$, $\langle a \rangle$.
For uniqueness, 
suppose $\lambda a, \lambda' a \in R^{m+1}_{\prim/I}$ for some $\lambda, \lambda' \in K^\times$.
Using Lemma \ref{lem:jon}, 
choose $x$ in $R$ such that 
	(i) $v(x) = v(\lambda/\lambda')$ for all $v$ where this is positive, and 
	(ii) $v(x) = 0$ for all $v \mid I$.
(This \emph{is} possible, because if $v \mid I$ then $v(\lambda) = v(\lambda')$.)
Then $x \in R^1_{\prim/I}$, 
so $\Bar x \in (R/I)^\times$;  
let $y \in R$ be any representative of $\Bar x^{-1}$
(necessarily, $y \in R^1_{\prim/I}$).

Put $u = xy \lambda'/\lambda$.
Since $v(x) + v(\lambda'/\lambda) \ge 0$ for all $v$, with equality whenever $v(I) > 0$ or $v(\lambda'/\lambda) < 0$, 
we get that $u \in R$.
In fact, $u \in R^1_{\prim/I}$
because $v(x) = v(y) = v(\lambda'/\lambda) = 0$ for all $v \mid I$. 
Thus $\Bar u \in (R/I)^\times$.
Finally, $\Bar u \Bar{\lambda a} = \Bar{\lambda' a}$ because 
every coordinate of
\[u\lambda a - \lambda' a = (xy - 1)\lambda' a\]
belongs to $I$.
\end{proof}

\begin{defn} \label{def:red_I}
	Let $I \leq R$ be a nonzero ideal. 
	The \emph{reduction map} modulo $I$ is the function 
	\begin{align*}
		\pi_I : K^{m+1} - 0^{m+1} &\to \P^m(R/I) \\
		(x_0, \ldots, x_m) &\mapsto [\Bar{\lambda x}]
	\end{align*}
	where $\lambda \in K^\times$ is any scalar such that
$\lambda x \in R^{m+1}_{\prim/I}$ (i.e.,~$v(\lambda) = -v(x)$ for all $v \mid I$).
\end{defn}

\begin{rmk}
	The existence of a map $K^{m+1} \to 0^{m+1} \to \P^m(R/I)$ could also be achieved by first reducing $x$ modulo the prime-power ideals dividing $I$ and then applying the Chinese Remainder Theorem.
	With this approach it is not immediately clear that $\pi_I(x) = [\Bar {\lambda x}]$ for some scalar $\lambda$.
\end{rmk}

\begin{rmk}
	We cannot avoid using projective space, as there is no well-defined scaling map from $K^{m+1} - 0^{m+1}$ to $R^{m+1}_{\prim/I}$, nor even to $(R/I)^{m+1}_\prim$.
	For instance, if $K = \Q$ and $x = (2, 6)$ then $\tfrac{1}{2}x = (1, 3)$ and $\tfrac{3}{2}x = (3, 9)$ both lie in $\Z^2_{\prim/4\Z}$, 
	and remain distinct in $(\Z/4\Z)^2_\prim$, agreeing only modulo units: $(\Bar 1, \Bar 3) = \Bar{-1} \cdot (\Bar 3, \Bar 9)$.
\end{rmk}

Proposition \ref{prop:reduction_mod_I} says that $\pi_I$ is well-defined. 
Note also that $\pi_I$ is $K^\times$-invariant: 
if $\lambda$ works for $x$ (meaning $\lambda x \in R^{m+1}_{\prim/I}$) 
and if $c \in K^\times$, 
then $\lambda c^{-1}$ works for $cx$:
\[\pi_I(cx) = [\Bar{\lambda c^{-1} c x}] = [\Bar{\lambda x}] = \pi_I(x).\]

\section{Excess divisors} \label{sec:excess}

\subsection{Local considerations} 

Let $K$ be a field with discrete valuation $v$.
We begin with an elementary but fundamental ``continuity'' estimate.

\begin{lem} \label{lem:F_cts}
	Let $F_0, \ldots, F_M$ be homogeneous forms of degree $d$ in $m+1$ variables and let $x, y \in K^{m+1}$.
	Then 
	\[
	v(F(x) - F(y)) \ge v(F) + v(x - y) + (d - 1)v(x, y).
	\]
\end{lem}

\begin{proof}
See Section \ref{sec:multiindices} for conventions on multiindices. 
Without loss of generality, assume $d \ge 1$.
By the multibinomial theorem,
\[x^\alpha - y^\alpha 
= 
(x - y + y)^\alpha - y^\alpha 
= 
\sum_{\beta \ne 0} \binom{\alpha}{\beta} (x - y)^\beta y^{\alpha - \beta}.\] 
Since the coefficients are integral and vanish for all $\beta \not \le \alpha$, we have by the ultrametric inequality
\[
v(x^\alpha - y^\alpha) 
\ge 
\min_{0 \ne \beta \le \alpha} 
v((x-y)^\beta) + v(y^{\alpha-\beta}).
\]
Now if $\gamma$ is any multiindex and 
if $z$ is any point, 
then $v(z^\gamma) \ge |\gamma| v(z)$.
This implies
\begin{equation} \label{affine_thing}
v((x-y)^\beta) + v(y^{\alpha-\beta}) 
\ge 
|\beta| v(x-y)
+ |\alpha - \beta| v(x, y)
\end{equation}
whenever $0 \le \beta \le \alpha$.
The r.h.s.~of \eqref{affine_thing} is a non-decreasing extended-real valued function of $|\beta|$ because $v(x - y) \ge v(x, y)$ and $|\alpha - \beta| = |\alpha| - |\beta|$.
Since $\beta \ne 0$, the minimum occurs at $|\beta| = 1$.
Therefore 
\[v(x^\alpha - y^\alpha) \ge v(x-y) + (|\alpha| - 1) v(x,y)\]
for all multiindices $\alpha$ and all points $x, y$.
The lemma follows:
\[
v(F(x) - F(y))
= \min_j v\Big(\sum_{|\alpha|=d} F_{j,\alpha} (x^\alpha - y^\alpha) \Big) \\
\ge v(F) + \min_{|\alpha|=d} v(x^\alpha - y^\alpha).
\qedhere 
\]
\end{proof}

\begin{defn} \label{def_eps}
	Let $f : \P^m \to \P^M$ be a morphism of degree $d$ defined over $K$ and let $F$ be a homogeneous lift of $f$.
	Let $x \in K^{m+1} - 0^{m+1}$.
	The \emph{excess valuation} of $F$ at $x$ is the quantity
	\[\varepsilon_f(x) := v(F(x)) - dv(x) - v(F).\]
\end{defn}

\begin{lem} \label{lem:eps_trio}
The excess valuation is lift-independent, scale-invariant, non-negative, and bounded.
In fact, 
\[\lVert \varepsilon_f \rVert_\infty := \sup_{x \ne 0} \varepsilon_f(x) \le \inf_G -\big(v(F) + v(G)\big) \le v(\Res f).\]
where $G$ ranges over all pseudoinverses of $F$.
\end{lem}
\begin{proof}
	Lift-independence and scale-invariance are immediate from homogeneity of $F$ and additivity of $v$:
	\[v((\lambda F)(x)) - v(\lambda F) = v(F(x)) - v(F) \quad\text{and}\quad v(F(\lambda x)) - dv(\lambda x) = v(F(x)) - dv(x)\]
	for all $\lambda \in K^\times$.
	Non-negativity is the special case $y = 0$ of Lemma \ref{lem:F_cts}.
	Boundedness is standard: since $f$ is a morphism, Proposition \ref{sylv_null} implies $F$ has a pseudoinverse $G$ of some degree $e-d$; and for any such $G$,
	\begin{align*}
	ev(x_i) &= v(x_i^e) \\
	&\ge \min_j v(G_{ij}(x)) + v(F_j(x)) \\
	&\ge \min_j v(G_{ij}) + (e-d)v(x) + v(F_j(x)) \\
	&\ge v(G) + (e - d)v(x) + v(F(x)).
	\end{align*}
	for all $i$. Since the r.h.s.~is independent of $i$,
	it is dominated by $ev(x)$.
	Rearranging gives $\varepsilon_f(x) \le -\big(v(F) + v(G)\big)$.
	The resultant enters the picture via Proposition \ref{vRes}(ii).
\end{proof}

\begin{rmk}
	The quantity $\inf_G -\big(v(F) + v(G)\big)$ is lift-independent, because pseudoinverses for $F$ are in bijection with those for $\lambda F$ via scaling by $\lambda^{-1}$.
\end{rmk} 

The inequalities in Lemma \ref{lem:eps_trio} are not sharp,
as shown by the following example.

\begin{eg}
Let 
\[f(z) = \frac{z^2 + 4}{z^2 + 1} \text{ over } K = \Q \text{ at the place } v = 3.\]
Then:
\begin{enumerate}[(i)]
\item $\lVert \varepsilon_f \rVert_\infty = 0$;
\item $\inf_G -{\big(v(F) + v(G)\big)} = 1$;
\item $v(\Res f) = 2$.
\end{enumerate}
\end{eg}
\begin{proof}
Choose the lift $F(X, Y) = (X^2 + 4Y^2, X^2 + Y^2)$.
Then: 
\begin{enumerate}[(i)]
	\item $v(x^2 + y^2) = 2v(x, y)$ for all $(x, y) \in \Q^2$ because the binary form $X^2 + Y^2$ has no roots over $\mathbf{F}_3$.
	\item Over the extension $K = \Q(i)$, where $v$ remains inert, 
	we have $\varepsilon_f(i, 1) = v(3, 0) = 1$. 
	On the other hand, the relations
	\begin{align*}
		(X^2 + 4Y^2) - 4(X^2 + Y^2) &= -3X^2 \\
		(X^2 + 4Y^2) - (X^2 + Y^2) &= 3Y^2
	\end{align*}
	imply that $F$ has a pseudoinverse $G$ of degree 0 with $v(G) = -1$.
	\item The Sylvester--Macaulay matrix 
	\[\begin{bmatrix}
		1 &   & 1 &   \\
		0 & 1 & 0 & 1 \\
		4 & 0 & 1 & 0 \\
		  & 4 &   & 1
		\end{bmatrix}\]
	is square with determinant 9.
	\qedhere
\end{enumerate}
\end{proof}

We record an important relationship between good reduction and the excess valuation.

\begin{cor} \label{good_reduction_epsilon}
	If $f$ has good reduction at $v$, then $\varepsilon_f$ is identically zero.
\end{cor}
\begin{proof}
	Let $F$ be a $v$-normalized lift of $f$. 
	By Proposition \ref{good_red},
	$F$ has a $v$-integral pseudoinverse $G$.
	By Lemma \ref{lem:eps_trio}, 
	$0 \le \lVert \varepsilon \rVert_\infty \le -\big(v(F) + v(G)\big) = -v(G) \le 0$.
\end{proof}

Let $O = \{x \in K : v(x) \ge 0\}$ denote the ring of $v$-integers of $K$.
Then $O$ is a local ring with unique maximal ideal $\mf m = \{x \in K : v(x) > 0\}$.
By Lemma \ref{lem:R_prim_units},
\[
O^{m+1}_\prim 
= \{x \in K^{m+1} : v(x) = 0\},
\]
which is precisely the set of \emph{$v$-normalized} points.

The final result of this section says that $\varepsilon_f$ is 
locally constant.

\begin{propn} \label{eps_per}
	Let $x, y \in O^{m+1}_\prim$.
	If $v(x - y) \ge \lVert \varepsilon_f \rVert_\infty$ then $\varepsilon_f(x) = \varepsilon_f(y)$.
\end{propn}

\begin{proof}
We suppress the subscript $f$.
By lift-independence, 
we may assume $v(F) = 0$.
Since $x$ and $y$ are $v$-normalized, $v(x) = v(y) = 0$.
Thus 
$\varepsilon(x) = v(F(x))$
and
$\varepsilon(y) = v(F(y))$.
Suppose $\varepsilon(x) \ne \varepsilon(y)$.
Then by Lemma \ref{lem:F_cts} and the nonarchimedean property,
\begin{align*}
	v(x - y) &\le v(F(x) - F(y)) 
	\\
	&= \min\{\varepsilon(x), \varepsilon(y)\}
	\\
	&< \max\{\varepsilon(x), \varepsilon(y)\} \\
	&\le \lVert \varepsilon \rVert_\infty. \qedhere 
\end{align*}
\end{proof}

\begin{q}
	The excess valuation $\varepsilon_f$ extends continuously to $0$ just when it's constant: indeed, if $x$ is arbitrary and $|t| < 1$, then 
	$\varepsilon_f(x) = \lim_k \varepsilon_f(t^k x) = \varepsilon_f(0)$ by scale-invariance and continuity.
	Can $\varepsilon_f$ be a \emph{nonzero} constant?
\end{q}

\subsection{Global consequences}

For the rest of Section \ref{sec:excess},
\begin{itemize} 
	\item $K$ is the field of fractions of a Dedekind domain $R$;
\end{itemize}
for each finite place $v$ of $K$, 
\begin{itemize}
	\item $\mf p_v$ is the corresponding prime ideal of $R$, 
	\item $O_v = \{z \in K : v(z) \ge 0\} \supseteq R$,
	\item $\mf m_v = \{z \in K : v(z) = 0\} = \mf p_v O_v$,
	\item $\varepsilon_{f,v}$ is the excess valuation of $f$ with respect to $v$.
\end{itemize}

\begin{lem} \label{only_finitely_many_bad_places}
	Let $f : \P^m \to \P^M$ be a morphism defined over $K$.
	Then at all but at most finitely many finite places $v$ of $K$, $f$ has good reduction at $v$ (and hence $\varepsilon_{f,v}$ is identically zero).
\end{lem}
\begin{proof}
	Let $F$ be a lift of $f$.
	By Proposition \ref{sylv_null}(iii), 
	$F$ admits a pseudoinverse $G$.
	Since $F$ and $G$ are defined over $K$
	and each have a finite nonzero number of nonzero coefficients, 
	$v(F) = v(G) = 0$ for almost all $v$.
	The Lemma now follows from Proposition \ref{good_red} (and Corollary \ref{good_reduction_epsilon}).
\end{proof}

\begin{defn}
	Let $x \in K^{m+1} - 0^{m+1}$.
	The \emph{excess divisor} of $f$ at $x$ is the ideal
	\[\ell_f(x) := \prod_{v \nmid \infty} \mf p_v^{\varepsilon_{f,v}(x)}.\]
\end{defn}

Note that the product always converges: in light of Lemma \ref{only_finitely_many_bad_places}, only finitely many terms differ from $R = \langle 1 \rangle$. 

\begin{propn} \label{ell_properties}
	Let $f : \P^m \to \P^M$ be a morphism defined over $K$. 
\begin{enumerate}[(i)]
	\item Let $F$ be a lift of $f$. 
	Consider the three fractional ideals 
\begin{gather*}
	\langle F \rangle = \sum_{\substack{|\alpha|=d \\ 0 \le j \le M}} F_{j,\alpha} R,
	\\
	\langle x \rangle = \sum_{i=0}^m x_i R, \quad\text{and}\quad \langle F(x) \rangle = \sum_{j=0}^M F_j(x) R.
\end{gather*}
Then
\[
	\langle F(x) \rangle = \langle x \rangle^d \langle F \rangle \ell_f(x)
\]
for all $x \ne 0$.
	\item $\ell_f$ is $K^\times$-invariant, i.e.,~$\ell_f(\lambda x) = \ell_f(x)$ for all $\lambda \in K^\times$.
	\item Let $S = \{v \nmid \infty : \lVert \varepsilon_{f,v} \rVert_\infty > 0\}$.
	Then the set of excess divisors of $f$ is precisely
	\[\im \ell_f = \Big\{\prod_{v\in S} \mf p_v^{i_v} : i_v \in \im \varepsilon_{f,v} \text{ for all $v$ in $S$}\Big\}.\]
	In particular, for each $x \ne 0$, $\ell_f(x)$ is one of a fixed finite set of integral ideals.
	\item The set of excess divisors is closed under $\gcd$ and $\lcm$.
\end{enumerate}
\end{propn}

\begin{rmk*}
Property (i) is the \textit{raison d'\^etre} for $\ell_f(x)$, 
for if $K$ is a number field, 
then 
\[H_K(F(x)) = \frac{1}{\Nm {\langle F(x) \rangle}} \prod_{\sigma : K \into \C} |\sigma(F(x))|_\infty.\]
This, along with property (ii) and $\subseteq$ of (iii), will be essential to the proof of Theorem \ref{thm:main}.
Property (iv) is just a curiosity.
\end{rmk*} 

\begin{proof} \hfill 
\begin{enumerate}[(i)]
	\item 
	Recall that for any fractional ideals $\mf a_i$ and for any principal ideal $aR$, 
	\[v\big(\sum_i \mf a_i\big) = \min_i v(\mf a_i) \quad\text{and}\quad v(aR) = v(a).\]
	Thus, by unique factorization and the definition
	of the excess valuation,
	\begin{align*}
\langle F(x) \rangle 
	&= \prod_{v \nmid \infty} \mf p_v^{v(F(x))}
= \prod_{v \nmid \infty}\mf p_v^{dv(x) + v(F) + \varepsilon_{f,v}(x)}
= \langle x \rangle^d \langle F \rangle \ell_f(x).
\end{align*}
	\item 
	This is immediate from scale-invariance of $\varepsilon_{f,v}$ (cf.~Lemma \ref{lem:eps_trio}).
	\item 
	Let $x \ne 0$. 
	That $\ell_f(x)$ is an integral ideal is immediate from non-negativity of $\varepsilon_{f,v}(x)$ (cf.~Lemma \ref{lem:eps_trio}).
	Suppose $\mf p_v \mid \ell_f(x)$. 
	Then by definition,
	\[1 
	= v(\mf p_v) \le v(\ell_f(x)) 
	= \varepsilon_{f,v}(x)
	\in \im \varepsilon_{f,v}.\]
	Thus $v \in S$. This proves $\subseteq$; we postpone the proof of $\supseteq$ til after Proposition \ref{delta_product_formula}.
	In any case, 
	by Lemma \ref{only_finitely_many_bad_places}, $S$ is a finite set, 
	and by Lemma \ref{lem:eps_trio}, $\# \im \varepsilon_{f,v} \le 1 + \lVert \varepsilon_{f,v} \rVert < \infty$.
	Therefore
	\[\#\im \ell_f = \prod_{v\in S} \# \im \varepsilon_{f,v}\]
	is finite.
	\item 
	Suppose $\ell_f(x) = \prod_v \mf p_v^{i_v}$ and $\ell_f(y) = \prod_v \mf p_v^{j_v}$ are two elements of $\im \ell_f$.
	Then 
		\[
		\textstyle 
		\ell_f(x) + \ell_f(y) = \prod_v \mf p_v^{\min\{i_v, j_v\}}
		\quad\text{and}\quad 
		\ell_f(x) \cap \ell_f(y) = \prod_v \mf p_v^{\max\{i_v, j_v\}}.
		\]
		But for all $v$ in $S$, 
		the exponents $\min\{i_v, j_v\}$ and $\max\{i_v, j_v\}$, being among $i_v$ and $j_v$ themselves, 
		belong to $\im \varepsilon_{f,v}$.
		Thus, by part (iii), 
		there exist $z$ and $w$ 
		such that 
		\[\ell_f(x) + \ell_f(y) = \ell_f(z)
		\quad\text{and}\quad 
		\ell_f(x) \cap \ell_f(y) = \ell_f(w).
		\qedhere\]
\end{enumerate}
\end{proof}

We now show that the distribution of $\ell_f$ is completely determined by its values on a finite set.
To do so, we need the following technical result.
For ease of exposition, 
we declare a nonzero ideal $I$ of $R$ to be \emph{$v$-sufficiently large} if $v(I) \ge \lVert \varepsilon_{f,v} \rVert_\infty$,
and \emph{sufficiently large} if it is $v$-sufficiently large for all finite places $v$ of $K$.

\begin{lem} \label{eps_factors_thru_pi}
Let $I \leq R$ be a nonzero ideal
and let 
$\pi_I : K^{m+1} - 0^{m+1} \to \P^m(R/I)$ 
be the reduction map (cf.~Definition \ref{def:red_I}).
If $I$ is $v$-sufficiently large,
then there exists a function $\Tilde \varepsilon_{f,v} : \P^m(R/I) \to \Z$ 
such that \[\Tilde \varepsilon_{f,v} \circ \pi_I = \varepsilon_{f,v}.\]
\end{lem}
\begin{proof}
For clarity, drop the subscripts $f$, $v$, $I$.
The existence of 
$\Tilde \varepsilon$
is equivalent to $\varepsilon$ being constant on the fibres of $\pi$, so we show the latter.
This is trivial if $v(I) = 0$ (in this case $\varepsilon$ vanishes identically)
so assume $v(I) > 0$, 
and let us suppose $\pi(x) = \pi(y)$. 
By definition of $\pi$, there exist $\lambda, \zeta \in K^\times$ such that $[\Bar{\lambda x}] = [\Bar{\zeta y}]$.
By definition of $\P^m(R/I)$, there exists $\Bar u \in (R/I)^\times$ such that $\Bar{\lambda x} = \Bar{u \zeta y}$.
By definition of $(R/I)^{m+1}$, every coordinate of $\lambda x - u\zeta y$ belongs to $I$.
Thus by hypothesis,
\[v(\lambda x - u\zeta y) \ge v(I) \ge \lVert \varepsilon \rVert_\infty.\]
Now $\lambda x$ and $u\zeta y$ are primitive mod $I$, hence primitive mod every divisor of $I$, including $\mf p = \mf p_v$.
Also, $R \subseteq O$, 
so by Lemma \ref{prim_phi} applied to the inclusion map (and by Lemma \ref{lem:R_prim_units}),
\[\lambda x, u\zeta y \in R^{m+1}_{\prim/\mf p} \subseteq O^{m+1}_{\prim/\mf m} = O^{m+1}_\prim.\]
Applying Proposition \ref{eps_per} 
yields, by scale-invariance,
\[
\varepsilon(x) =
\varepsilon(\lambda x) = 
\varepsilon(u\zeta y) =
\varepsilon(y). \qedhere 
\]
\end{proof}

\begin{cor} \label{ell_descends}
The excess divisor descends to $\P^m(R/I)$ 
for any sufficiently large ideal $I$.
\end{cor}
\begin{proof}
Since $I$ is $v$-sufficiently large for every $v$, 
the excess valuations $\varepsilon_{f,v}$ descend to well-defined functions $\Tilde\varepsilon_{f,v}$ on $\P^m(R/I)$.
Thus for $Q \in \P^m(R/I)$ we may define
\[
\Tilde{\ell}_f(Q) := \prod_{v \nmid \infty} \mf p_v^{\Tilde{\varepsilon}_{f,v}(Q)}.
\]
Since $\Tilde{\varepsilon}_{f,v} \circ \pi = \varepsilon_{f,v}$ we likewise have $\Tilde{\ell}_f \circ \pi = \ell_f$.
\end{proof}

Henceforth, we will drop the tilde ( $\Tilde{}$ ) from the notation, 
thereby allowing $\ell_f$ and $\varepsilon_{f,v}$, originally defined on $K^{m+1} - 0^{m+1}$, to take inputs from any sufficiently large projective space $\P^m(R/I)$, as well as $\P^m(K)$ itself.

\begin{rmk}
By Lemma \ref{lem:eps_trio}, 
$\Res f$ is a sufficiently large ideal of $R$.
In particular, every excess divisor of $f$ divides the resultant ideal.
\end{rmk}

\subsection{Divisor densities}

\begin{defn}
Let $v$ be a finite place of $K$,
let $i$ be an integer,
and let $\mf l$ be a fractional ideal.
The \emph{local density}
is the rational number
\[
\delta_{f,v}(i) := \frac{\#\{Q \in \P^m(R/I) : \varepsilon_{f,v}(Q) = i\}}{\#\P^m(R/I)}
\]
where $I$ is any $v$-sufficiently large ideal;
similarly,
the \emph{global density} 
is 
\[
	\delta_f(\mf l) := \frac{\#\{Q \in \P^m(R/I) : \ell_f(Q) = \mf l\}}{\# \P^m(R/I)}
\]
where $I$ is sufficiently large.
\end{defn}

Note that the above definitions are independent of the choice of ideal:
if $J \le I$ are sufficiently large 
then by Proposition \ref{Pm_preserves_products}, 
Lemma \ref{Pm_constant_degree}, 
and the Chinese Remainder Theorem,
the reduction map $\P^m(R/J) \to \P^m(R/I)$ has constant degree.

The following properties of the local density
are straightforward consequences of the definitions,
so we omit the proof.

\begin{lem} \label{delta_trio}
For each $v \nmid \infty$ and each $i \in \Z$,
	\begin{enumerate}[(i)]
		\item $\delta_{f,v}(i) \ge 0$ with equality if and only if $i \not \in \im \varepsilon_{f,v}$,
		\item $\delta_{f,v}(i) \le 1$ with equality if and only if $\im \varepsilon_{f,v} = \{i\}$.
	\end{enumerate}
	Moreover,
	\begin{enumerate}[(i)] \setcounter{enumi}{2}
		\item $\displaystyle \sum_{i=0}^\infty \delta_{f,v}(i) = 1$.
	\end{enumerate}
\end{lem}

\begin{propn}[Density Formula] \label{delta_product_formula}
	For all fractional ideals $\mf l$, 
	\[\delta_f(\mf l) = \prod_{v \nmid \infty} \delta_{f,v}(v(\mf l)).\]
\end{propn}

\begin{rmk*}
	The global density is not a multiplicative function unless
	$f$ has no nontrivial excess divisor, 
	in which case $\delta_f(\mf l) = 0$ for all $\mf l \ne \langle 1 \rangle$.
\end{rmk*}

\begin{proof}
Fix a sufficiently large ideal $I$.
If $\mf l \nmid I$ 
then for some $v$ we have $v(\mf l) > v(I) \ge \lVert \varepsilon_{f,v} \rVert_\infty$.
This means 
$v(\mf l) \not \in \im \varepsilon_{f,v}$,
so $\mf l \not \in \im \ell_f$ 
(by $\subseteq$ of Proposition \ref{ell_properties}(iii)),
so the Density Formula holds trivially ($0 = 0$).

Now suppose $\mf l \mid I$.
If $v(I) = 0$
then $v(\mf l) = 0$ and $\im \varepsilon_{f,v} = \{0\}$ (since $I$ is sufficiently large).
For these places $v$, 
we have $\delta_{f,v}(v(\mf l)) = \delta_{f,v}(0) = 1$.
Thus the r.h.s.~of the Density Formula 
reduces to a product over just $v \mid I$.
Here, let 
	\[\varphi : \P^m(R/I) \overset{\sim}{\to} \prod_{v \mid I} \P^m(R/\mf p_v^{v(I)})\]
	be the bijection induced by the Chinese Remainder Theorem.
Put
	\[Z(\mf l) := \{Q \in \P^m(R/I) : \ell_f(Q) = \mf l\}\]
and, for each integer $i$, 
	\[Z_v(i) := \{Q \in \P^m(R/\mf p_v^{v(I)}) : \varepsilon_{f,v}(Q) = i\}.\]
Then 
\[
\delta_f(\mf l) = \frac{\# Z(\mf l)}{\# \P^m(R/I)}
\quad\text{and}\quad 
\delta_{f,v}(v(\mf l)) = \frac{\# Z_v(v(\mf l))} {\# \P^m(R/\mf p_v^{v(I)})}
\]
for all $v \mid I$, the latter holding because the ideal $\mf p_v^{v(I)}$ is $v$-sufficiently large.
Therefore it's enough to prove
	\[\varphi(Z(\mf l)) = \prod_{v \mid I} Z_v(v(\mf l)). \tag{*}\]
But \emph{this} holds thanks to the following observation:
if $\varphi(Q) = (Q_v : v \mid I)$
then the $Q_v$'s are just the reductions of $Q$ modulo the $\mf p_v^{v(I)}$'s, and hence 
\[v(\ell_f(Q)) = \varepsilon_{f,v}(Q) = \varepsilon_{f,v}(Q_v)\]
for all $v \mid I$.
We leave the verification of the mutual inclusions in (*) to the reader.
\end{proof}

Armed with the Density Formula,
we may now complete our description of $\im \ell_f$.

\begin{proof}[Proof of Proposition \ref{ell_properties}(iii) ($\supseteq$)]
Suppose $\mf l = \prod_{v \in S} \mf p_v^{i_v}$
where each $i_v \in \im \varepsilon_{f,v}$.
Then by Lemma \ref{delta_trio} 
and the Density Formula,
\[\delta_f(\mf l) 
	= \prod_{v \nmid \infty} \delta_{f,v}(v(\mf l))
	= \prod_{v \in S} \delta_{f,v}(i_v) \prod_{v \not\in S} \delta_{f,v}(0)
	> 0\]
so the set defining $\delta_f(\mf l)$ is non-empty.
In particular, $\mf l = \ell_f(x)$ for some $x$.
\end{proof}

The following calculation is of fundamental importance.
\begin{propn} \label{delta_mult_sum} 
Let $\chi$ be any multiplicative function.
Then 
\[\sum_{\mf l} \chi(\mf l) \delta_f(\mf l) = \prod_{v \nmid \infty} \sum_{i=0}^\infty \chi(\mf p_v^i) \delta_{f,v}(i)\]
where $\mf l$ ranges over all (nonzero, integral) ideals.
\end{propn}

\begin{proof}
	Let $I$ be any sufficiently large ideal.
Then
\begin{align*}
\sum_{\mf l} \chi(\mf l) \delta_f(\mf l)
&= \sum_{\mf l \mid I} \chi(\mf l) \delta_f(\mf l) 
&& \text{[since $\delta_f(\mf l) = 0$ for $\mf l \nmid I$]} 
\\
&= \sum_{\mf l \mid I} \chi(\mf l) \prod_{v \nmid \infty} \delta_{f,v}(v(\mf l)) 
&& \text{[by the Density Formula]} 
\\
&= \sum_{\mf l \mid I} \chi(\mf l) \prod_{v \mid I} \delta_{f,v}(v(\mf l))
&& \text{[since $\delta_{f,v}(v(\mf l)) = 1$ for $v \nmid I$]} \\
&= \sum_{\mf l \mid I} \prod_{v \mid I} \chi(\mf p_v^{v(\mf l)}) \delta_{f,v}(v(\mf l))
&& \text{[by multiplicativity]}
\\
&= \prod_{v\mid I} \sum_{i=0}^{v(I)} \chi(\mf p_v^i) \delta_{f,v}(i)
&& \text{[by unique factorization]} 
\\
&= \prod_{v\mid I} \sum_{i=0}^\infty \chi(\mf p_v^i) \delta_{f,v}(i)
&& \text{[by Lemma \ref{delta_trio}(i)]}
\\
&= \prod_{v \nmid \infty} \sum_{i=0}^\infty \chi(\mf p_v^i) \delta_{f,v}(i);
\end{align*}
the last line holds because if $v \nmid I$ then $\sum_{i=0}^\infty \chi(\mf p_v^i) \delta_{f,v}(i) = \chi(1) \delta_{f,v}(0) = 1$.
\end{proof}

\section{Nonarchimedean computations} \label{sec:nonarch}

In this section,
\begin{itemize} 
\item $K$ is a number field 
\item $\mathcal{O}_K$ is the ring of integers of $K$
\item $v$ is a nonarchimedean place of $K$ corresponding to the prime ideal $\mf p_v \subset \mathcal{O}_K$
\item $|{\,\cdot\,}|_v = q_v^{-v(\cdot)}$ is the $v$-adic absolute value
\item $K_v$ is the completion of $K$ w.r.t.~$v$
\item $O_v$ is the ring of $v$-integers of $K_v$ (n.b.~$O_v \subset \mathcal{O}_K$)
\item $\mf m_v$ is the maximal ideal of $O_v$
\item $\lVert {\,\cdot\,} \rVert_v = [O_v : \mf m_v]^{-v(\cdot)}$ is the normalized $v$-adic absolute value
\item $B_v(a, r) = \{z \in K_v^{m+1} : |z - a|_v \le r\}$ is the polydisc of radius $r$ centered at $a$
\item $\mu_v$ is the Haar measure on $K_v^{m+1}$ with $\mu_v(B_v(0, 1)) = 1$
\end{itemize}
Keep in mind that the Haar measure is homogeneous and translation-invariant:
\begin{equation} \label{eq:Haar_properties}
\mu_v(\lambda S) = \lVert \lambda \rVert_v^{m+1} \mu_v(S)
\quad\text{and}\quad 
\mu_v(z + S) = \mu_v(S)
\end{equation}
for all measurable sets $S \subseteq K_v^{m+1}$, scalars $\lambda \in K_v$, and points $z \in K_v^{m+1}$.

\begin{lem} \label{many_faces_of_delta}
Let $\varepsilon_{f,v}$ be the excess valuation
and let $\delta_{f,v}$ be the local density.
Then for all integers $i$,
\[
\delta_{f,v}(i) = \lim_{X \to \infty} \frac{\#\{P \in \P^m(K) : H(P) \le X \text{ and } \, \varepsilon_{f,v}(P) = i\}}{\Num_{H, \P^m(K)}(X)}
\tag{\textit{i}}
\]
and
\[
\delta_{f,v}(i) = \mu_v \big(B_v(0, 1) \cap \varepsilon_{f,v}^{-1}(i)\big).
\tag{\textit{ii}}
\]
\end{lem}
\begin{rmk}
	Part (i) justifies the intuition that $\delta_{f,v}(i)$ is the ``probability'' that $v(F(x)) = dv(x) + v(F) + i$.
	Part (ii) identifies $\delta_{f,v}(i)$ as a local volume, which we will use to calculate $c_{K,v}(f)$ soon.
\end{rmk} 
\begin{proof}
	To ease notation, drop the subscripts $f$ and $v$.
Let $s \ge \lVert \varepsilon \rVert_\infty$
and let $\pi : K^{m+1}-0^{m+1} \onto \P^m(\mathcal{O}_K/\mf p^s)$ be the reduction map.
Since $s$ is sufficiently large, 
$\varepsilon(P) = \varepsilon(\pi(P))$ for all $P \in \P^m(K)$ (cf.~Lemma \ref{eps_factors_thru_pi}).
	Put 
	\[Z := \{Q \in \P^m(\mathcal{O}_K/\mf p^s) : \varepsilon(Q) = i\}.\]
	Then
	\begin{align*}
		\#\{P \in \P^m(K) : H(P) \le X, \ \varepsilon(P) = i\} 
		&= \sum_{Q \in Z} \{P \in \P^m(K) : H(P) \le X, 
		\ \pi(P) = Q\} \\
		&\sim |Z| \cdot \frac{\Num_{H, \P^m(K)}(X)}{\#\P^m(\mathcal{O}_K/\mf p^s)}
		\qquad (X \to \infty)
	\end{align*}
	which proves part (i).
	
	To prove (ii), 
	again let $s \ge \lVert \varepsilon \rVert_\infty$.
	If $s = 0$ then $\varepsilon^{-1}(i)$ is either empty or the whole space (according as $i$ is positive or zero) and the result is trivial.
	So, assume $s > 0$.
Let $\pi : K_v^{m+1} - 0^{m+1} \onto \P^m(O/\mf m^s)$ be the reduction map (note the different codomain);
in terms of a uniformizer $t$, 
we have $\pi(z) = [\Bar{t^{-v(z)} z}]$ for all $z$.
	If $Q = [\Bar y]$ for some primitive $y$
	then 
	\[\pi^{-1}(Q) = \bigsqcup_{k \in \Z} \bigsqcup_{\Bar u \in (O/\mf m^s)^\times} t^k B(uy, q^{-s}).\]
	Intersecting with $B(0, 1)$ restricts the union to non-negative $k$; indeed, if $z = t^k a \in B(0, 1) \cap t^k B(uy, q^{-s})$ 
	for some $a$ with $|a - uy| \le q^{-s} < 1$, 
	then $|a| = 1$ because $|uy| = |y| = 1$, 
	so $|z| = |t^k a| = q^{-k} \le 1$.
	It follows that
	\[
	\mu \big(B(0, 1) \cap \pi^{-1}(Q)\big) 
	= \sum_{k \ge 0} \sum_{\Bar u \in (O/\mf m^s)^\times} \mu \big(t^k B(uy, q^{-s})\big).
	\]
	By properties of the Haar measure \eqref{eq:Haar_properties} we have
	\[
	\mu \big(t^k B(uy, q^{-s})\big) 
	= 
	\lVert t^{s+k} \rVert^{m+1}
	= [O : \mf m]^{-(s+k)(m+1)}
	\]
	for all $v$-units $u$.
	Summing the geometric series 
	yields 
	\[\mu \big( B(0, 1) \cap \pi^{-1}(Q) \big)
	=
	|(O/\mf m^s)^\times| \frac{[O : \mf m]^{-s(m+1)}}{1 - [O : \mf m]^{-(m+1)}} 
	= \frac{J_{O,1}(\mf m^s)}{J_{O,m+1}(\mf m^s)} = \frac{1}{\#\P^m(O/\mf m^s)}\]
	by Proposition \ref{J_local} and the fact that $[O : \mf m]^s = [O : \mf m^s]$.
	Finally,
	the inclusion $\mathcal{O}_K \subset O$ 
	induces a bijection $\P^m(\mathcal{O}_K/\mf p^s) \simto \P^m(O/\mf m^s)$.
	Identifying $Z \subseteq \P^m(\mathcal{O}_K/\mf p^s)$ from part (i) with its image in $\P^m(O/\mf m^s)$ under this bijection,
	part (ii) follows from the identity 
	\[
	\varepsilon^{-1}(i) = \bigsqcup_{Q \in Z} \pi^{-1}(Q). \qedhere 
	\]
\end{proof}

\begin{cor} \label{epsilon_pushforward}
	As discrete measures on $\Z$, 
	\[(\varepsilon_{f,v})_* \big(\mu_v|_{B_v(0,1)}\big) = \sum_{i \ge 0} \delta_{f,v}(i) \mu_i\]
where $\mu_i$ is the point mass at $i$.
\end{cor}
\begin{proof}
Evaluate both sides at each singleton $\{i\}$
using Lemma \ref{many_faces_of_delta}(ii).
\end{proof}

\begin{cor} \label{nonarch_not_vol}
\hfill
\begin{enumerate}[(i)]
\item In the notation of Theorem \ref{thm:main}, 
\[c_{K,v}(f) = \displaystyle \sum_{i \ge 0} \Nm \mf p_v^{(m+1)i/d} \delta_{f,v}(i).\]
\item Let $D_{f,v} = \{z \in K_v^{m+1} : |F(z)|_v \le |F|_v\}$. Then
\[
\mu_v(D_{f,v}) = \sum_{i \ge 0} \Nm \mf p_v^{(m+1)\lfloor i/d \rfloor} \delta_{f,v}(i).
\]
\item $\mu_v(D_{f,v}) \le c_{K,v}(f)$ with equality if and only if $\delta_{f,v}(i) = 0$ for all $d \nmid i$.
\end{enumerate}
\end{cor}

\begin{proof}
To declutter, 
we drop the subscripts $f$ and $v$ where no confusion can arise.
By definition of $\varepsilon$, 
\[
|F(z)| = |z|^d |F| q^{-\varepsilon(z)} \qquad (z \ne 0).
\]
In terms of the local degree $n_v = [K_v : \Q_v]$,
we have $q = \Nm \mf p^{1/n_v}$.
Thus
\begin{align*}
	c_{K,v}(f)
	&= \int_{O^{m+1}} 
	\bigg( 
	\frac{|z| |F|^{1/d}}{|F(z)|^{1/d}} 
	\bigg)^{n_v(m+1)}
	\d \mu(z) \\
	&= \int_{O^{m+1}} 
	q^{n_v(m+1)\varepsilon(z)/d} \, \d \mu(z) \\
	&= \sum_{i \ge 0} q^{n_v(m+1)i/d} \delta(i)
\end{align*}
by Corollary \ref{epsilon_pushforward}.
Part (i) 
follows from the identity $q^{n_v} = \Nm \mf p$.
For part (ii), 
note that
\[
|F(z)| \le |F| \iff |z| \le q^{\varepsilon(z)/d}.
\]
Thus
\begin{equation} \label{eq:Dfv_decomposition}
	D_{f,v}
	= \{0\} \sqcup \bigsqcup_{i \ge 0} B(0, q^{i/d}) \cap \varepsilon^{-1}(i).
\end{equation}
In general, 
if $|K_v^\times| = q^\Z$ for some $q > 1$ then 
\begin{equation} \label{rounded_radius}
	B(a, r) = B(a, \lfloor r \rfloor_v) 
	\ \text{where} \ \lfloor r \rfloor_v = \sup |K_v| \cap [0, r] = q^{\lfloor \frac{\log r}{\log q}\rfloor}
\end{equation}
for all real $r \ge 0$.
Thus, if $t$ is a uniformizer (i.e.,~$|t| = q^{-1}$)
then it follows from \eqref{rounded_radius} and scale-invariance of $\varepsilon$ 
that 
\[B(0, q^{i/d}) \cap \varepsilon^{-1}(i) = t^{-\lfloor i/d \rfloor}\big(B(0, 1) \cap \varepsilon^{-1}(i)\big).\]
This, \eqref{eq:Dfv_decomposition},
and	Lemma \ref{many_faces_of_delta}(ii) imply
\[
\mu(D_{f,v}) = 0 + \sum_{i \ge 0} \lVert t^{-\lfloor i/d \rfloor\!} \rVert^{m+1} \mu \big(B(0, 1) \cap \varepsilon^{-1}(i)\big)
= \sum_{i \ge 0} \Nm \mf p^{(m+1)\lfloor i/d \rfloor} \delta(i)
\]
which proves part (ii).
Part (iii) now follows from parts (i) and (ii):
since $\lfloor i/d \rfloor \le i/d$ for all $i$, 
we have $\mu(D_{f,v}) \le c_{K,v}(f)$ 
with equality if and only if  
\[ \big({\Nm \mf p^{(m+1)\{i/d\}} - 1}\big)\delta(i) = 0 \ \text{for all} \ i \ge 0 \]
which is equivalent to the vanishing of $\delta(i)$ for each $i$ not divisible by $d$.
\end{proof}

The divisibility condition in Corollary \ref{nonarch_not_vol}(iii) is not vacuous.

\begin{eg} \label{eg:nonarch_not_vol}
For $d > 1$ let 
\[f(z) = pz^d + 1 \text{ over } K = \Q \text{ at the place } v = p.\]
Then
\[\delta_{f,v}(0) = \frac{p}{p+1} \quad\text{and}\quad \delta_{f,v}(1) = \frac{1}{p+1}.\]
\end{eg}
\begin{proof}
Choose the lift $F(X, Y) = (pX^d + Y^d, Y^d)$.
Clearly $v(F) = 0$. 
If $(x, y) \in \Q^2 - (0, 0)$ then
$v(px^d) \ge v(y^d)$ implies $v(px^d + y^d) \ge v(y^d)$ 
while $v(px^d) < v(y^d)$ implies $v(px^d + y^d) = v(px^d)$.
It follows that
\[v(F(x, y)) 
= \begin{dcases*}
	dv(y)  & if $v(x) \ge v(y),$ \\
	1 + dv(x) & if $v(x) < v(y).$
\end{dcases*}
\]
Thus, $\varepsilon_{f,v}(x, y) = 1$ if and only if $(x, y)$ lies over $[\Bar 1 : \Bar 0] \in \P^1(\Z/p\Z)$.
\end{proof}

\begin{cor} \label{nonarch_vol}
In the notation of Theorem \ref{thm:main}, we have 
\[
c_{K,0}(f) 
=
\sum_{\mf l} \Nm \mf l^{(m+1)/d} \delta_f(\mf l).
\]
\end{cor}
\begin{proof}
This is immediate from
Proposition \ref{delta_mult_sum}
and
Corollary \ref{nonarch_not_vol}.
\end{proof}

\section{Some number theory} \label{sec:nt}

The calculations in this section
will simplify some steps in the proof of Theorem \ref{thm:main}.

\subsection{Solving congruences}

\begin{lem}[Generalized Chinese Remainder Theorem] \label{GCRT}
	Let $R$ be a ring containing ideals $I, J$
	and elements $a, b$.
	Then the system of congruences
	\[x \equiv a \!\!\! \pmod I \quad\text{and}\quad x \equiv b \!\!\! \pmod J\]
	has a solution $x_0$ in $R$ 
	if and only if 
	$a \equiv b \pmod {I + J}$, 
	in which case the full solution set is $x_0 + (I \cap J)$.
\end{lem}

\begin{proof}
	Consider the ring homomorphism
	\begin{align*}
		\varphi : R &\to R/I \times R/J \\
		x &\mapsto (x + I, x + J).
	\end{align*}
	Solving the given system is equivalent to 
	describing the fibre of $\varphi$ over $(a + I, b + J)$.
	Thus, 
	a solution $x_0$ exists iff $(a + I, b + J) \in \im \varphi$,
	in which case $\varphi^{-1}(a + I, b + J) = x_0 + \ker \varphi$.
	It's clear that $\ker \varphi = I \cap J$,
	so we're done once we show that 
	\[\im \varphi = \{(r + I, s + J) : r - s \in I + J\}.\]
	One inclusion is obvious: $x - x = 0$.
	For the other, suppose $r - s \in I + J$, 
	say $r - s = i - j$ for some $i$ in $I$ and $j$ in $J$.
	If we let $x = r - i$,
	then $r - x = i \in I$
	and $s - x = s - r + i = j - i + i = j \in J$.
	Thus $(r + I, s + J) = (x + I, x + J) \in \im \varphi$.
\end{proof}

\begin{lem} \label{union_of_cosets}
Let $\mf a, \mf b, \mf l, \mf R$ be nonzero integral ideals,
with
$\mf R$ sufficiently large
and coprime to $\mf a$.
Recall that
\[Z(\mf l) = \{Q \in \P^m(\mathcal{O}_K/\mf R) : \ell_f(Q) = \mf l\}
\quad\text{and}\quad 
U(\mf R) = (\mathcal{O}_K/\mf R)^\times.\]
Then there exist points $x_{Q,\Bar u}$,
one for each $Q \in Z(\mf l)$
and $\Bar u \in U(\mf R)$,
such that
\begin{gather*}
\hspace{-12em}
\{x \in (\mathcal{O}_K)^{m+1}_{\prim/\mf R} : \langle x \rangle \subseteq \mf{ab} \ \text{and} \ \, \ell_f(x) = \mf l\} 
\\ 
\hspace{12em} =
\begin{dcases*}
	\bigsqcup_{Q \in Z(\mf l)}
	\bigsqcup_{\Bar u \in U(\mf R)}
	x_{Q, \Bar u} + (\mf{ab} \cap \mf R)^{(m+1)}
	& if $\mf b \perp \mf R$, \\
	\varnothing & otherwise.
\end{dcases*}
\end{gather*}
\end{lem}

\begin{proof}
Of course,
by definition of $\P^m(\mathcal{O}_K/\mf R)$,
\begin{gather*}
\hspace{-12em}
\{x \in (\mathcal{O}_K)^{m+1}_{\prim/\mf R} : \langle x \rangle \subseteq \mf{ab} \ \text{and} \ \ell_f(x) = \mf l\} 
\\ 
\hspace{12em} =
\bigsqcup_{[\Bar y] \in Z(\mf l)} 
\bigsqcup_{\Bar u \in U(\mf R)} 
\{x \in (\mathcal{O}_K)^{m+1}_{\prim/\mf R} : \langle x \rangle \subseteq \mf{ab} \ \text{and} \ \Bar x = \Bar u \Bar y\}
\end{gather*}
where the second condition holds in the ring $(\mathcal{O}_K/\mf R)^{m+1}$.
Note that $Z(\mf l)$ is well-defined by Corollary \ref{ell_descends};
note also that the inner union is disjoint because $U(\mf R)$ acts freely on $(\mathcal{O}_K/\mf R)^{m+1}_\prim$ (Lemma \ref{lem:unit_group_acts_freely}).
The conditions $\langle x \rangle \subseteq \mf{ab}$
and $\Bar x = \Bar {uy}$ 
are, respectively, equivalent to the (componentwise!)~congruences $x \equiv 0^{m+1}$ (mod $\mf{ab}$) and $x \equiv uy$ (mod $\mf R$).
Moreover, since $uy$ is primitive mod $\mf R$, 
Lemma \ref{prim_mod_I} 
implies that
any $x \in \mathcal{O}_K^{m+1}$ congruent (mod $\mf R$) to $uy$ is likewise primitive mod $\mf R$.
It follows that 
the set 
\[\{x \in (\mathcal{O}_K)^{m+1}_{\prim/\mf R} : \langle x \rangle \subseteq \mf{ab} \ \text{and} \ \Bar x = \Bar{uy}\}\]
is precisely the solution space of the system of congruences
	\begin{equation} \label{congruences}
		x \equiv 0^{m+1} \!\!\! \pmod{(\mf{ab})^{(m+1)}} 
		\quad\text{and}\quad 
		x \equiv uy \!\!\! \pmod{\mf R^{(m+1)}}
	\end{equation}
in the ring $\mathcal{O}_K^{m+1}$.
	By the Generalized Chinese Remainder Theorem,
	this set is empty if 
	$uy \not \in (\mf{ab})^{(m+1)} + \mf R^{(m+1)} =
	(\mf{ab} + \mf R)^{(m+1)}$,
	or else it is a coset of $(\mf{ab})^{(m+1)} \cap \mf R^{(m+1)} = (\mf{ab} \cap \mf R)^{(m+1)}$.
	Now, since $u \in \mathcal{O}_K$ is invertible modulo $\mf R$,
	$u$ is \emph{a fortiori} invertible modulo any ideal containing $\mf R$.
	Also, 
	\begin{equation} \label{eq:ab_plus_R}
		\mf {ab} + \mf R 
		= 
		\mf {ab} + \mf{Rb} + \mf R
		=
		(\mf a + \mf R) \mf b + \mf R 
		=
		\mf b + \mf R
	\end{equation}
	because $\mf {Rb} \subseteq \mf R$ and $\mf a + \mf R = \mathcal{O}_K$.
	Thus the condition on $y$ simplifies to
	\[
	uy \in (\mf{ab} + \mf R)^{(m+1)} 
	\iff 
	y \in (\mf b + \mf R)^{(m+1)} 
	\iff 
	\langle y \rangle \subseteq \mf b + \mf R.
	\]
	But because $y$ is primitive mod $\mf R$, 
	if $\langle y \rangle \subseteq \mf b + \mf R$ 
	then $\mathcal{O}_K = \langle y \rangle + \mf R \subseteq \mf b + \mf R + \mf R = \mf b + \mf R \subseteq \mathcal{O}_K$, 
	so in fact $\mf b + \mf R = \mathcal{O}_K$;
	the converse is obvious.
	To finish the proof of the Lemma (as stated),
	define $x_{Q,\Bar u} \in \mathcal{O}_K^{m+1}$ for each $Q = [\Bar y] \in Z(\mf l)$ and $\Bar{u} \in U(\mf R)$ as follows:
	If $\mf b$ and $\mf R$ are coprime,
	let $x_{Q, \Bar u}$ be a solution to \eqref{congruences};
	otherwise, let $x_{Q, \Bar u}$ be arbitrary.
\end{proof}

\subsection{Dirichlet series}

The number of positive integers up to $x$ 
is obviously
\[
N_\Q(x) := \sum_{k \le x} 1 = \lfloor x \rfloor = x + O(1).
\]
The generalization of this formula 
to number fields 
is given by \emph{Weber's theorem}.

\begin{propn}[Weber]
Let $K$ be a number field of degree $n$,
and let $a_k$ be the number of (nonzero) 
integral ideals $\mf a \subseteq \mathcal{O}_K$ with $\Nm \mf a = k$.
Then 
\[N_K(x) := \sum_{k \le x} a_k = \frac{2^{r_1} (2\pi)^{r_2} hR}{w\sqrt{d_K}} x + O(x^{1-1/n}).\]
In particular, 
there exists a constant $A$
(depending only on $K$)
such that $|N_K(x)| \le Ax$ for all $x \ge 0$.
\end{propn}
\begin{proof}
See \cite[Theorem 5]{MurtyVanOrder}.
\end{proof}

\begin{lem} \label{lem:weber}
	Let $K$ be a number field of degree $n$.
	Then for all $x \ge 1$,
	\[
	\sum_{\Nm \mf b \le x} \frac{1}{\Nm \mf b} = O(1 + \log x)
	\]
	and, for $\sigma > 1$, 
	\[
	\sum_{\Nm \mf b > x} \frac{1}{\Nm \mf b^\sigma} = O(x^{1-\sigma})
	\]
	where the implicit constants depend only on $K$ and $\sigma$.
\end{lem}
\begin{proof}
This is just Abel summation.
By Weber's theorem we have
	\[
	\sum_{\Nm \mf b \le x} \frac{1}{\Nm \mf b}
	= 
	\sum_{k \le x} \frac{a_k}{k} 
	= 
	\frac{N_K(x)}{x} + \int_1^x \frac{N_K(t)}{t^2} \, \d t
	\le 
	A + \int_1^x \frac{A}{t} \, \d t = A(1 + \log x).
	\]
	Similarly, 
	\[
	\sum_{x < \Nm \mf b \le y} \frac{1}{\Nm \mf b^\sigma} 
	=
	\sum_{x < k \le y} \frac{a_k}{k^\sigma} 
	= 
	\frac{N_K(y)}{y^\sigma} - \frac{N_K(x)}{x^\sigma} + \sigma \int_x^y \frac{N_K(t)}{t^{\sigma+1}} \, \d t.
	\]
	If $\sigma > 1$, then $N_K(y) = o(y^\sigma)$, so we may let $y \to \infty$, giving
	\[
	\sum_{\Nm \mf b > x} \frac{1}{\Nm \mf b^\sigma} 
	= 
	-\frac{N_K(x)}{x^\sigma} + \sigma\int_x^\infty \frac{N_K(t)}{t^{\sigma+1}} \, \d t
	\le
	\sigma \int_x^\infty \frac{A}{t^\sigma} \, \d t
	= 
	\frac{A\sigma}{\sigma - 1} x^{1-\sigma}. \qedhere 
	\]
\end{proof}

Up next is the technical heart of the nonarchimedean part.

\begin{lem} \label{lem:dirichlet_series}
	Let $\mf a, \mf R \subseteq \mathcal{O}_K$ be nonzero coprime ideals, 
	and let $s \in \C$ with $\sigma := \Re(s) \ge 1$.
	Define 
	\[
	\chi_s(\mf b) := \mu_K(\mf b) [\mf b \perp \mf R] \frac{\Nm \mf a^s}{\Nm (\mf{ab} \cap \mf R)^s}.
	\]
	(Here, $\Nm \mf a^s$ means $(\Nm \mf a)^s$, etc.)
	Then as $x \to \infty$, 
	\[
	\sum_{\Nm \mf b \le x} \chi_s(\mf b)
	= 
	\frac{1}{\zeta_K(s) J_{\mathcal{O}_K,s}(\mf R)} + O\Big(\frac{x^{1-\sigma}}{\Nm \mf R^\sigma}\Big) \qquad (\sigma > 1)
	\tag{\textit{i}}
	\]
	(cf.~[Corollary 4.11]) and
	\[
	\sum_{\Nm \mf b \le x} \left| \chi_s(\mf b) \right|
	= 
	O\Big(\frac{1+\log x}{\Nm \mf R^\sigma}\Big) \qquad (\sigma \ge 1)
	\tag{\textit{ii}}
	\] 
	where the implicit constants depend only on $K$ and $\sigma$,
	and the $\log x$ disappears for $\sigma > 1$.
\end{lem}
\begin{rmk}
	For our purposes, $s$ will be $m+1$ or $m+1-1/n$, 
	where $m$ is the dimension of $\P^m$ and $n$ is degree of $K$. 
	Note,
	$m + 1 - 1/n \ge 1$
	with equality if and only if $n = m = 1$.
\end{rmk}
\begin{proof}
	We simplify $\chi_s(\mf b)$ first.
	By \eqref{eq:ab_plus_R} 
	and the lcm--gcd identity
	\[(I \cap J)(I + J) = IJ\]
	we have 
	\[
	\mf{ab} \cap \mf R 
	= 
	\frac{\mf{abR}}{\mf{ab} + \mf R}
	= 
	\frac{\mf{abR}}{\mf{b} + \mf R}
	\]
	and so
	\[
	\frac{\mf a}{\mf{ab} \cap \mf R}
	= \frac{\mf b + \mf R}{\mf {bR}}.
	\]
	Given that $\Nm (\mf b + \mf R) = 1$ whenever $\mf b$ is coprime to $\mf R$, 
	it follows that 
	\[
	\chi_s(\mf b) = \frac{\mu_K(\mf b) [\mf b \perp \mf R]}{(\Nm \mf {bR})^s}
	\]
	for all $\mf b$ and all $s$.
	In particular, 
	\begin{equation} \label{eq:chi_bound}
		|\chi_s(\mf b)|
		\le \frac{1}{(\Nm \mf{bR})^\sigma}
	\end{equation}
	for all $\mf b$ and all $s$.
	
	Now we prove part (ii).
	If $\sigma = 1$, 
	then by \eqref{eq:chi_bound} and Lemma \ref{lem:weber},
	\[
	\sum_{\Nm \mf b \le x} |\chi_s(\mf b)|
	\le
	\frac{1}{\Nm \mf R} \sum_{\Nm \mf b \le x} \frac{1}{\Nm \mf b}
	= O\Big(\frac{1 + \log x}{\Nm \mf R}\Big).
	\]
	On the other hand, if $\sigma > 1$, 
	then by \eqref{eq:chi_bound} again,
	\[
	\sum_{\Nm \mf b \le x} |\chi_s(\mf b)| 
	< \frac{1}{\Nm \mf R^\sigma} 
	\sum_{\mf b} \frac{1}{\Nm \mf b^\sigma}
	= \frac{\zeta_K(\sigma)}{\Nm \mf R^\sigma}
	\]
	regardless of $x$.
	
	It remains to prove part (i).
	By what we just showed,
	the series 
	\begin{equation} \label{eq:chi_series}
		\sum_{\mf b} \chi_s(\mf b)
	\end{equation}
	converges absolutely.
	Moreover, its tail decays like
	\[\bigg| \sum_{\Nm \mf b > x} \chi_s(\mf b) \, \bigg|
	\le \sum_{\Nm \mf b > x} \!\! |\chi_s(\mf b)|
	\le \frac{1}{\Nm \mf R^\sigma} \! \sum_{\Nm \mf b > x} \frac{1}{\Nm \mf b^\sigma}
	= O\Big(\frac{x^{1-\sigma}}{\Nm \mf R^\sigma}\Big)
	\]
	by \eqref{eq:chi_bound} and Lemma \ref{lem:weber}.
	To evaluate the sum \eqref{eq:chi_series},
	we use Euler products.
	Let 
	\[\xi_s(\mf b) := 
	\frac{\mu_K(\mf b) [\mf b \perp \mf R]}{\Nm \mf b^s}\]
	so that 
	\[
	\sum_{\mf b} \chi_s(\mf b)
	= \frac{1}{\Nm \mf R^s}
	\sum_{\mf b} \xi_s(\mf b).
	\]
	We claim that $\xi_s$ is a multiplicative function of $\mf b$. 
	Since the M\"obius function and the ideal norm are already multiplicative, 
	it suffices to prove our claim for the Iverson bracket $\mf b \mapsto [\mf b \perp \mf R]$.
	To that end, 
	suppose $\mf b$ and $\mf b'$ are coprime.
	Then 
	\begin{align*}
		(\mf b + \mf R)(\mf b' + \mf R) &= \mf {bb}' + \mf {bR} + \mf {Rb}' + \mf R^2 
		\\ &= \mf {bb}' + (\mf b + \mf b')\mf R + \mf R^2 
		\\ &= \mf {bb}' + \mf R + \mf R^2 
		\\ &= \mf {bb}' + \mf R,
	\end{align*}
	which means $\mf {bb}' + \mf R = \langle 1 \rangle$ if and only if $\mf b + \mf R = \langle 1 \rangle$ and $\mf b' + \mf R = \langle 1 \rangle$.
	Thus $\xi_s$ is indeed multiplicative,
	and therefore 
	\[\sum_{\mf b} \chi_s(\mf b)
	= \frac{1}{\Nm \mf R^s} \prod_{\mf p} \sum_{i=0}^\infty \xi_s(\mf p^i).
	\]
	Noting that $\mu_K(\mf p^i) = 0$ unless $i \le 1$, 
	and that $\mf p \perp \mf R$ if and only if $\mf p \nmid \mf R$,
	we calculate
	\[\xi_s(\mf p^i) = 
	\begin{dcases} 1 & i = 0, \\ 
		-{\Nm \mf p^{-s}} & i = 1 \text{ and } \mf p \nmid \mf R, \\ 0 & \text{otherwise.}
	\end{dcases}
	\]
	It follows (upon completing the zeta function) that
	\begin{align*}
		\sum_{\mf b} \chi_s(\mf b)
		= \frac{1}{\Nm \mf R^s}
		\prod_{\mf p \nmid \mf R} (1 - \Nm \mf p^{-s})
		= \frac{1}{\zeta_K(s) \underbrace{\textstyle \Nm \mf R^s \prod_{\mf p \mid \mf R} (1 - \Nm \mf p^{-s})}_{J_{\mathcal{O}_K,s}(\mf R)}}
	\end{align*}
	as desired.
\end{proof}

\section{Auxiliary results} \label{sec:aux}

\subsection{Lattices}

A \emph{lattice} is a discrete subgroup of Euclidean space.
Every lattice is a free abelian group 
and is thus isomorphic to $\Z^k$ for some integer $k$ called its \emph{rank}.
Let $\Lambda \subset \R^l$ be a lattice of rank $k$.
By discreteness, $k \le l$.
If $B = \begin{bmatrix} b_1 & \ldots & b_k\end{bmatrix}$ is an $l$-by-$k$ matrix whose columns constitute a $\Z$-basis for $\Lambda$, 
then the \emph{determinant} of $\Lambda$ 
is the number
\begin{align*} 
	\det \Lambda 
	&= \sqrt{\det B^T \! B}, \\
	&= |{\det B}| \quad\quad\quad \text{if $k = l$}.
\end{align*}
Note this is independent of the choice of $A$.
(Careful---$\det(BB^T) = 0$ if $k < l$.)
Geometrically, $\det \Lambda$ 
measures the $k$-volume of 
the \emph{fundamental parallelotope}
\[
B[0,1)^k = b_1[0, 1) + \ldots + b_k[0, 1)
\]
defined by $B$.
Every such parallelotope is a fundamental domain for the translation action of $\Lambda$ on $\Span_\R \{b_1, \ldots, b_k\}$.
The \emph{successive minima} of $\Lambda$
are the numbers 
\[\lambda_i = \inf \{\lambda \ge 0 : \text{$\Lambda$ contains $i$ linearly independent vectors of length at most $\lambda$}\}.\]
The first successive minimum $\lambda_1$ coincides with the length of the shortest nonzero vector in $\Lambda$.
By definition,
\[
\lambda_0 = 0 < \lambda_1 \le \ldots \le \lambda_k < \infty = \lambda_{k+1} = \ldots
\]
Beware that $\Lambda$ need not have a basis of vectors of length at most $\lambda_k$: 
the standard counterexample is
\[
\Z^l + \tfrac{1}{2}(1, \ldots, 1)\Z \subset \R^l \qquad (l \ge 5).
\]
Finally,
\emph{Minkowski's second theorem} asserts that for a lattice $\Lambda$ of rank $k$,
\begin{equation} \label{minkowski_minima}
	\frac{2^k \det \Lambda}{k! V_k}
	\le 
	\lambda_1 \ldots \lambda_k 
	\le 
	\frac{2^k \det \Lambda}{V_k}
\end{equation}
where 
\begin{equation} \label{vol_kball}
V_k = \frac{\pi^{k/2}}{\Gamma(k/2 + 1)}
\end{equation}
is the $k$-volume of the unit $k$-ball.

\subsection{Ideals and units}

\begin{propn}[Minkowski] \label{propn:Minkowski}
	The diagonal map 
	\begin{align*}
		\Phi : K &\to \prod_{v \mid \infty} K_v \ (\cong \R^n) \\
		x &\mapsto (x : v \mid \infty) = (x, \ldots, x)
	\end{align*}
	is a $K$-linear\footnote{i.e.,~$\Phi(x + y) = \Phi(x) + \Phi(y)$ for all $x, y$ in $K$ and $\Phi(\lambda x) = \lambda \Phi(x)$ for all $\lambda$ in $K$.} embedding.
	The image of any nonzero fractional ideal $\mf a$ of $K$ is a full-rank lattice with
\begin{equation*} 
	\det \Phi(\mathfrak{a}) = \frac{\Nm \mathfrak{a} \sqrt{d_K}}{2^{r_2}}.
\end{equation*}
\end{propn}

\begin{cor} \label{map_Psi}
The map
\begin{align*}
	\Psi : K^{m+1} &\to \prod_{v \mid \infty} K_v^{m+1} \ (\cong \R^{n(m+1)}) \\
	(x_0, \ldots, x_m) &\mapsto ((x_0, \ldots, x_m) : v \mid \infty)
\end{align*}
is a $K$-linear embedding. 
The image of any nonzero fractional ideal $\mf a$ of $K$ is a full-rank lattice with 
\[\det \Psi(\mf a) = \det \Phi(\mf a)^{m+1}.\]
\end{cor}

\begin{lem}
For all nonzero $\mf a$,
$\lambda_1(\Psi(\mf a)) \ge \Nm \mf a^{1/n}$.
\end{lem}
\begin{proof}[Proof (based on {\cite[Lemma 5]{MV}})]
It's clear that $\Psi(\mf a) \cong \Phi(\mf a)^{m+1}$ has the same first minimum as $\Phi(\mf a)$.	
	For each finite place $v \nmid \infty$ let $\mf p_v \subset \mathcal{O}_K$ be the corresponding prime ideal;
	for each infinite place $v \mid \infty$ 
	let $\sigma_v : K \into \C$ be the corresponding embedding. 
	If $0 \ne x \in \mf a$, 
	then 
	from the inclusion $x\mathcal{O}_K \subseteq \mf a$ and the product formula we have
	\[
	\Nm \mf a 
	\le 
	\Nm x\mathcal{O}_K 
	= 
	\prod_{v \nmid \infty} \Nm \mf p_v^{v(x)}
	=
	\prod_{v \mid \infty} |\sigma_v(x)|^{n_v}.
	\]
	By the AM-GM inequality with weights $n_v$, 
	\[\Nm \mf a^{2/n}
	\le \sqrt[n]{\prod_{v \mid \infty} |\sigma_v(x)|^{2n_v}}
	\le \frac{1}{n}\sum_{v\mid\infty} n_v |\sigma_v(x)|^2
	\le \frac{1}{n} \max_{v \mid \infty} n_v \sum_{v \mid \infty} |\sigma_v(x)|^2
	\] 
	and the claim follows on recognizing the final sum as the squared length of $\Phi(x)$.
\end{proof}

A typical element $z$ of the \emph{Minkowski--Schanuel space} 
\[
K_\R^{m+1} := 
\prod_{v \mid \infty} K_v^{m+1}\] 
shall be denoted 
by a $v$-tuple $z = (z_v : v \mid \infty)$ of $(m+1)$-tuples $z_v = (z_{v0}, \ldots, z_{vm})$ of elements of $K_v$.
For instance, the image of $\Psi$ is given by all elements of the form 
\begin{equation} \label{eq:image_MS}
z = (x_0, \ldots, x_m : v \mid \infty) \qquad (x \in K^{m+1}).
\end{equation}

\begin{propn}[Dirichlet] \label{map_lambda}
	Let 
	\begin{align*}
		\Lambda &: \mathcal{O}_K^\times \to \prod_{v \mid \infty} \R \ (\cong \R^{r+1}) \\
		u &\mapsto (\log {\lVert u \rVert_v} : v \mid \infty).
	\end{align*}
	Then $\Lambda$ is a group homomorphism
	whose kernel is the group $\mu(K)$ of roots of unity in $K$
	and whose image is a lattice of determinant $R \sqrt{r+1}$ contained in the hyperplane $\Pi = \{y : (1, \ldots, 1) \cdot y = 0\}$.
\end{propn}

\subsection{Schanuel's maps}

We now introduce Schanuel's maps $\rho$, $\eta$, and $\pr$, modified to accommodate $f$,
and record their fundamental properties.
It's worth keeping in mind that the Schanuel--Minkowski space $K_\R^{m+1}$ comes equipped with \emph{two} actions by scalar multiplication: by elements of $K$ and by elements of $\R = \Q_v \subseteq K_v$.

\begin{propn} \label{map_rho}
Define
	\begin{align*}
		\rho &: K_\R^{m+1} \to [0, \infty) \\ 
		z &\mapsto \prod_{v \mid \infty} \lVert F(z_v) \rVert_v^{1/d} \ \Big({=} \prod_{v \mid \infty} \max_{0 \le j \le M} |F_j(z_{v0}, \ldots, z_{vm})|_v^{n_v/d} \Big).
	\end{align*}
	Then for all $z$, we have: 
	\begin{enumerate}[(i)]
		\item $\rho(\lambda z) = |N_{K/\Q}(\lambda)| \rho(z)$ for all $\lambda$ in $K$;
		\item $\rho(tz) = |t|^n \rho(z)$ for all $t$ in $\R$.
	\end{enumerate}
	Moreover, 
	\begin{enumerate}[(i)] \setcounter{enumi}{2}
		\item $\rho(\Psi(x)) = H_\infty(F(x))^{1/d}$
		for all $x$ in $K^{m+1}$.
	\end{enumerate}
\end{propn}
\begin{proof}
Because $F$ is homogeneous of degree $d$, 
we have
	\[\frac{\rho(cz)}{\rho(z)} = \prod_{v \mid \infty} \lVert c^d \rVert_v^{1/d}
	= \begin{dcases}
		|N_{K/\Q}(c)| & c \in K \\
		|c|^n & c \in \R
	\end{dcases}\]
whence parts (i) and (ii).
For part (iii), if $x = (x_0, \ldots, x_m)$ 
then by \eqref{eq:image_MS}
	\[
	\rho(\Psi(x))
	= \prod_{v\mid\infty} \lVert F(x_0, \ldots x_m) \rVert_v^{1/d} = H_\infty(F(x))^{1/d}.\qedhere
	\]
\end{proof}

\begin{propn} \label{map_eta}
Define 
\begin{align*}
	\eta &: \prod_{v \mid \infty} (K_v^{m+1} - 0^{m+1})
	\to \prod_{v \mid \infty} \R \\
	z &\mapsto (\log{ \lVert F(z_v) \rVert_v^{1/d} } : v \mid \infty).
\end{align*}
Then for all $z$ in the domain of $\eta$, we have:
\begin{enumerate}[(i)]
	\item $\eta(uz) = \Lambda(u) + \eta(z)$ for all $u$ in $\mathcal{O}_K^\times$;
	\item $\eta(tz) = (n_v) \log |t| + \eta(z)$ for all nonzero $t$ in $\R$.
\end{enumerate}
\end{propn}
\begin{proof}
Since $F$ is homogeneous of degree $d$, 
\[
\log {\lVert F(cz_v) \rVert_v^{1/d}} 
= \log {\lVert c \rVert_v} + \log {\lVert F(z_v) \rVert_v^{1/d}}
\]
for each $v \mid \infty$ and each scalar $c$.
Part (i) follows from the definition of $\Lambda$, and part (ii) follows from the fact that $\lvert c \rvert_v = |c|$ for all $c$ in $\R$.
\end{proof}

\begin{propn}\label{map_pr}
Define
\begin{align*}
	\pr &: \prod_{v \mid \infty} \R \to \Pi \\
	y &\mapsto (y_v - \frac{n_v}{n} \sum_{w \mid \infty} y_w : v \mid \infty).
\end{align*}
Then:
\begin{enumerate}[(i)]
	\item $\pr$ is $\R$-linear,
	\item $\pr(y) = y$ for all $y$ in $\Pi$, and
	\item $(n_v) \in \ker \pr$.
\end{enumerate}
Moreover:
\begin{enumerate}[(i)] \setcounter{enumi}{3}
	\item We have the direct sum decomposition 
	\begin{align*}
		\prod_{v \mid \infty} \R &= \Pi \oplus \R(n_v) \\
		y &= \pr(y) + c(n_v)
		\shortintertext{where}
		c &= \frac{1}{n} \sum_{v \mid \infty} y_v.
	\end{align*}
\end{enumerate}
\end{propn}
\begin{rmk*}
	In terms of dot products, 
	\begin{equation} \label{map_pr_vector}
\pr(y) = y - \frac{(1, \ldots, 1) \cdot y}{(1, \ldots, 1) \cdot (n_v)} (n_v).
	\end{equation}
	Thus, $\pr$ is \emph{not} the orthogonal projection along the vector $(n_v)$---unless $K$ is totally real.
\end{rmk*}
\begin{proof}
Parts (i--iii) are immediate from \eqref{map_pr_vector}.
Now, parts (ii) and (iii) imply $\im \pr \supseteq \Pi$ and $\ker \pr \supseteq \R(n_v)$ respectively, so by the rank--nullity theorem 
\[\im \pr = \Pi \quad \text{and} \quad  \ker \pi = \R(n_v).\]
Since $(n_v) \not \in \Pi$ 
these subspaces are independent, i.e.,~$\Pi \cap \R (n_v)$ is trivial.
Thus 
\[\dim \big( {\Pi + \R(n_v)} \big) = \dim \Pi + \dim \R(n_v) - \dim \big(\Pi \cap \R(n_v)\big) = r + 1\]
whence
\[\prod_{v \mid \infty} \R = \Pi \oplus \R(n_v).\]
Explicitly, if $y$ is arbitrary, then 
\[y = \pr(y) + (y - \pr(y)) = \pr(y) + \Big(\!\underbrace{\frac{1}{n} \sum_{w\mid\infty} y_w}_{c}\!\Big) (n_v)\]
which proves (iv). 
\end{proof}

We conclude with a beautiful functional equation 
relating all three of Schanuel's maps,
which seems to have gone unnoticed.

\begin{cor} \label{map_trio}
For all $z$ in the domain of $\eta$, 
\[\eta(z) = \pr(\eta(z)) + \frac{\log \rho(z)}{n} (n_v).\]
\end{cor}
\begin{proof}
By Proposition \ref{map_pr}(iv),
\[\eta(z) = \pr(\eta(z)) + c(n_v)\]
where 
\[
c 
= \frac{1}{n} \sum_{v \mid \infty} \eta(z)_v
= \frac{1}{n} \sum_{v \mid \infty} \log {\lVert F(z_v) \rVert_v^{1/d} }
= \frac{1}{n} \log \rho(z). \qedhere
\]
\end{proof}

\begin{figure}[h]
	\[
	\begin{tikzcd}[row sep=large]
		K^{m+1} \arrow[r,"F"] \arrow[hookrightarrow,d,"\Psi"] & K^{M+1} \arrow[d,"H_\infty^{1/d}"]\\
		K_\R^{m+1} \arrow[r,"\rho"] &
		{[0, \infty)}
		\\
		\displaystyle \prod_{v \mid \infty} (K_v^{m+1} - 0^{m+1}) \arrow[hookrightarrow,u] \arrow[r,"\eta"] &
		\displaystyle \prod_{v \mid \infty} \R 
		\arrow[twoheadrightarrow,r,"\pr"]
		& 
		\Pi
		\\ 
		& \mathcal{O}_K^\times \arrow[u,"\Lambda"]
	\end{tikzcd}
	\]
	\caption{Diagram of the auxiliary maps involved in the proof of Theorem \ref{thm:main}.} \label{figaux}
\end{figure}

\subsection{Actions of the unit group}

Recall that $K^\times$ acts on $K^{m+1}$ and $K_\R^{m+1}$ 
by scalar multiplication; these actions are inherited by $\mathcal{O}_K^\times$.
Meanwhile, the Dirichlet embedding $\Lambda$
induces an action of $\mathcal{O}_K^\times$ on $\prod_{v\mid\infty} \R$
via 
\[
u\cdot y := \Lambda(u) + y.
\]

\begin{lem} \label{maps_equivar}
The maps $\Psi$, $\eta$, and $\pr$ are $\mathcal{O}_K^\times$-equivariant.
\end{lem}
\begin{proof}
Equivariance of $\Psi$ follows from its $K$-linearity (Corollary \ref{map_Psi}); equivariance of $\eta$ is Proposition \ref{map_eta}(i).
As for the projection, Proposition \ref{map_pr}(i, ii) implies
\[\pr(u \cdot y) = \pr(\Lambda(u) + y)
= \pr(\Lambda(u)) + \pr(y)
= \Lambda(u) + \pr(y)
= u \cdot \pr(y)
\]
for all units $u$ and all vectors $y$.
\end{proof}

\begin{propn} \label{counting_orbits}
Let $A$ be an $\mathcal{O}_K^\times$-invariant subset of
$K^{m+1} - 0^{m+1}$.
Then the number of orbits in $A$ is 
\[\big|A/\mathcal{O}_K^\times \big| = \frac{1}{w} \big|\Psi(A) \cap ({\pr} \circ \eta)^{-1}(D_K) \big|\]
where $w$ is the number of roots of unity in $K$
and $D_K$ is any fundamental domain for $\Lambda(\mathcal{O}_K^\times)$.
\end{propn}
\begin{rmk*}
This is extracted from Schanuel \cite[p.~438\textit{f}]{Schanuel}.
\end{rmk*}
\begin{proof}
First, since $\Psi$ is injective and $\mathcal{O}_K^\times$-equivariant,  
\begin{equation} \label{eq_orbit_count_1}
\big|A/\mathcal{O}_K^\times \big| = \big|\Psi(A)/\mathcal{O}_K^\times \big|.
\end{equation}
The hypothesis that $A$ does not contain $0^{m+1}$ 
implies that 
the action of $\mathcal{O}_K^\times$ is free
and that the co-restricted inclusion map 
\[\iota : \Psi(A) \to \prod_{v\mid\infty} K_v^{m+1} - 0^{m+1} \ \big(\! = \operatorname{dom} \eta \big)\]
is well-defined. Note that $\iota$ is obviously $\mathcal{O}_K^\times$-equivariant,
so by Lemma \ref{maps_equivar} the composite map ${\pr} \circ \eta \circ \iota$ is $\mathcal{O}_K^\times$-equivariant.

Let $D_K$ be any fundamental domain for the unit lattice.
Then $D_K$ is, in other words, 
fundamental (mod 1) for 
the regular action of $\Lambda(\mathcal{O}_K^\times)$ on the hyperplane $\Pi$. 
By Lemma \ref{schanuel_missing_lemma}, 
$D_K$ is fundamental (mod $\Lambda^{-1}(1) = \ker \Lambda = \mu(K)$)
for $\mathcal{O}_K^\times$ acting on $\Pi$.
Lemma \ref{schanuel_lemma_1} 
then says that $({\pr} \circ \eta \circ \iota)^{-1}(D_K) = \Psi(A) \cap ({\pr} \circ \eta)^{-1}(D_K)$ is fundamental (mod $\mu(K)$) for $\mathcal{O}_K^\times$ acting on $\Psi(A)$.
Thus by Proposition \ref{enlargement}, 
\begin{equation} \label{eq_orbit_count_2}
	\big| \Psi(A) / \mathcal{O}_K^\times \big| = \big|\big(\Psi(A) \cap ({\pr} \circ \eta)^{-1}(D_K)\big) / \mu(K) \big|.
\end{equation}
Now since $\mu(K)$, being a subgroup of $\mathcal{O}_K^\times$, acts freely on $\Psi(A)$, each orbit has cardinality $w$.
Therefore
\begin{equation} \label{eq_orbit_count_3}
	\big|\big(\Psi(A) \cap ({\pr} \circ \eta)^{-1}(D_K)\big) / \mu(K) \big| = \frac{1}{w} \big|\Psi(A) \cap ({\pr} \circ \eta)^{-1}(D_K) \big|.
\end{equation}
Combining \eqref{eq_orbit_count_1}, \eqref{eq_orbit_count_2}, and \eqref{eq_orbit_count_3} completes the proof.
\end{proof}

\subsection{The homogeneously expanding domain}
During the proof of Theorem \ref{thm:main}, 
we will need to count $\mathcal{O}_K^\times$-orbits 
in sets of the form $A = B \cap \Psi^{-1}(D_F(T))$ 
where 
\begin{equation} \label{def:DF}
	D_F(T)
	:= 
	\rho^{-1}(-\infty, T].
\end{equation}
By Proposition \ref{counting_orbits},
this is equivalent to 
counting how many points of $\Psi(B)$ lie in
\begin{equation} \label{def:DFK}
D_{F,K}(T) := D_F(T) \cap ({\pr} \circ \eta)^{-1}(D_K).
\end{equation}
In light of this,
the next few results focus on the geometry of this domain.

\begin{propn} \label{properties_DFS}
Let $S \subseteq \Pi$ be any set
and let $T > 0$ be a real number.
Put 
\[D_{F,S}(T) = D_F(T) \cap ({\pr} \circ \eta)^{-1}(S).\]
Then:
\begin{enumerate}[(i)]
	\item $D_{F,S}(T) = T^{1/n} D_{F,S}(1)$;
	\item $D_{F,S}(1) = \eta^{-1}(S + (-\infty, 0](n_v))$.
\end{enumerate}
\end{propn}
\begin{rmk*}
Part (i) is the analogue of \cite[Lemma 3]{Schanuel}.
Part (ii) can be gleaned from the proof of \cite[Lemma 4]{Schanuel} and was essentially taken as the definition of $D_{F,S}$ in \cite[p.~439]{MV} (\textit{alias} $S_F(T)$).
\end{rmk*}
\begin{rmk*}
	What if $T = 0$? 
	Because $F$ has no nontrivial zeroes, 
	\[D_F(0) = \{z \in K_\R^{m+1} : z_v = 0 \text{ for some } v \mid \infty\}.\]
	However, this set is disjoint from $\dom \eta$, 
	so $D_{F,S}(0) = \varnothing$.
	Thus, part (i) is false (as $0D_F(1) = \{0\}$),
	as is part (ii)---unless we interpret $(-\infty, \log 0]$ as $\varnothing$.
\end{rmk*}
\begin{proof}
	By Proposition \ref{map_rho}(ii), $\rho^{-1}$ is homogeneous of degree $1/n$.
	Hence
\begin{equation} \label{DFY_DF1}
	D_F(T) 
	= \rho^{-1}(T(-\infty, 1]) 
	= T^{1/n} \rho^{-1}(-\infty, 1] 
	= T^{1/n} D_F(1).
\end{equation}
Meanwhile,
if $z \in \dom \eta$ 
and $t \in \R^\times$
then 
Propositions \ref{map_eta}(ii)
and \ref{map_rho}(i,iii) imply
\[
\pr(\eta(tz))
=
\pr((n_v) \log |t| + \eta(z))
= 
\pr((n_v) \log |t|) + \pr(\eta(z))
= 
\pr(\eta(z)).
\]
Thus, ${\pr} \circ \eta$ is $\R^\times$-invariant.
It follows that 
\begin{equation} \label{Y_pr_eta}
T^{1/n} ({\pr} \circ \eta)^{-1}(S)
=
({\pr} \circ \eta)^{-1}(S).
\end{equation}
Combining \eqref{DFY_DF1} 
and \eqref{Y_pr_eta}, 
proves part (i):
\begin{align*}
T^{1/n} D_{F,S}(1)
&= 
T^{1/n} \left( D_F(1) \cap ({\pr} \circ \eta)^{-1}(S) \right) \\
&= 
T^{1/n} D_F(1) \cap T^{1/n} ({\pr} \circ \eta)^{-1}(S) 
\\
&= D_F(T) \cap ({\pr} \circ \eta)^{-1}(S) \\
&= D_{F,S}(T).
\end{align*}
As for part (ii),
if $z \in D_{F,S}(1)$
then $\rho(z) \le 1$ and $\pr(\eta(z)) \in S$,
so by the functional equation (Corollary \ref{map_trio})
\[\eta(z) 
= \underbrace{\pr(\eta(z))}_{\in S} + \underbrace{\frac{\log \rho(z)}{n}}_{\le 0} (n_v)
\in S + (-\infty, 0](n_v).
\]
Conversely,
if $\eta(z) = y + c(n_v)$ for some $y \in S$ and $c \le 0$ 
then by uniqueness of the direct sum decomposition (Proposition \ref{map_pr}(iv))
$\pr(\eta(z)) = y \in S$ and $\rho(z) = e^{cn} \le 1$.
\end{proof} 

\begin{cor} \label{D_FS_is_bounded}
If $S$ is bounded, then $D_{F,S}(1)$ is bounded.
\end{cor}
\begin{rmk*}
This is the analogue of \cite[Lemma 4]{Schanuel}.
\end{rmk*}
\begin{proof}
Suppose $\lVert y \rVert_2 \le B$ for all $y \in S$.
Then also
\[y_v \le |y_v| \le \lVert y \rVert_\infty \le B\]
for all $v \mid \infty$.
Let $z \in D_{F,S}(1)$.
By Proposition \ref{properties_DFS}(ii),
$\eta(z)$
is the sum of a vector in $S$ and a vector with non-positive coordinates. 
Thus $\eta(z)_v \le B$ as well.
But now by \eqref{intro_constants_1},
\[
e^B 
\ge e^{\eta(z)_v}
= \lVert F(z_v) \rVert_v^{1/d} 
\ge \frac{ \lVert z_v \rVert_v \lVert F \rVert_v^{1/d} }{ C_v^{n_v} }
\]
so, taking $n_v$\textsuperscript{th} roots,
\[
|z_{vi}|_v \le |z_v|_v \le C_v e^{B/n_v} |F|_v^{-1/d}
\]
for all $0 \le i \le m$ and all $v \mid \infty$.
Therefore by \eqref{error_constants},
\[
\lVert z \rVert_2
= 
\bigg(\sum_{v \mid \infty} 
\sum_{i=0}^m |z_{vi}|_v^2
\bigg)^{1/2}
\le
\sqrt{(r+1)(m+1)} C_f^{\infty} e^B \max_{v \mid \infty} |F|_v^{-1/d}.
\qedhere
\]
\end{proof}

The following archimedean volume computation 
is a generalization of \cite[p.~443\textit{f}]{Schanuel}
and a special case of \cite[Lemma 4]{MV};
we offer a bit more detail in the proof.
Note we have assumed $\bdy D_{F,K}(1)$ is measurable;
this will follow from 
Corollary \ref{counting_principle}
and Lemma \ref{exists_good_DFK} below.

\begin{propn} \label{volume_DFK}
Let $D_K$ be any fundamental parallelotope for the unit lattice,
and put 
\[D_{F,K}(1) = D_F(1) \cap ({\pr} \circ \eta)^{-1}(D_K).\]
Then, under the canonical identification $K_\R^{m+1} \cong \R^{n(m+1)}$, we have
\[
\vol D_{F,K}(1) = (m+1)^r R \, c_{K,\infty}(f) \prod_{v \mid \infty} \lVert F \rVert_v^{-(m+1)/d}
\]
\end{propn}
\begin{proof}
Let $\mu_v$ denote the Haar measure on $K_v^{m+1}$.
By ``$\vol D_{F,K}(1)$'' we mean $\mu (D_{F,K}(1))$ 
where 
\begin{equation} \label{mu_defn}
\mu := \bigotimes_{v \mid \infty} \mu_v
\end{equation}
is the measure on $K_\R^{m+1}$ 
corresponding to the Lebesgue measure on $\R^{n(m+1)}$.
By Proposition \ref{properties_DFS}(ii)
and the change-of-variables\footnote{If 
	$f : X \to Y$ is measurable, 
	then 
	\[\int_X f^* g \, \d \mu = \int_Y g \, \d f_*\mu.\]
} 
formula,
\begin{equation} \label{vol_DFK2}
	\vol D_{F,K}(1) = \int_{D_{F,K}(1)} \d\mu
	= 
	\int_{D_K + (-\infty, 0](n_v)} \d (\eta_* \mu).
\end{equation}
We claim that
the pushforward measure $\eta_* \mu$ on $\R^{r+1}$
decomposes as a product of the one-dimensional measures $(\eta_v)_* \mu_v$ on $\R$,
where
\begin{align*}
	\eta_v &: K_v^{m+1} - 0^{m+1} \to \R \\
	     z &\mapsto \log {\lVert F(z) \rVert_v^{1/d}}.
\end{align*}
This is purely formal:
since $\eta(z)_v = \eta_v(z_v)$,
\begin{equation} \label{cyl_sets}
\eta^{-1}\big(\prod_{v \mid \infty} E_v\big) = \prod_{v \mid \infty} \eta_v^{-1}(E_v)
\end{equation}
for all cylinder sets $\prod_{v\mid\infty} E_v$;
it follows from \eqref{mu_defn} and \eqref{cyl_sets}
that 
\[
(\eta_* \mu) \big( \prod_{v\mid\infty} E_v \big)
= \mu \big( \prod_{v \mid \infty} \eta_v^{-1}(E_v) \big) 
= \prod_v \big( (\eta_v)_* \mu_v \big) (E_v)
\]
Thus 
\begin{equation} \label{vol_DFK3}
	\eta_* \mu = \bigotimes_v (\eta_v)_* \mu_v.
\end{equation}
To identify what the measure $(\eta_v)_* \mu_v$ is,
we calculate its distribution function, making key use of homogeneity:
\begin{align*}
((\eta_v)_* \mu_v)(-\infty, \xi] 
&= \mu_v(\eta_v^{-1}(-\infty, \xi]) \\
&= \mu_v(\{z \in K_v^{m+1} - 0^{m+1} : \log {\lVert F(z) \rVert_v^{1/d}} \le \xi\}) \\
&= \mu_v(\{z \in K_v^{m+1} : |F(z)|_v \le e^{\xi d/n_v}\}) \\
&= \mu_v(e^{\xi/n_v} |F|_v^{-1/d} \{z \in K_v^{m+1} : |F(z)|_v \le |F|_v\}) \\
&= \big \lVert e^{\xi/n_v} |F|_v^{-1/d} \big \rVert_v^{m+1} \mu_v (D_{f,v}) \\
&= e^{\xi(m+1)} \lVert F \rVert_v^{-(m+1)/d} c_{K,v}(f).
\end{align*}
In particular, $(\eta_v)_* \mu_v$
is absolutely continuous,
and its Radon--Nikodym derivative 
with respect to the Lebesgue measure on $\R$
is
\begin{equation} \label{vol_DFK4}
\frac{\d}{\d \xi} \frac{c_{K,v}(f)}{\lVert F \rVert_v^{(m+1)/d}} e^{\xi(m+1)} = (m+1) \frac{c_{K,v}(f)}{\lVert F \rVert_v^{(m+1)/d}} e^{\xi(m+1)}.
\end{equation}
Thus from \eqref{vol_DFK2}, \eqref{vol_DFK3} and \eqref{vol_DFK4}, 
we have
\[\vol D_{F,K}(1) 
=
(m+1)^{r+1} \prod_{v \mid \infty} \frac{c_{K,v}(f)}{\lVert F \rVert_v^{(m+1)/d}}
\cdot 
\int_{D_K + (-\infty, 0](n_v)} 
e^{(m+1) \sum_v y_v} \, \d y.\]
To evaluate the integral,
we perform a substitution.
Let $u_1, \ldots, u_r \in \mathcal{O}_K^\times$ be the fundamental units defining $D_K$, 
and let 
\[y = \sum_{i=1}^r x_i \Lambda(u_i) + t(n_v) \qquad 0 < x_i < 1, \quad -\infty < t < 0.\]
The integrand becomes $e^{n(m+1)t}$, 
and
the Jacobian determinant is
\[
\frac{\partial y}{\partial(x, t)}
= 
\left| 
\begin{matrix}
\Lambda(u_1) & \ldots & \Lambda(u_r) & (n_v)
\end{matrix} 
\right|.
\]
Note that every column sums to $0$ except the last, which sums to $n$.
Adding one copy of each of the first $r$ rows to the last one yields a row of the form $\begin{bmatrix}0 & \ldots & 0 & n\end{bmatrix}$. 
Since these operations do not alter the determinant, we conclude (expanding along the new bottom row) that the Jacobian is $\pm nR$.
Therefore, 
\begin{align*}
\int_{D_K + (-\infty, 0](n_v)} 
e^{(m+1) \sum_v y_v} \, \d y
&= nR \int_0^1 \ldots \int_0^1 \int_{-\infty}^0 e^{n(m+1)t} \, dt \, dx_1 \ldots dx_r \\
&= nR \cdot 1 \cdot \ldots \cdot 1 \cdot \frac{1}{n(m+1)}. \qedhere 
\end{align*}
\end{proof}

\section{Geometry of numbers} \label{sec:geometry}

\subsection{Counting lattice points}
A fundamental counting principle in the geometry of numbers is that
the number of integer lattice points in a plane region $D$
is approximately equal to the area of $D$,
up to an error proportional to the perimeter of $D$.
If the boundary of $D$ 
can be covered by $N$ curves
each of length at most $L$,
then the perimeter of $D$ is at most $NL$.
It's easy to see that 
the length of a rectifiable curve $C$
is no greater than any Lipschitz constant 
for any map from the unit interval onto $C$
(in fact, the length is the infimum of said Lipschitz constants).
In this manner,
Lipschitz parametrizability provides a robust handle on the error term in lattice-point counting problems.

The following definition of Lipschitz parametrizability is adapted from Masser--Vaaler \cite[p.~431]{MV} and Widmer \cite[Definition 2.2]{Widmer_class}.

\begin{defn}
A subset of $\R^k$
is said to be of \emph{Lipschitz class} $(N, L)$ 
(\emph{in codimension $o$})
if it can be covered by the images of 
$N$ maps $\phi_i : [0,1]^{k-o} \to \R^k$
such that
\[
\Lip(\phi_i) := \sup_{x \ne y} \frac{\lVert \phi_i(x) - \phi_i(y) \rVert}{\lVert x - y \rVert} \le L \qquad (i = 1, \ldots, N)
\]
where $\lVert \, \cdot \, \rVert$ is the Euclidean norm.
\end{defn}

Henceforth we will neglect to explicitly specify the codimension when it is 1.
With these conventions in hand, 
we can now state a version of the fundamental counting principle due to Widmer \cite[Theorem 5.4]{Widmer_prim}, 
modified to suit our purposes.

\begin{propn} \label{propn:geometry_of_numbers}
	Let $z$ be a point,
	$\Lambda$ a full-rank lattice,
	and $D$ a bounded set, all in $\R^k$.
	Suppose every nonzero vector in $\Lambda$ has length at least $\lambda > 0$, 
	and suppose $\bdy D$ is of Lipschitz class $(N, L)$.
	Then $D$ is measurable, and
	\[
	|(z + \Lambda) \cap tD| = \frac{\vol D}{\det \Lambda} t^k
	+ O\big(N (1 + (tL/\lambda)^{k-1})\big)
	\]
	for all $t \ge 0$,
	where the implicit constant depends only on $k$.
\end{propn}
\begin{proof}
Already Spain \cite{Spain} observed the measurability of bounded sets with ``piecewise Lipschitzable'' boundary;
thus $D$ is measurable.
	Obviously
	\[
		|(z + \Lambda) \cap tD| = |\Lambda \cap (tD - z)|
	\]
	and,
	by translation-invariance and homogeneity of the Lebesgue measure,
	\[
		\vol {(tD - z)} = \vol {tD} = t^k \vol D.
		\] 
	Moreover, the boundary of $tD - z$, namely
	\[\bdy (tD - z) = \bdy (tD) - z = t (\bdy D) - z,\]
	is clearly of Lipschitz class $(M, tL)$: simply scale and shift the maps for $\bdy D$; their number is unchanged.
	Thus by \cite[Theorem 5.4]{Widmer_prim}, 
	\[
		\left| |(z + \Lambda) \cap tD| - \frac{\vol D}{\det \Lambda} t^k \right| 
		= 
		\left| |\Lambda \cap (tD - z)| - \frac{\vol {(tD - z)}}{\det \Lambda} \right|
		\le 
		N k^{3k^2/2} \max_{0 \le j < k} \frac{(tL)^j}{\lambda_1 \ldots \lambda_j}
	\] 
	where $\lambda_1 \le \ldots \le \lambda_k$ are the successive minima of $\Lambda$.
	By hypothesis, $\lambda_i \ge \lambda$ for all $i$,
	so 
	\[
	\frac{(tL)^j}{\lambda_1 \ldots \lambda_j}
	= 
	\prod_{i=1}^j tL/\lambda_i
	\le 
	(tL/\lambda)^j
	\]
	for each $0 \le j < k$.
	Thus 
	\[
		\max_{0 \le j < k} 
		\frac{(tL)^j}{\lambda_1 \ldots \lambda_j} 
		\le 
		\max_{0 \le j < k} (tL/\lambda)^j
		= 
		\max \{1, (tL/\lambda)^{k-1}\}
		\le 
		1 + (tL/\lambda)^{k-1}
		\]
	and the Proposition follows.
\end{proof}

\begin{cor} \label{counting_principle}
	Let $x \in K^{m+1}$ be a point 
	and let $\mf c \subset K$ be a nonzero fractional ideal.
	Let $D \subset \prod_{v \mid \infty} K_v^{m+1} \cong \R^{n(m+1)}$ be a bounded set whose boundary is of Lipschitz class $(N, L)$. 
	Then $D$ is measurable, and
	\[
	|\Psi(x + \mf c^{(m+1)}) \cap tD|
	= \frac{2^{r_2(m+1)} \vol D}{\sqrt{d_K}^{m+1} \Nm \mf c^{m+1}}
	t^{n(m+1)}
	+ O\Big(N \big(1 + \frac{(tL)^{n(m+1)-1}}{\Nm \mf c^{m+1-1/n}}\big)\Big)
	\]
	for all $t \ge 0$,
	where the implicit constant depends only on $n$ and $m$.
\end{cor}
\begin{proof}
	Apply Proposition \ref{propn:geometry_of_numbers} 
	with $k = n(m+1)$, $z = \Psi(x)$, $\Lambda = \Psi(\mf c^{(m+1)})$, and $\lambda = \Nm \mf c^{1/n}$ (cf.~Proposition \ref{propn:Minkowski} and its Corollary).
\end{proof}

As alluded to in the previous section, 
the proof of Theorem \ref{thm:main} involves counting lattice points in certain dilations of $D_{F,K}(1)$.
By Corollary \ref{D_FS_is_bounded}, $D_{F,K}(1)$ is bounded for any choice of fundamental parallelotope $D_K$;
so, to apply Corollary \ref{counting_principle}, 
we just have to show that $\bdy D_{F,K}(1)$ is 
Lipschitz parametrizable.
Of course, this is immediate from \cite[Lemma 3]{MV} of Masser--Vaaler, 
but their result yields no information about the Lipschitz class.
The latter was supplied in a heroic calculation by Widmer \cite[Lemma 7.1]{Widmer_prim}, 
which we shall very gratefully use to obtain the explicit error term advertised in Theorem \ref{thm:main}.

The application of Widmer's lemma requires two major ingredients:
\begin{enumerate}
	\item an estimate for the Lipschitz class of $\bdy D_{f,v}$ for each $v \mid \infty$;
	\item a ``well-rounded'' fundamental parallelotope for $\Lambda(\mathcal{O}_K^\times)$.
\end{enumerate}
These are prepared in the next two sections.

\subsection{Local fundamental domains}

We begin with a geometric interpretation 
of some invariants associated to $f$.

\begin{propn} \label{propn:kappa}
Let $K$ be a field
with an absolute value $|\cdot|_v$
and let $K_v$ denote the completion.
Let $f : \P^m \to \P^M$ 
be a morphism of degree $d > 0$
defined over $K$
and let $F$ be a lift of $f$.
Let
\[
D_{f,v} = \{z \in K_v^{m+1} : |F(z)|_v \le |F|_v\}
\]
be the local fundamental domain for $f$,
and define
\begin{align*}
\kappa_{F,v} &:= \sup \{\kappa \ge 0 : |F(z)|_v^{1/d} \ge \kappa |z|_v \textup{ for all } z \in K_v^{m+1}\} 
\\
R_{f,v} &:= \inf \{R \le \infty : |z|_v \le R \textup{ for all } z \in D_{f,v}\}.
\end{align*} 
Then:
\begin{enumerate}[(i)]
	\item $\displaystyle \kappa_{F,v} = \inf_{\substack{z \in K_v^{m+1} \\ z \ne 0}} \frac{|F(z)|_v^{1/d}}{|z|_v}$, 
	\item $\displaystyle R_{f,v} = \sup_{z \in D_{f,v}} |z|_v$.
\end{enumerate}
Moreover, putting $\kappa_{f,v} := |F|_v^{-1/d} \kappa_{F,v}$, 
the following hold.
\begin{enumerate}[(i)] \setcounter{enumi}{2}
	\item If $v$ is archimedean or trivial, then $R_{f,v} = \kappa_{f,v}^{-1}$.
	\item If $v$ is nonarchimedean, 
	then 
	\[R_{f,v} = q^{\lfloor \lVert \varepsilon_{f,v} \rVert / d \rfloor} \quad \text{while} \quad \kappa_{f,v}^{-1} = q^{\lVert \varepsilon_{f,v} \rVert / d}\]
assuming $|\cdot|_v = q^{-v(\cdot)}$ for some valuation $v$ and some real number $q > 1$.
\end{enumerate}
\end{propn}
\begin{proof}
To ease notation, omit the subscript $v$.
Parts (i) and (ii) are straightforward.
For part (iii),
note that 
$|F(z)|^{1/d} \ge \kappa_F |z|$ for all $z \ne 0$
by part (i),
and plainly this holds for $z = 0$ as well.
In particular, if $z \in D_f$, 
then
\[|z| \le \kappa_F^{-1} |F(z)|^{1/d} \le \kappa_F^{-1} |F|^{1/d} = \kappa_f^{-1}.\]
Thus $R_f \le \kappa_f^{-1}$ by part (ii).

Now if $|\cdot|$ is trivial then $R_f = 1 = \kappa_f^{-1}$.
Moving on, suppose $|\cdot|$ is archimedean 
and let $\varepsilon > 0$. 
By part (i) there exists $z_0 \ne 0$ such that 
	\[\frac{|F(z_0)|^{1/d}}{|z_0|} < \kappa_F + \varepsilon.\]
Since the l.h.s.~is homogeneous of degree $d/d - 1 = 0$, it remains unchanged if we scale $z_0$. 
Given that the value group $|K_v^\times| = (0, \infty)$ is divisible,
there exists a scalar $\lambda$ such that $|\lambda|^d = |F|/|F(z_0)|$. 
With $z_1 := \lambda z_0$ 
we observe by homogeneity 
that 
\[\frac{|F(z_1)|^{1/d}}{|z_1|} < \kappa_F + \varepsilon\]
and
\[|F(z_1)| = |F(\lambda z_0)| = |\lambda^d F(z_0)| = |\lambda|^d |F(z_0)| = |F|.\]
In particular, $z_1 \in D_f$, 
whence by part (ii)
\[R_f \ge |z_1| > \frac{|F(z_1)|^{1/d}}{\kappa_F + \varepsilon} = \frac{|F|^{1/d}}{\kappa_F + \varepsilon}.\]
Since $\varepsilon$ was arbitrary, $R_f \ge |F|^{1/d} \kappa_F^{-1} = \kappa_f^{-1}$.
Together with the preceding paragraph, 
this proves part (iii).
	
Part (iv) is proved by direct calculation.
By definition of the excess valuation,
\[|F(z)| = |z|^d |F| q^{-\varepsilon_f(z)}\]
for all $z \ne 0$;
combining this with part (i) yields
\[\kappa_f = |F|^{-1/d} \inf_{z \ne 0} \frac{|F(z)|^{1/d}}{|z|} 
= \inf_{z \ne 0} q^{-\varepsilon_f(z)/d}
= q^{-\lVert \varepsilon_f \rVert / d}.
\]
Meanwhile, 
the partition
\eqref{eq:Dfv_decomposition}
and scale-invariance of $\varepsilon_f$ 
imply
\[\{|z| : z \in D_f\}
= \{0\} \sqcup 
\!\! \bigsqcup_{i \in \im \varepsilon_f} \!\!
\{q^j : j \le i/d \text{ and } j \in \Z\}
\]
whence by part (ii)
\[
R_f 
= \sup_{i \in \im \varepsilon_f} q^{\lfloor i/d \rfloor}
= q^{\lfloor \lVert \varepsilon_f \rVert / d \rfloor}. \qedhere
\]
\end{proof}

Next, we prove a lemma concerning the parametrizability of level sets of homogeneous functions.

\begin{lem} \label{level_sets}
Let $\theta : \R^k \to \R$ be a continuous function which is continuously differentiable on $\R^k \setminus \theta^{-1}(0)$. 
Suppose $\theta$ is positive-homogeneous, i.e.,~there exists some $r \ne 0$ such that 
\begin{equation} \label{homog}
\theta(tx) = t^r \theta(x)
\end{equation}
for all $x \in \R^k$ and all $t > 0$.
Then every point of $\theta^{-1}(1)$ has a relatively open neighbourhood contained in the image of some continuously differentiable map $[0, 1]^{k-1} \to \R^k$.
\end{lem}
\begin{proof}
Differentiating \eqref{homog} at $t = 1$ (via chain rule)
shows that $\theta$ satisfies Euler's equation
\[
	\grad \theta(x) \cdot x = r \theta(x)
\]
for all $x \not \in \theta^{-1}(0)$.
In particular, if $a \in \theta^{-1}(1)$ then $\grad \theta(a) \ne 0$
(n.b.~it has nonzero dot product with $a$).
Without loss of generality, 
\[\frac{\del \theta}{\del x_k}(a) \ne 0.\]
It follows that the continuously differentiable map 
\begin{align*}
	\vartheta : \R^k \setminus \{0\} &\to \R^k \\
	x &\mapsto (x_1, \ldots, x_{k-1}, \theta(x))
\end{align*}
has invertible derivative at $a$, 
so by the Inverse Function Theorem 
there exist open neighbourhoods $U$ of $a$ and $V$ of $\vartheta(a) = (a_1, \ldots, a_{k-1}, 1)$ 
such that $\vartheta : U \to V$ is bijective and $\vartheta^{-1} : V \to U$ is continuously differentiable.
Let $\varepsilon > 0$ be such that 
\[V_0 := 
\{x \in \R^k : |x_i - a_i| \le \varepsilon \text{ for all } i < k \text{ and } |x_k - 1| \le \varepsilon\} \subseteq V\]
and let $U_0$ be the preimage of the interior of $V_0$ under $\vartheta$. 
Then $U_0$ is an open neighbourhood of $a$.
Lastly, define
\begin{align*}
	\Psi : [0, 1]^{k-1} &\to \R^k \\
	(t_1, \ldots, t_{k-1}) &\mapsto \vartheta^{-1}(a_1 + (2t_1 - 1)\varepsilon, \ldots, a_{k-1} + (2t_{k-1} - 1)\varepsilon, 1).
\end{align*}
Then $\Psi$ is continuously differentiable.
Now if $x \in \theta^{-1}(1) \cap U_0$ 
then $\theta(x) = 1$ and $\vartheta(x) = (x_1, \ldots, x_{k-1}, 1) \in V_0$, 
meaning $x_i \in (a_i - \varepsilon, a_i + \varepsilon)$ for all $i = 1, \ldots, k-1$. 
Hence there exist $t_i \in (0, 1)$ such that $x_i = a_i + (2t_i - 1)\varepsilon$ for each $i$,
and so
\[\Psi(t_1, \ldots, t_{k-1}) = \vartheta^{-1}(x_1, \ldots, x_{k-1}, 1) = x \in \Psi([0, 1]^{k-1}). \qedhere\]
\end{proof}

\begin{propn} \label{local_fundamental_domains}
Let $K$ be a number field 
and let
	\[D_{f,v} = \{z \in K_v^{m+1} : |F(z)|_v \le |F|_v\}\]
be the local fundamental domain for $f$
at some archimedean place $v$.
Then: 
\begin{enumerate}[(i)]
	\item $D_{f,v}$ is compact, and
	\item $\bdy D_{f,v}$ 
	is of finite Lipschitz class.
\end{enumerate}
\end{propn}

\begin{proof} \hfill 
\begin{enumerate}[(i)]
		\item The set $D_{f,v}$ is closed because the map $z \mapsto |F(z)|_v$ is continuous.
		By \eqref{intro_constants_1} 
		and Proposition \ref{propn:kappa}, 
		\[|z|_v \le R_{f,v} = \kappa_{f,v}^{-1} = \sup_{z \ne 0} \frac{|z|_v |F|_v^{1/d}}{|F(z)|_v^{1/d}} \le C_v < \infty\]
		for all $z \in D_{f,v}$. 
		Thus $D_{f,v}$ is bounded.
		\item 
This argument is based on \cite[Lemma 8]{Schanuel}.
Rescaling $F$ if necessary, we can assume $|F|_v = 1$. 
Recall that the boundary of an intersection is contained in the union of the boundaries.
Here, 
\begin{align*}
D_{f,v} 
&= \{z \in K_v^{m+1} : \max_{0 \le j \le M} |F_j(z)|_v \le 1\} 
= \bigcap_{j=0}^M \{z \in K_v^{m+1} : |F_j(z)|_v \le 1\}
\shortintertext{so, by \cite[IV.2, Theorem I.A, p.~105]{Cassels},}
\bdy D_{f,v} &\subseteq \bigcup_{j=0}^M \bdy \{z \in K_v^{m+1} : |F_j(z)|_v \le 1 \} 
= \bigcup_{j=0}^M \{z \in K_v^{m+1} : |F_j(z)|_v = 1\}.
\end{align*}
In particular, 
for each $z$ in $\bdy D_{f,v}$ 
there exists $j \in \{0, \ldots, M\}$ 
such that $|F_j(z)|_v^2 = 1$.
But since $F$ is defined by 
homogeneous polynomials of degree $d > 0$,
the function 
$\theta(x) := |F_j(x)|_v^2$ is real-analytic and positive-homogeneous of degree $2d > 0$.
Thus, Lemma \ref{level_sets} applies:
for each $z$ in $\bdy D_{f,v}$ 
there exists an open set $U_z$ and a $C^1$ map $\Psi_z : [0, 1]^{n_v(m+1)-1} \to K_v^{m+1}$ 
such that $z \in U_z \cap \bdy D_{f,v} \subseteq \im \Psi_z$.
Since $[0, 1]^{n_v(m+1) - 1}$ is convex, 
a standard argument involving the Mean Value Theorem 
and the Cauchy--Schwarz inequality 
shows that $\Lip(\Psi_z) < \infty$
for each $z$.
Now, $\bdy D_{f,v}$ is covered by the open sets $U_z$,
so by compactness
there exists a finite subcover, indexed (say) by $z_1, \ldots, z_N$.
Taking the corresponding maps $\Psi_{z_1}, \ldots, \Psi_{z_N}$ does the trick. 
\qedhere 
\end{enumerate}
\end{proof}

\subsection{``Well-rounded'' parallelotopes} \label{sec:wellrounded}

Let $K$ be a number field
with degree $n$, 
regulator $R$, and unit rank $r$.
A fundamental parallelotope $D_K$ for the unit lattice  $\Lambda(\mathcal{O}_K^\times)$ 
will be called \emph{well-rounded}
if either:
\begin{enumerate}[(a)]
	\item $r = 0$ and $D_K = \{0\}$, or
	\item $r > 0$ and 
	\begin{enumerate}[(i)]
		\item $D_K$ is contained in a Euclidean ball of radius $\ll_n R$, and
		\item $\bdy D_K$ is of Lipschitz class $\ll_n (1, R)$ in codimension 2.		
	\end{enumerate}
\end{enumerate}
The existence of well-rounded fundamental parallelotopes 
will be a consequence of the following three Lemmas.

\begin{lem} \label{lattices_c1_c2}
	There exist increasing functions $c_1(k)$ and $c_2(k)$
	such that every lattice $\Lambda \subset \R^l$ 
	of rank $k \ge 1$ 
	admits a fundamental parallelotope
	$D$
	whose every member has length at most $c_1(k) \lambda_k$
	and whose boundary is 
of Lipschitz class $(2k, c_2(k) \lambda_k)$ in codimension $l-k+1$.
\end{lem}

\begin{proof}
	Write $\lVert x \rVert$ for the length of $x$,
	and $B(0, r)$ for the ball of radius $r$.
	By definition,
	$\Lambda$
	contains $k$ linearly independent vectors 
	$a_1, \ldots, a_k$
	of length at most $\lambda_k$.
	Reordering if necessary, 
	we may assume $\lVert a_1 \rVert \le \ldots \le \lVert a_k \rVert$.
	By the Mahler--Weyl lemma \cite[V.4, Lemma 8, p.~135\textit{f}]{Cassels},
	there exists a basis $b_1, \ldots, b_k$ of $\Lambda$ 
	such that
	\[\lVert b_i \rVert \le \max\{\lVert a_i \rVert, \frac{1}{2}(\lVert a_1 \rVert + \ldots + \lVert a_i \rVert)\}
	\le \max\{1, \frac{i}{2}\} \lVert a_i \rVert 
	\le k \lambda_k
	\]
	for all $i = 1, \ldots, k$.
	Define
	\[D := [0, 1) b_1 + \ldots + [0, 1) b_k.\]
	First, if $t \in [0, 1)^k$ then 
	\[
	\Big\lVert \sum_{i=1}^k t_i b_i \Big\rVert < \sum_{i=1}^k \lVert b_i \rVert \le \sum_{i=1}^k k \lambda_k
	\]
	so that we may take $c_1(k) := k^2$.
	Next, $\bdy D$ may be parametrized by the maps 
	\[\phi_{j,\epsilon}(t) = \sum_{i=1}^{j-1} t_i b_i + \epsilon b_j + \sum_{i=j}^{k-1} t_i b_{i+1} \quad\qquad t = (t_1, \ldots, t_{k-1}) \in [0,1]^{k-1}\]
	for $1 \le j \le k$ and $\epsilon \in \{0, 1\}$.
	Clearly there are $2k$ such maps.
	And if $t, t' \in [0,1]^{k-1}$ 
	then 
	by the Cauchy--Schwarz inequality,
	\[
	\lVert \phi_{j,\epsilon}(t) - \phi_{j,\epsilon}(t') \rVert 
	= \lVert \phi_{j,0}(t - t') \rVert 
	\le 
	k \lambda_k \sum_{i=1}^{k-1} |t_i - t_i'|
	\le k \lambda_k \lVert t - t' \rVert \sqrt{k-1}
	\]
	so that we may take $c_2(k) := k^{3/2}$.
\end{proof}

\begin{rmk}
	Since $\Bar D$ is compact and convex, Bauer's maximum principle implies that the norm is maximized at an extreme point of $\Bar D$.
	Thus, in practise, one only has to check $2^r - 1$ vectors,
	namely $\sum_{i\in I} b_i$
	for $\varnothing \ne I \subseteq \{1, \ldots, r\}$, 
	to find the bounding radius.
\end{rmk}

\begin{rmk}
	Since $[0,1]^k - [0,1]^k = [-1, 1]^k$, the optimal Lipschitz constant of $\phi_{j,\epsilon}$ is given by the operator norm of $\phi_{j,0}$ w.r.t.~the Euclidean norms. This, in turn, is given by the largest singular value of $\phi_{j,0}$.
\end{rmk}

\begin{lem} \label{lattices_c3}
	There exists an increasing function $c_3(k)$
	such that for every lattice $\Lambda$ of rank $k \ge 1$,
	\[
	\lambda_k \le c_3(k) \frac{\det \Lambda}{\lambda_1^{k-1}}.
	\]
\end{lem}
\begin{proof}
	Minkowski's second theorem \eqref{minkowski_minima} implies
	\[\lambda_1^{k-1} \lambda_k \le \lambda_1 \ldots \lambda_k \le \frac{2^k}{V_k} \det \Lambda;\]
	let 
	\[c_3(k) := \frac{2^k}{V_k}.\]
	Now $c_3(k) < c_3(k+1)$ if and only if $V_{k+1} < 2 V_k$.
	By \eqref{vol_kball},
	\[
	\frac{V_{k+1}}{V_k}
	= \sqrt{\pi} \frac{\Gamma(k/2 + 1)}{\Gamma(k/2 + 3/2)}
	\]
	which is decreasing in $k$,
	ultimately because the digamma function is increasing.\footnote{If $g(x) := \Gamma(x) / \Gamma(x+t)$ then $g(x) > 0$ and $g'(x)/g(x) = (\log g(x))' = \psi(x) - \psi(x+t) < 0$ for all $x, t > 0$.}
	Thus, 
	\[
	\frac{V_{k+1}}{V_k} 
	\le 
	\frac{V_2}{V_1}
	= \frac{\pi}{2} < 2. \qedhere 
	\]
\end{proof}

\begin{rmk}
	By Stirling's approximation,
	\[c_3(k) \sim \sqrt{\pi k} \Big(\frac{2k}{\pi e}\Big)^{k/2}.\]
\end{rmk}

\begin{lem} \label{jon_lambda1_units}
	$\lambda_1(\Lambda(\mathcal{O}_K^\times)) \gg_n 1$.
\end{lem}

\begin{proof}
	Let $u \in \mathcal{O}_K^\times$.
	In terms of the 1-norm,
	\[
	\lVert \Lambda(u) \rVert_1 
	= \sum_{v \mid \infty} \big|{\log {\lVert u \rVert_v}}\big|
	\ge \sum_{v \mid \infty} \max\{0, \log {\lVert u \rVert_v}\}
	= nh(u)
	\]
	because every unit is an algebraic integer.
	But the 1-norm is related to the 2-norm via
	\[
	\lVert y \rVert_1 
	\le \sqrt{r+1} \, { \lVert y \rVert_2 }
	\]
	for all $y$ in $\R^{r+1}$.
	Since $r+1 \le n$, we thus have
	\[
	\lVert \Lambda(u) \rVert_2 
	\ge \sqrt{n} \, h(u).
	\]
	Now, 
	the set 
	\[\{\alpha \in \Bar \Q : 1 < H(\alpha) \le 2 \text{ and } [\Q(\alpha) : \Q] \le n\}\]
	is finite by Northcott's theorem,
	and nonempty because it contains $\alpha = 2$;
	let $c_0(n)$ be the minimum value of $h$ on this set.
	If $\alpha$ is algebraic of degree at most $n$, 
	then either 
	\begin{enumerate}[(a)]
		\item $h(\alpha) = 0$, in which case $\alpha$ is zero or a root of unity by Kronecker's theorem; or
		\item $h(\alpha) > 0$, in which case $h(\alpha) \ge c_0(n) > 0$ by construction.
	\end{enumerate}
	It follows that if $u \not \in \ker \Lambda = \mu(K)$, 
	then $\lVert \Lambda(u) \rVert_2 \ge \sqrt{n} \, c_0(n)$
	as desired.
\end{proof}

\begin{rmk} \label{rmk:Lehmer}
	The example $\alpha = 2^{1/n}$ shows that $c_0(n) \le n^{-1} \log 2$; it is widely believed that $c_0(n) \ge n^{-1} c$ for some absolute positive constant $c$ (cf.~\textit{Lehmer's conjecture}).
\end{rmk}

\begin{propn} \label{shapely_domains_exist}
Every number field has a well-rounded fundamental parallelotope for its unit lattice.
\end{propn}
\begin{proof}
Trivial if $r = 0$, so assume $r > 0$.
Combining Lemmas \ref{lattices_c1_c2}, \ref{lattices_c3}, and \ref{jon_lambda1_units},
we obtain a fundamental parallelotope $D_K$ 
for $\Lambda(\mathcal{O}_K^\times)$
such that 
\[
y \in D_K \implies 
\lVert y \rVert_2 \le c_1(r) \lambda_r
\]
and $\bdy D_K$ is covered by the images $2r$ maps 
$\phi_i : [0, 1]^{r-1} \to \R^{r+1}$ 
with 
\[
1 \le i \le 2r \implies \Lip(\phi_i) \le c_2(r) \lambda_r
\]
where 
\[
\lambda_r \le c_3(r) \frac{\sqrt{r+1}}{\lambda_1^{r-1}} R
\]
and 
$\lambda_1 \ge \sqrt{n} c_0(n)$, 
which is less than 1 by Remark \ref{rmk:Lehmer}.
Since the right-hand sides of the last three displays are increasing in $r < n$,
$D_K$ is well-rounded.
\end{proof}

\subsection{Lipschitz parametrizability}

At last:

\begin{lem} \label{exists_good_DFK}
Suppose that for each $v \mid \infty$, 
the boundary of
$D_{f,v}$ 
is of Lipschitz class $(N_f, L_f)$.
If $D_K$ is well-rounded,
then the boundary of  
\[D_{F,K}(1) := D_F(1) \cap ({\pr} \circ \eta)^{-1}(D_K)\]
is of Lipschitz class 
$(\Tilde N_F, \Tilde L_F)$,
where $\Tilde N_F \ll_n N_f^n$ 
and 
\[
\Tilde L_F 
\ll_{m,n,R}
\max_{v \mid \infty} |F|_v^{-1/d} 
\cdot 
\begin{dcases}
L_f & r = 0, \\
L_f + C^\infty_f & r > 0.
\end{dcases}
\]
\end{lem}

\begin{proof}
If $r = 0$ then there is only one archimedean place $v$, 
and $n = 1$ or $2$ (according as $v$ is real or complex).
In this case, $D_{F,K}(1) = |F|_v^{-1/d} D_{f,v}$.
By hypothesis, we may take $\Tilde{N}_F := N_f \le N_f^n$ and $\Tilde{L}_F := |F|_v^{-1/d} L_f$.

So, suppose $r = 1$. 
We wish to apply \cite[Lemma 7.1]{Widmer_prim}, 
but it takes a bit of care 
to translate perspectives.
Widmer works in terms of the ``normalized'' fundamental domains
\[D_{F,v} := \{z \in K_v^{m+1} : |F(z)|_v \le 1\}\] 
and writes $(N_{F,v}, L_{F,v})$ for the Lipschitz class of 
$\bdy D_{F,v}$, 
for each $v \mid \infty$ 
\cite[Definition 2.2, axiom (iii)]{Widmer_prim}.
He then defines 
\[N_F := \max_{v\mid\infty} N_{F,v} \quad\text{and}\quad L_F := \max_{v \mid \infty} L_{F,v}.\]
He also lets $c_v \in (0, 1]$ be constants satisfying
\begin{equation} \label{Widmer_constants}
	|F(z)|_v^{1/d} \ge c_v |z|_v
\end{equation}
for all $z \in K_v^{m+1}$,
and defines [\textit{sic}]
\[
C_F^{inf} := \max_{v \mid \infty} c_v^{-1}.
\]
Equipped with these data, 
\cite[Lemma 7.1]{Widmer_prim} states that 
if $r \ge 1$ 
and $S \subseteq \Pi \subset \R^{r+1}$
is contained in an origin-centered ball of radius $B_S$
and has boundary of Lipschitz class $(N_S, L_S)$ (in codimension 2), 
then $\bdy D_{F,S}(1)$ is of Lipschitz class $(\Tilde{N}_F, \Tilde{L}_F)$ 
where 
\begin{equation} \label{Widmer_lemma}
\begin{aligned}
\Tilde N_F &= (N_S + 1) N_F^{r+1}, \\
\Tilde L_F &= 3\sqrt{n(m+1)}(L_S + B_S + 1) e^{\sqrt{r}(L_S + B_S)} (L_F + C_F^{inf}).
\end{aligned}
\end{equation}

Now we translate.
Widmer's fundamental domains 
are related to ours via
\[D_{F,v} = |F|_v^{-1/d} D_{f,v}.\]
Thus, by hypothesis, we may take $N_{F,v} = N_f$ and $L_{F,v} = |F|_v^{-1/d} L_f$
for each $v \mid \infty$.
Then
\[
N_F = N_f 
\quad\text{and}\quad 
L_F = \max_{v \mid \infty} |F|_v^{-1/d} L_f.
\]
Next, although Widmer stipulates 
that the constants $c_v$ satisfying
\eqref{Widmer_constants}
should not exceed 1,
inspecting his proof of Lemma 7.1 \cite[Appendix A, final paragraph]{Widmer_prim} reveals that this requirement is superfluous.
So, we may take 
\[
c_v = \inf_{z \ne 0} \frac{|F(z)|_v^{1/d}}{|z|_v}
\]
(cf.~Proposition \ref{propn:kappa}).
Then 
by \eqref{error_constants},
\[C_F^{inf} 
= \max_{v \mid \infty} c_v^{-1}
\le \max_{v \mid \infty} |F|_v^{-1/d} C_v
\le \max_{v \mid \infty} |F|_v^{-1/d} C_f^\infty.
\]
Lastly, 
taking $S$ to be a well-rounded fundamental domain $D_K$, 
we get that $S$ is contained in an origin-centered ball of radius $B_S \ll_n R$ 
and has boundary of Lipschitz class $(N_S, L_S)$ 
where $N_S \ll_n 1$ and $L_S \ll_n R$.
Putting it all together, 
we conclude from \eqref{Widmer_lemma} 
that
$\bdy D_{F,K}(1)$ is of Lipschitz class $(\Tilde{N}_F, \Tilde{L}_F)$
where 
\[\Tilde N_F \ll_n N_F^{r+1} \le N_f^n\]
and
\[\Tilde L_F \ll_{n,m,R} L_F + C_F^{inf}
\le \max_{v \mid \infty} |F|_v^{-1/d} (L_f + C_f^\infty).
\qedhere 
\]
\end{proof}

\section{Proof of the Main Theorem} \label{sec:final}

Schanuel's original proof \cite{Schanuel} may be straightened out as follows.
\begin{enumerate}
	\item Lift to affine space: $\P^m(K) = (K^{m+1} - 0^{m+1}) / K^\times$
	\item Fibre over ideal classes \textcolor{gray}{[Corollary, p.~447]}\footnote{There's a typo: $\lambda^m$ should be $\Tilde\lambda^m$}
	\item Replace $K^\times$ by $\mathcal{O}_K^\times$ \textcolor{gray}{[Theorem 3, p.~446]}
	\item Perform a M\"obius inversion \textcolor{gray}{[Lemma 12, p.~446]}
	\item Replace $\mathcal{O}_K^\times$ by $\mu(K)$ by picking a fundamental domain $\Delta$ \textcolor{gray}{[pp.~436--438]}\footnote{This is a somewhat breathless passage, involving several auxiliary functions; and there's a gap---filled by our Lemma \ref{schanuel_missing_lemma}.}
	\item Embed into $K_\R^{m+1} \cong \R^{n(m+1)}$ 
	to count lattice points in $\Delta$ \textcolor{gray}{[Proposition 1, p.~438]}
	\item Show that $\Delta$ is homogeneously expanding \textcolor{gray}{[Lemma 3, p.~438]}
	and bounded with $C^1$-parametrizable boundary \textcolor{gray}{[Proposition 2, p.~439]}\footnote{Itself proved as [Lemmas 4--10, pp.~439--442].}
	\item Calculate the volume of $\Delta$ \textcolor{gray}{[Proposition, p.~443\textit{f}]}
	\item Apply geometry of numbers \textcolor{gray}{[Theorem 2, p.~438]}
\end{enumerate}

Our proof diverges from Schanuel's in the following respects.
Because we must account for the presence of $f$:
in (2) we also fibre over excess divisors;
in (3) we involve a lift of $f$; 
in (5) our fundamental domain depends on $f$; 
and in (6) we count points in \emph{cosets} of a lattice.
The volume calculation (8) is supplied by Masser--Vaaler \cite[Lemma 4]{MV}, 
while the Lipschitz parametrizability (7) and geometry of numbers (9) are provided by Widmer \cite[Theorem 5.4 and Lemma 7.1]{Widmer_prim}.
In order to better keep track of our error constants, 
we do not sum the Dirichlet series (4) until the very end.

\begin{proof}[Proof of Theorem \ref{thm:main}]
If $X < 1$ then both sides vanish,
so assume $X \ge 1$.

\smallskip 
\noindent\textit{Step 1---Ideal classes and excess divisors.}
Let $F$ be a homogeneous lift of $f$
and write $K^{m+1}_\prim = K^{m+1} - 0^{m+1}$.
The objective is to count orbits under the $K^\times$-action:
\[
\Num_{f^* \! H, \P^m(K)}(X) 
= \left|\frac{\{x \in K^{m+1}_\prim : H(F(x)) \le X\}}{K^\times} \right|.
\]
Define
\begin{equation} \label{eq:R}
\mf R := \prod_{v \nmid \infty} \mf p_v^{\lVert \varepsilon_{f,v} \rVert}
\end{equation}
and select (using Corollary \ref{cor:class_rep}) class group representatives $\mf a_1, \ldots, \mf a_h$ coprime to $\mf R$.
Writing $\langle x \rangle$ for the fractional ideal $x_0 \mathcal{O}_K + \ldots + x_m \mathcal{O}_K$ 
generated by the coordinates of $x$,
we observe that
\begin{equation*} \label{idk}
\langle \lambda x \rangle = \langle \lambda \rangle \langle x \rangle
\end{equation*}
for all $\lambda$ in $K$ and all $x$ in $K^{m+1}$.
Thus, for each $i$, the set 
\[\{x \in K^{m+1}_\prim : \Cl{\langle x \rangle} = \Cl \mf a_i\}\] 
is $K^\times$-invariant.
Moreover, it is easily verified that this set admits its subset
\[\{x \in (\mathcal{O}_K)^{m+1}_{\prim/\mf R} : \langle x \rangle = \mf a_i\}\]
as a fundamental domain modulo $\mathcal{O}_K^\times$.
Thus by Proposition \ref{enlargement}, we have
\begin{align*}
\Num_{f^* \! H, \P^m(K)}(X) 
&= \sum_{i=1}^h \left|\frac{\{x \in K^{m+1}_\prim : \Cl{\langle x \rangle} = \Cl{\mathfrak{a}_i} \text{ and } H(F(x)) \le X\}}{K^\times} \right| \\
&= \sum_{i=1}^h \left|\frac{\{x \in (\mathcal{O}_K)^{m+1}_{\prim/\mf R} : \langle x \rangle = \mathfrak{a}_i \text{ and } H(F(x)) \le X\}}{\mathcal{O}_K^\times} \right|.
\end{align*}

Next, 
\[
H(F(x)) \le X 
\iff H_K(F(x)) \le X^n
\iff H_\infty(F(x)) 
\le X^n \Nm {\langle F(x) \rangle}.
\]
But if $\langle x \rangle = \mf a_i$ and $\ell_f(x) = \mf l$ 
then by Proposition \ref {ell_properties}(i),
\[{\langle F(x) \rangle} = \mf a_i \langle F \rangle \mf l.\]
By the same Proposition, the excess divisor is scale-invariant (ii) and divides $\mf R$ (iii).
Thus we may fibre the sum further:
\[
\Num_{f^* \! H, \P^m(K)}(X)
= \sum_{i=1}^h \sum_{\mf l} \left| 
\frac{\{x \in (\mathcal{O}_K)^{m+1}_{\prim/\mf R} : \langle x \rangle = \mf a_i, \, \ell_f(x) = \mf l, \, H_\infty(F(x)) \le Y_{i,\mf l}\}}{\mathcal{O}_K^\times}\right|
\]
where $\mf l \mid \mf R$ and
\begin{equation} \label{Y}
Y_{i, \mf l} := X^n \Nm \mf a_i^d \langle F \rangle \mf l.
\end{equation}

\smallskip 
\noindent\textit{Step 2---M\"obius inversion.}
Fix an ideal $\mf l$ and a real number $Y \in \R$.
For the nonce, define two functions of a nonzero integral ideal variable $\mf a$:
\begin{align*}
	\varsigma(\mf a) &= \left|\frac{\{x \in (\mathcal{O}_K)^{m+1}_{\prim/\mf R} : 
		\langle x \rangle = \mf a, \, 
		\ell_f(x) = \mf l, \,
		H_\infty(F(x)) \le Y\}}
	{\mathcal{O}_K^\times} \right|
	\\
	\upsilon(\mf a) &= \left|\frac{\{x \in (\mathcal{O}_K)^{m+1}_{\prim/\mf R} : 
		\langle x \rangle \subseteq \mf a, \, 
		\ell_f(x) = \mf l, \, 
		H_\infty(F(x)) \le Y\}}
		{\mathcal{O}_K^\times} \right|
\end{align*}
and observe that 
\begin{equation} \label{g_intermsof_f}
\upsilon(\mf a) = \sum_{\mf b} \varsigma(\mf{ab})
\end{equation}
where $\mf b$ ranges over nonzero integral ideals.
This is because
\[\langle x \rangle \subseteq \mf a \iff \mf a \mid \langle x \rangle \iff \langle x \rangle = \mf{ab} \text{ for some } \mf b \subseteq \mathcal{O}_K.\]
In order for the M\"obius inversion formula 
\begin{equation} \label{f_intermsof_g}
\varsigma(\mf a) = \sum_{\mf b} \mu_K(\mf b) \upsilon(\mf {ab})
\end{equation}
to hold, it suffices to show that
$\varsigma(\mf a) = 0$ for all but finitely many $\mf a$
(see \cite[Corollary 1]{Rota}).
To that end, note that the absolute height is bounded below by 1;
thus if $x$ lies in the set defining $\varsigma(\mf a)$ 
then 
\[Y \ge H_\infty(F(x)) \ge \Nm {\langle F(x) \rangle} = \Nm \mf a^d \langle F \rangle \mf l
\]
whence $\Nm \mf a \le (Y / \Nm {\langle F \rangle \mf l})^{1/d}$.
In particular, $\upsilon(\mf a) = 0$ outside the same range, by \eqref{g_intermsof_f}.
So \eqref{f_intermsof_g} holds, and is actually a finite sum,
with 
$\Nm \mf b 
\le (Y / \Nm {\langle F \rangle \mf l})^{1/d} / \Nm \mf a$.
It follows that 
\begin{align*}
\Num_{f^* \! H, \P^m(K)}(X) & \\
=
\sum_{i=1}^h 
\sum_{\mf l}
&\sum_{\mf b} \mu_K(\mf b)
\left|\frac{\{x \in (\mathcal{O}_K)^{m+1}_{\prim/\mf R} : 
	\langle x \rangle \subseteq \mf a_i \mf b, \, 
	\ell_f(x) = \mf l, \, 
	H_\infty(F(x)) \le Y_{i,\mf l}\}}
{\mathcal{O}_K^\times} \right|,
\end{align*}
where $\mf b$ satisfies
\[\Nm \mf b \le (Y_{i,\mf l} / \Nm {\langle F \rangle \mf l})^{1/d} / \Nm \mf a_i = X^{n/d}\]
by \eqref{Y}.

\smallskip 
\noindent\textit{Step 3---Embedding in Minkowski space.}
To rid ourselves of the quotient by $\mathcal{O}_K^\times$,
we select fundamental domains. 
Let $D_K$ be a fundamental parallelotope 
for the unit lattice $\Lambda(\mathcal{O}_K^\times)$;
by Proposition \ref{shapely_domains_exist}, 
we may assume $D_K$ is well-rounded.
For any $T > 0$, let
\begin{align*}
D_F(T) &= \rho^{-1}(-\infty, T]
\shortintertext{and}
D_{F, K}(T) &= D_F(T) \cap ({\pr} \circ \eta)^{-1}(D_K).
\end{align*}
Let $\Psi : K^{m+1} \to K_\R^{m+1}$ be the Minkowski--Schanuel embedding.
By Propositions \ref{map_rho}(iii)
and \ref{counting_orbits},
\begin{align*}
&\left|\frac{\{x \in (\mathcal{O}_K)^{m+1}_{\prim/\mf R} : 
	\langle x \rangle \subseteq \mf {ab}, \, 
	\ell_f(x) = \mf l, \, 
	H_\infty(F(x)) \le Y\}}
{\mathcal{O}_K^\times} \right|
\\
&=
\frac{1}{w} \left|
 \Psi\big(\{x \in (\mathcal{O}_K)^{m+1}_{\prim/\mf R} : \langle x \rangle \subseteq \mf {ab}, \, \ell_f(x) = \mf l\}\big) \cap D_{F,K}(Y^{1/d})
 \right|,
\end{align*}
and it follows that 
\begin{align*}
\Num_{f^* \! H, \P^m(K)}(X) & \\
= \frac{1}{w}
\sum_{i=1}^h 
\sum_{\mf l}
& \sum_{\mf b} 
\mu_K(\mf b)
\left|
\Psi\big(\{x \in (\mathcal{O}_K)^{m+1}_{\prim/\mf R} : \langle x \rangle \subseteq \mf a_i \mf b, \, \ell_f(x) = \mf l\}\big) \cap 
D_{F,K}(Y_{i,\mf l}^{1/d})
\right|.
\end{align*}
By Proposition \ref{local_fundamental_domains}(ii),
there exist an integer $N_f$ and a real number $L_f$ 
such that $\bdy D_{f,v}$ is of Lipschitz class $(N_f, L_f)$ for each $v \mid \infty$.
By Proposition \ref{properties_DFS}(i) and Lemma \ref{exists_good_DFK}, 
$D_{F,K}(Y^{1/d}) = Y^{1/dn} D_{F,K}(1)$ is a homogeneously-expanding domain whose boundary is of Lipschitz class $(\Tilde{N}_F, Y^{1/dn} \Tilde{L}_F)$ 
where $\Tilde{N}_F \ll_n N_f^n$ and 
\begin{equation} \label{final_Lipschitz_constant}
\Tilde{L}_F \ll_{m, K} (L_f + C_f^\infty)
\max_{v\mid\infty} |F|_v^{-1/d}
\end{equation}
(and we may drop $C_f^\infty$ if $r = 0$).
Meanwhile, 
each
\[
\Psi\big(\{x \in (\mathcal{O}_K)^{m+1}_{\prim/\mf R} : \langle x \rangle \subseteq \mf a_i \mf b, \, \ell_f(x) = \mf l\}\big)
\]
is a (possibly empty) union of disjoint cosets of the lattice $\Psi\big((\mf a_i \mf b \cap \mf R)^{(m+1)}\big)$.
More precisely,
in the notation of 
Lemma \ref{union_of_cosets},
there exist points $x_{i, \mf b, Q, \Bar u} \in \mathcal{O}_K^{m+1}$, 
where 
$Q \in Z(\mf l)$ and $\Bar u \in U(\mf R),$
such that 
\begin{align*}
	\Num_{f^* \! H, \P^m(K)}(X) & \\
	= \frac{1}{w}
	\sum_{i=1}^h 
	\sum_{\mf l}
	& \sum_{\mf b} 
	\sum_Q
	\sum_{\Bar u}
	\mu_K(\mf b)
	[\mf b \perp \mf R]
	\left|
	\Psi\big(x_{i,\mf b, Q, \Bar u} +(\mf a_i \mf b \cap \mf R)^{(m+1)}\big) \cap
Y_{i,\mf l}^{1/dn} D_{F,K}(1)
	\right|.
\end{align*}
It now follows from Widmer's counting principle (Corollary \ref{counting_principle}) that 
\begin{equation} \label{last_true_equality}
\begin{aligned}
	\Num_{f^* \! H, \P^m(K)}(X) & \\
	= \frac{1}{w}
	\sum_{i=1}^h 
	\sum_{\mf l}
	& \sum_{\mf b} 
	\sum_Q
	\sum_{\Bar u}
	\mu_K(\mf b)
	[\mf b \perp \mf R]
	\bigg(
	\frac{2^{r_2(m+1)} \vol D_{F,K}(1) (Y_{i,\mf l}^{1/dn})^{n(m+1)}}
	{\sqrt{d_K}^{m+1} \Nm {(\mf a_i \mf b \cap \mf R)}^{m+1}} 
	+ E
	\bigg)
\end{aligned}
\end{equation}
where 
\[
|E| \ll_{n,m} \Tilde{N}_F \bigg(1 + \frac{(Y_{i,\mf l}^{1/dn} \Tilde{L}_F)^{n(m+1)-1}}{\Nm {(\mf a_i \mf b \cap \mf R)}^{m+1-1/n}}\bigg).\]

\smallskip 
\noindent\textit{Step 4---The main term.}
Note that by \eqref{Y}
we have 
\begin{equation} \label{Y_dn}
	Y_{i,\mf l}^{1/dn} =
	X^{1/d} \Nm \mf a_i^{1/n} \Nm {\langle F \rangle}^{1/dn} \Nm \mf l^{1/dn}.
\end{equation}
A bit of rearranging\footnote{Note that all five sums are finite.} and simplification gives
that the main term of \eqref{last_true_equality} is
\begin{align*}
& 
\frac{1}{w}
\sum_{i=1}^h 
\sum_{\mf l}
\sum_{\mf b} 
\sum_Q
\sum_{\Bar u}
\mu_K(\mf b)
[\mf b \perp \mf R]
\frac{2^{r_2(m+1)} \vol D_{F,K}(1) (Y_{i,\mf l}^{1/dn})^{n(m+1)}}
{\sqrt{d_K}^{m+1} \Nm {(\mf a_i \mf b \cap \mf R)}^{m+1}} 
\\
&= 
\frac{2^{r_2(m+1)}}{w \sqrt{d_K}^{m+1}}
\vol D_{F,K}(1)\Nm {\langle F \rangle}^{(m+1)/d}
\\
&\ \cdot \ 
	\sum_{i=1}^h 
	\sum_{\mf l}
	\sum_{\mf b} 
	\sum_Q 
	\sum_{\Bar u}
	\Nm \mf l^{(m+1)/d}
	\mu_K(\mf b)
	[\mf b \perp \mf R]
	\frac{\Nm \mf a_i^{m+1}}{\Nm {(\mf a_i \mf b \cap \mf R)}^{m+1}} X^{n(m+1)/d}.
\end{align*}
We simplify three groups of terms.
First, by Proposition \ref{volume_DFK},
\[
\vol D_{F,K}(1) \Nm {\langle F \rangle}^{(m+1)/d}
= 
(m+1)^r R \frac{c_{K,\infty}(f)}{H(f)^{n(m+1)/d}}.\]
Second,
by definition of $\delta_f$ 
and by Proposition \ref{J_Dedekind} and 
Corollary \ref{nonarch_vol},
\begin{align*}
\sum_{\mf l} \sum_Q \sum_{\Bar u} \Nm \mf l^{(m+1)/d}
&= 
\sum_{\mf l} \# Z(\mf l) \varphi_K(\mf R) \Nm \mf l^{(m+1)/d}
\\
&= 
\varphi_K(\mf R) \#\P^m(\mathcal{O}_K/\mf R) 
\sum_{\mf l} \Nm \mf l^{(m+1)/d} \delta_f(\mf l) 
\\
&= 
J_{K,m+1}(\mf R) c_{K,0}(f).
\end{align*}
Third, by Lemma \ref{lem:dirichlet_series} applied with 
$s = \sigma = m+1 > 1$ and $\Nm \mf b \le X^{n/d}$,
\[
\sum_{\mf b} \mu_K(\mf b) [\mf b \perp \mf R]
\frac{\Nm \mf a_i^{m+1}} {\Nm{(\mf a_i \mf b \cap \mf R)^{m+1}}}
= 
\frac{1}{\zeta_K(m+1) J_{K,m+1}(\mf R)} + O_{m,K}\Big(\frac{X^{-nm/d}}{\Nm \mf R^{m+1}}\Big).
\]
Putting these three identities together,
and noting that 
\begin{equation} \label{J_bound}
J_{K,m+1}(\mf R) \le \Nm \mf R^{m+1},
\end{equation}
we get that the main term simplifies to
\begin{align*}
& \frac{2^{r_2(m+1)}} {w \sqrt{d_K}^{m+1}}
(m+1)^r R_K \frac{c_{K,\infty}(f)} {H(f)^{n(m+1)/d}}
J_{K,m+1}(\mf R) c_{K,0}(f) 
\\
& \cdot \ \sum_{i=1}^h \bigg( 
\frac{1}{\zeta_K(m+1) J_{K,m+1}(\mf R)} + O_{m,K}\Big(\frac{X^{-nm/d}}{\Nm \mf R^{m+1}} \Big)
\bigg) X^{n(m+1)/d}
\\
&=  
c_K(f) X^{n(m+1)/d} 
+ 
O_{m,K}\big(c_K(f) X^{n/d}\big).
\end{align*}

\smallskip 
\noindent\textit{Step 5---The error term.}
We must estimate
\begin{align*}
\bigg| \frac{1}{w} 
\sum_{i, \mf l, \mf b, Q, \Bar u}
& 
\mu_K(\mf b) 
[\mf b \perp \mf R] 
E
\bigg|
\\
&\ll_{n, m} 
\frac{N_f^n}{w} 
\sum_{i, \mf l, \mf b, Q, \Bar u}
|\mu_K(\mf b)| [\mf b \perp \mf R]
\bigg(1 + \frac{ (Y_{i,\mf l}^{1/dn} \Tilde{L}_F)^{n(m+1)-1} } {\Nm {(\mf a_i \mf b \cap \mf R)}^{m+1-1/n}} \bigg).
\end{align*}
By the trivial estimate $|\mu_K(\mf b)| [\mf b \perp \mf R] \le 1$, the inequality \eqref{J_bound}, and Weber's theorem (cf.~Lemma \ref{lem:weber}), 
\[
\frac{N_f^n}{w} 
\sum_{i,\mf l, \mf b, Q, \Bar u} |\mu_K(\mf b)| [\mf b \perp \mf R]
\ll_K 
N_f^n 
\Nm \mf R^{m+1} X^{n/d}.
\]
For the dominant term, put $\sigma = m + 1 - 1/n$.
Then by \eqref{Y_dn},
\begin{align*}
\frac{ (Y_{i,\mf l}^{1/dn} \Tilde{L}_F)^{n(m+1)-1} } {\Nm {(\mf a_i \mf b \cap \mf R)}^{m+1-1/n}}
&=
\Big( \frac{ Y_{i, \mf l}^{1/d} \Tilde{L}_F^n }{\Nm {(\mf a_i \mf b \cap \mf R)} } \Big)^\sigma
\\
&= \Big( X^{n/d} \Tilde{L}_F^n \Nm {\langle F \rangle}^{1/d} \Nm \mf l^{1/d} \frac{\Nm \mf a_i}{\Nm {(\mf a_i \mf b \cap \mf R)}} \Big)^\sigma.
\end{align*}
Thus
\begin{align*}
\frac{N_f^n}{w} 
& \sum_{i,\mf l, \mf b, Q, \Bar u} 
|\mu_K(\mf b)| [\mf b \perp \mf R] 
\frac{ (Y_{i,\mf l}^{1/dn} \Tilde{L}_F)^{n(m+1)-1} } {\Nm {(\mf a_i \mf b \cap \mf R)}^{m+1-1/n}} \\
&=
\frac{N_f^n \Tilde{L}_F^{n\sigma} \Nm {\langle F \rangle}^{\sigma/d}}{w} 
\sum_{i, \mf l, \mf b, Q, \Bar u} 
\Nm \mf l^{\sigma/d} 
|\mu_K(\mf b)|
[\mf b \perp \mf R]
\frac{\Nm \mf a_i^\sigma}
     {\Nm {(\mf a_i \mf b \cap \mf R)^\sigma}}
X^{n\sigma/d}.
\end{align*}
As before, 
\[
\sum_{\mf l, Q, \Bar u} 
\Nm \mf l^{\sigma/d} 
\le 
\sum_{\mf l, Q, \Bar u}
\Nm \mf l^{(m+1)/d}
= c_{K,0}(f) J_{K,m+1}(\mf R).
\]
Also, by Lemma \ref{lem:dirichlet_series}(ii),
\[
\sum_{\mf b} |\mu_K(\mf b)| [\mf b \perp \mf R] 
\frac{\Nm \mf a_i^\sigma}
{\Nm {(\mf a_i \mf b \cap \mf R)^\sigma}}
\ll_{m, K} \frac{1 + {\log^+} X^{n/d}}{\Nm \mf R^\sigma}
\]
where the ${\log^+}{X^{n/d}}$ disappears for $nm > 1$.
Observe that since powering is monotone and fixes 1,
\[
{\log^+} X^{n/d} = \frac{n}{d} \, {\log^+} X \ll_n \frac{1}{d} \, {\log^+} X = {\log^+} X^{1/d}.
\]
Note also that
\[\frac{J_{K,m+1}(\mf R)}{\Nm \mf R^\sigma}
\le \Nm \mf R^{m+1-\sigma} 
= \Nm \mf R^{1/n}.\]
It follows that 
the dominant term is bounded above by
\[
\ll_{m,K}
N_f^n (\Tilde{L}_F)^{n\sigma} \Nm {\langle F \rangle}^{\sigma/d}
\, c_{K,0}(f) \Nm \mf R^{1/n} X^{n\sigma/d}(1 + {\log^+} X^{1/d}).
\]
\smallskip 
\noindent\textit{Step 6---Wrapping up.}
Combining all three error terms
yields
\begin{align*}
\Big| \Num_{f^* \! H, \P^m(K)}&(X) 
 - c_K(f) X^{n(m+1)/d}\Big|
\\
& \ll_{m, K}
c_K(f) X^{n/d} 
\\
& \quad + 
N_f^n \Nm \mf R^{m+1} X^{n/d} 
\\ 
& \qquad + 
N_f^n 
(\Tilde{L}_F)^{n\sigma} \Nm {\langle F \rangle}^{\sigma/d}
\, c_{K,0}(f) \Nm \mf R^{1/n} X^{n\sigma/d}(1 + {\log^+} X^{1/d}).
\end{align*}
The left-hand side is lift-independent, 
so we may replace the right-hand side 
by its infimum over all lifts $F$ of $f$.
Recalling \eqref{final_Lipschitz_constant}
and using Lemma \ref{inf_lifts} below, 
we obtain
\[
\inf_F {(\Tilde{L}_F)^{n\sigma}} \Nm {\langle F \rangle}^{\sigma/d}
\ll_{m, K}
\frac{(L_f + C_f^\infty)^{n\sigma}}{H(f)^{n\sigma/d}}.
\]
Finally,
we estimate $\Nm \mf R$.
Noting that 
$\Nm \mf p_v = q_v^{n_v}$,
definition \eqref{eq:R}
yields
\begin{align*}
\Nm \mf R
&= 
\prod_{v \nmid \infty} \Nm \mf p_v^{\lVert \varepsilon_{f,v} \rVert}
\\
&=
\prod_{v \nmid \infty} q_v^{n_v \lVert \varepsilon_{f,v} \rVert}.
\end{align*}
If $C_v$ are any constants 
satisfying \eqref{intro_constants_1} and \eqref{intro_constants_2}
then 
by Proposition \ref{propn:kappa},
\[
C_v \ge \sup_{z \ne 0} \frac{|z|_v |F|_v^{1/d}}{|F(z)|_v^{1/d}} 
= q_v^{\lVert \varepsilon_{f,v}\rVert/d}.
\]
Therefore by \eqref{error_constants},
\[\Nm \mf R 
\le \prod_{v \nmid \infty} C_v^{dn_v}
= (C_f^0)^d. \qedhere 
\]
\end{proof}

\begin{lem} \label{inf_lifts}
Let $f : \P^m \to \P^M$ be a morphism defined over a number field $K$. Then
	\[
	\inf_F \big(\! \max_{v \mid \infty} |F|_v^{-n} \Nm {\langle F \rangle}  \big)
	= 
	\frac{1}{H(f)^n}
	\]
	where the infimum ranges over all lifts $F$ of $f$.
\end{lem}
\begin{proof}
Fix a lift $F$ of $f$.
From the identity
\[
H(f)^n
= \frac{1}{\Nm {\langle F \rangle}} \prod_{w\mid\infty} |F|_w^{n_w}
\]
we have
\[
\max_{v \mid \infty} |F|_v^{-n}
\Nm {\langle F \rangle} 
=
\frac{1}{H(f)^n}
\underbrace{
	\max_{v \mid \infty} |F|_v^{-n} \prod_{w \mid \infty} |F|_w^{n_w}
}_{(*)}.
\]
	Consider the latter factor $(*)$.
	Clearly it's at least 1: 
	\[
	\max_{v \mid \infty} |F|_v^{-n}
	= 
	\prod_{w \mid \infty} \max_{v \mid \infty} |F|_v^{-n_w}
	\ge 
	\prod_{w \mid \infty} |F|_w^{-n_w}
	\]
	because $\sum_{w\mid\infty} n_w = n$.
	Thus it remains to show that as we range over all lifts $\lambda F$ of $f$, the factor $(*)$ comes arbitrarily close to 1.
	To this end,
	for each $v \mid \infty$ let $j_v, \alpha_v$ be such that $|F|_v = |F_{j_v,\alpha_v}|_v$ and note that $F_{j_v,\alpha_v} \ne 0$. 
	Let $\varepsilon > 0$, 
	making sure $\varepsilon < \min_{v \mid \infty} |F|_v^{-1}$,
	and 
	use the Artin--Whaples approximation theorem \cite[Theorem 1]{ArtinWhaples}
	to pick $\lambda$ in $K$ 
	such that 
	\[|\lambda - F_{j_v,\alpha_v}^{-1}|_v < \varepsilon\]
	for all $v \mid \infty$.
	Then $|\lambda F_{j_v, \alpha_v} - 1|_v < \varepsilon |F|_v$
	and $|\lambda F_{j_v,\alpha_v}|_v
	= |\lambda F|_v$
	and it follows (from the reverse triangle inequality) that
	\[1 - \varepsilon|F|_v < |\lambda F|_v < 1 + \varepsilon|F|_v\]
	for all $v \mid \infty$. 
	For this particular lift $\lambda F$ of $f$, 
	we have
	\[
	\max_{v \mid \infty} |\lambda F|_v^{-n}
	\prod_{w \mid \infty} |\lambda F|_w^{n_w} 
	< \max_{v \mid \infty} {(1 - \varepsilon|F|_v)^{-n}}
	\prod_{w \mid \infty} (1 + \varepsilon |F|_w)^{n_w}.
	\]
	Taking the infimum over all $\lambda \ne 0$ 
	and then letting $\varepsilon \to 0$ completes the proof.
\end{proof}

\bibliographystyle{plain}
\nocite{*}
\bibliography{ds_refs.bib}
 
\end{document}